\documentclass{amsart}
\usepackage{xcolor,enumerate}       % blackboard math symbols
 \definecolor{darkgreen}{rgb}{0,0.5,0}
\usepackage{blindtext}
\usepackage{mathtools}
\usepackage{tikz}
\usetikzlibrary{arrows.meta,calc,positioning,decorations.pathmorphing}

\usepackage[T1]{fontenc}    % use 8-bit T1 fonts
\usepackage{hyperref}       % hyperlinks
\usepackage{url}            % simple URL typesetting
\usepackage{amsmath,amssymb}
\usepackage{amsmath,mathrsfs,bm}
\usepackage{parskip}
\usepackage{paralist}
\usepackage{hyperref}
\usepackage{bigints}
\usepackage{graphicx}     % For including images
\usepackage{subcaption}   % For subfigures with individual captions\textbf{}
\usepackage{caption}
\usepackage{graphicx} % Required for inserting images
\usepackage{comment}
\usepackage{mathrsfs}

\usepackage{lipsum}
\usepackage{amsfonts}
\usepackage{epstopdf}
\ifpdf
  \DeclareGraphicsExtensions{.eps,.pdf,.png,.jpg}
\else
  \DeclareGraphicsExtensions{.eps}
\fi

\usepackage{algorithm}       % Floating environment
\usepackage{algpseudocode}   % Pseudocode formatting
\usepackage{bbm}

\newtheorem{theorem}{Theorem}[section]

\newtheorem{lemma}[theorem]{Lemma}

\newtheorem{proposition}[theorem]{Proposition}

\newtheorem{example}[theorem]{Example}

\newtheorem{assumption}{Assumption}

\numberwithin{equation}{section}

\newcommand{\R}{\mathbb{R}}
\newcommand{\map}[3]{#1:#2 \rightarrow #3}
\newcommand{\E}{\mathbb{E}}
\newcommand{\KL}{D_{\mathrm{KL}}}

\newcommand{\tr}{\mathrm{tr}}

\newcommand{\dSR}{d_{\mathrm{SR}}}

\begin{document}

\begin{abstract}
We study the least-energy way to reshape a probability distribution when motion is constrained to a horizontal subbundle, i.e., optimal transport and distribution steering in sub-Riemannian geometry. It is well-known that the deterministic sub-Riemannian optimal transport problem can be complicated by the non-smooth geometry. To obtain a smooth and numerically tractable formulation, we introduce small noise that aligns with the geometry and study the associated Schr\"odinger bridge problem. We discuss the case where the given marginals can have compact support or have sub-Gaussian tails. Under bracket-generating hypotheses, we obtain smooth, strictly positive transition densities and a forward--backward characterization of the optimal bridge, which acts as a convex selection principle among potentially many deterministic optima. This leads to a practical Sinkhorn-type algorithm with geometric constraints and we discuss it convergence in both regimes. We show that as the noise level vanishes, the unique Schr\"odinger bridge minimizers converge to a minimizer of the deterministic sub-Riemannian optimal transport problem. In the case of full support regime, our result illustrates that, the existence and characterization of the unique Schr\"odinger bridge, Sinkhorn convergence and the convex selection principle is possible only after quantifying the smoothness and positivity of the transition density.  We demonstrate with a benchmark Heisenberg group.
\end{abstract}

\title[From Schr\"odinger Bridge to Optimal Transport over Sub-Riemannian Manifolds]{From Schr\"odinger Bridge to Optimal Transport over Sub-Riemannian Manifolds}

%\thanks{The work of the author is supported by the Natural Sciences and Engineering Research Council of Canada.} 

\author[D. O. Adu, K. Elamvazhuthi, and B. Gharesifard ]{D. O. Adu, K. Elamvazhuthi, and B. Gharesifard}
%\address{Department of Mathematics and Statistics\\
%Queen's University\\
%Kingston, ON, Canada}
%\email{bahman.gharesifard@queensu.ca}
%\urladdr{https://gharesifard.github.io}

\maketitle

\section{Introduction}\label{sec:intro}
Motivated by underactuated vehicles~\cite{murray2017mathematical,lee2010geometric}, minimum effort control over Sub-Riemannian manifolds~\cite{tiwariJena2020optimal,tiwari2022optimal,halder2021sub} and the need to reshape an entire distribution rather than single trajectories~\cite{brockett2007optimal,agrachev2009optimal,chen2021optimal,adu2024schrodinger,hindawi2011mass}, we ask the following question: what is the minimum-effort feedback law that steers the ensemble of trajectories of a control-affine system from an initial to a prescribed terminal distribution? Usually, initial data comprises an initial and a terminal probability measure $\mu_0,\mu_f\in\mathcal{P}(\mathbb{R}^d)$ with densities $\rho_0,\rho_f$ on $\mathbb{R}^d$ and a matrix-valued map  $\map{g}{\mathbb{R}^d}{\mathbb{R}^{d\times m}}$. The goal is to find a control $u$, if it exists, that transports with minimum effort along the control-affine system
\begin{equation}\label{eq:intro_det_dyn}
  \dot{x}_t^u  =  g(x_t^u)u(t, x_t^u), \qquad  x_0^u \sim \mu_0,\quad\text{and}\qquad  x_{1}^u\sim\mu_f.
\end{equation}
Here, $x^u_t\in\mathbb{R}^{d}$ denotes the state process influenced by the time-varying feedback control $\map{u}{[0,1]\times\mathbb{R}^d}{\mathbb{R}^{m}}$. The notation $x_0^u \sim \mu_0$ implies that at $t=0$ the random variable $x_0^u$ is distributed according to $\mu_0$. We say that the system is underactuated whenever $m<d$. We measure effort via %an $\alpha$-power control cost, with $\alpha\in\{1,2\}$ in applications that we consider:
\begin{equation}\label{eq:intro_det_obj}
  \inf_{u\in\mathcal U} \E \left[\int_{0}^{1}\frac12 \|u(t, x_t^u)\|^2 dt\right], %\quad\text{s.t.}\quad ,
\end{equation}
%\margin{Make this first footnote part of the text. Not good to have a footnote on page one.}
where the expectation is with respect to the law of $x^u$ from $ x_0^u\sim\mu_0$ to $x_{1}^u\sim\mu_f$. Here a convenient admissible class is 
%\margin{B: make the way you put expectation behind the integral consistent throughout. I fixed this one}
$$\mathcal U:=\Big\{u:\ [0,1]\times\mathbb{R}^d\to\mathbb{R}^m:
\E \left[\int_0^{1}\|u(t, x_t^u)\|^2 dt\right]<\infty\Big\}.$$
%Throughout, we assume that the densities of $\rho_0,\rho_f$ are positive and compactly supported. The latter assumption can be replaced with having an exponential tail on $\R^d$. 
The structural constraint $m<d$ implies that the admissible velocities $\dot{x}_t$ in~\eqref{eq:intro_det_dyn} are constrained to lie in the horizontal subbundle $\dot{x}_t\in \mathcal D_{x_t}$, where
\begin{equation}\label{eq: distribution}
\mathcal D_{x_t}:=\mathrm{span}\{g_1(x_t),\dots,g_m(x_t)\}\subset T_{x_t}\R^d\equiv\R^d,    
\end{equation}
with $g_i(x_t)$ the $i$-th column of $g(x_t)$ in~\eqref{eq:intro_det_dyn}.
This horizontal subbundle induces a sub-Riemannian structure with a metric $g_x^{\text{SR}}(v,v) := v^\top G(x)^\dagger v$ for $v \in \mathcal{D}_{x}$, where $G^\dagger$ is the pseudoinverse of $G := g g^\top$. The associated sub-Riemannian distance is  
 \[
 \dSR(q_0, q_1) = \inf\left\{\int_0^{1} \sqrt{g^{\mathrm{SR}}_{x_t}(\dot{x}_t, \dot{x}_t)} dt : \dot{x}_t \in \mathcal{D}_{x_t},~  ~  x_{0}=q_0~~\text{and}~~x_{1}=q_1\right\}
 \]
 and using~\eqref{eq:intro_det_obj}, we have that 
\begin{equation}\label{eq: Riemannian metric}
%c_2(x_0, x_f) = 
\frac12\dSR^2(q_0, q_1)= \inf\left\{\int_0^{1} \frac12\|\dot{x}_t\|_{G^\dagger}^2 dt : \dot{x}_t \in \mathcal{D}_{x_t},~ ~  x_{0}=q_0~~\text{and}~~x_{1}=q_1\right\},
\end{equation}
where $\|v\|_{G^\dagger}^2 = v^\top G^\dagger v$. Therefore,~\eqref{eq:intro_det_dyn}-\eqref{eq:intro_det_obj} is a dynamic optimal ensemble control problem over a sub-Riemannian manifold $(\R^d,\mathcal{D},g^{\mathrm{SR}})$, which seeks a feedback control law $u(t,x)$ steering $\mu_0$ to $\mu_f$ and the dynamics of the ensemble are constrained to $\mathcal{D}$ in~\eqref{eq: distribution}. This problem has been studied in~\cite{elamvazhuthi2023dynamical,elamvazhuthi2024benamou}, where the authors established the existence of optimal feedback control using Young measures for general nonlinear control affine-systems~\cite{elamvazhuthi2023dynamical} and subsequently, as regular feedback control function under suitable absence-of-abnormality assumptions~\cite{elamvazhuthi2024benamou} leveraging results of \cite{agrachev2009optimal,figalli2010mass}. Problem~\eqref{eq:intro_det_dyn}-\eqref{eq:intro_det_obj} is equivalent to the static optimal transport (OT) problem %over a sub-Riemannian manifold
\begin{equation}\label{eq:OT-static}
 \inf_{\gamma\in\Pi(\mu_0,\mu_f)} \int_{\mathbb{R}^d\times\mathbb{R}^d} \frac12\dSR^2(q_0, q_1) , d\gamma(q_0, q_1),
\end{equation}
where 
\begin{equation}\label{eq: transport_plans}
\Pi(\mu_0,\mu_f):=\{\gamma\in\mathcal P(\R^d\times\R^d) : \pi_{0\#}\gamma =\mu_0\quad\text{and}\quad \pi_{f\#}\gamma =\mu_f\}   
\end{equation}
is the set of couplings of $(\mu_0,\mu_f)$ and $\pi_0(x_0,x_f)=x_0$ and $\pi_f(x_0,x_f)=x_f$, for all $x_0,x_f\in\R^d$, are the projection maps. This static problem ~\eqref{eq:OT-static}-\eqref{eq: transport_plans}, which seeks the minimal cost coupling between $\mu_0$ and $\mu_f$, has been studied in~\cite{agrachev2009optimal,figalli2010mass}. %  for general nonholonomic systems and by~\cite{figalli2010mass}  for sub-Riemannian manifolds.} 
As stated in the literature~\cite{agrachev2009optimal,figalli2010mass,elamvazhuthi2023dynamical,elamvazhuthi2024benamou,rifford2012sub}, the core difficulty of OT over a sub-Riemannian manifold stems from the non-smooth geometry induced by the constraints: the control-to-state map can have singularities, so the regularity and uniqueness of minimizers may fail; even when a Monge map exists, it can be discontinuous or non-unique. The novelty in this work is to establish a framework that mitigates these difficulties.

Motivated by the regularizing effects of entropic penalization in optimal transport \cite{cuturi2013sinkhorn}, our contribution is of a different nature: we introduce an \emph{entropic regularization} of sub-Riemannian optimal transport. Specifically, we perturb~\eqref{eq:intro_det_dyn} by a small diffusion acting \emph{only along the horizontal distribution}:
\begin{equation}\label{eq:intro_sto_dyn}
dX_t^u = (b_{\epsilon}(X_t^u) + g(X_t^u) u(t,  X_t^u))  dt+ \sqrt{\epsilon} g(X_t^u)  dW_t,\quad X_0^u\sim\mu_0\quad\text{and}\quad X_{1}^u\sim\mu_f,
\end{equation}
where $b_{\epsilon} :=\frac\epsilon2\sum_{i=1}^{m}\nabla_{g_i}g_i$ %, the derivative of the vector field $ g_i$ in the direction of $ g_i$, 
is the Ito correction drift ensuring that %the uncontrolled SDE 
\begin{equation}\label{uncontrolled_dynamic_constraint}
dX_t = b_{\epsilon}(X_t) dt + \sqrt\epsilon g(X_t) dW_t, \quad  X_0 \sim \mu_0,
\end{equation}
naturally stays in the horizontal subbundle $\mathcal D$ in~\eqref{eq: distribution}, and %{. See Figure~\ref{fig:diffusion_comparison_three} for a demonstration.}\margin{Not sure this is the place to refer to this?} 
$\{W_t\}_{t \ge 0}$ is the standard $\mathbb{R}^m$ Brownian motion,  
$\epsilon>0$ is the noise intensity %and $W_t$ is the standard $\R^m$-Brownian motion, with $\epsilon>0$, 
and consider the problem
\begin{equation}\label{eq:intro_sto_obj}
\inf_{u\in\mathcal U}\ \E\left[\int_0^{1}\frac12\|u(t,X_t^u)\|^2dt\right], %\quad\text{subject to},
\end{equation}
where
\[
\mathcal{U} := \left\{ \map{u}{[0, 1] \times \mathbb{R}^d}{\mathbb{R}^m} \middle| \text{$u$ is adapted and } \mathbb{E} \int_0^{1} \|u(t,  X_t)\|^2 dt < \infty \right\}.
\]
%$b_{\epsilon} :=\frac\epsilon2\sum_{i=1}^{m}\nabla_{g_i}g_i$  is the Ito correction drift} a
This stochastic control problem is equivalent to the Schr\"odinger bridge problem (SBP)~\cite{schrodinger1931umkehrung}: among all path measures on $\Omega:=C([0,1];\R^d)$ that are absolutely continuous with respect to a reference law of~\eqref{uncontrolled_dynamic_constraint} and match the endpoint marginals $(\mu_0,\mu_f)$, find the one that minimizes the relative entropy with respect the reference law. A difficulty immediately arises: the reference law of~\eqref{uncontrolled_dynamic_constraint}, if it exists, is degenerate and the setup~\eqref{eq:intro_sto_dyn}-\eqref{eq:intro_sto_obj} is ill-posed. The associated Kolmogorov operator is not uniformly elliptic, and parabolic regularity theory in~\cite{dai1991stochastic,Leonard2012} does not apply directly.  To define a well-defined   Schr\"odinger system, one requires a strictly positive transition density function. Unlike the linear stochastic degenerate case~\cite{adu2022optimal,adu2024schrodinger,chen2021optimal,chen2016optimal} where the presence of drift is helpful, for a non-linear case, establishing the existence of a positive smooth transition density is delicate. While uniformly ellipticity ensures smooth and everywhere positive transition density function, these desirable properties are decoupled in this present case. For smoothness, one typically imposes the H\"ormander condition which leads to hypoellipticity of the generator~\cite{Hormander1967}. Under H\"ormander, positivity is tied to submersion property of the endpoint map~\cite{leandre2005positivity,arous1991decroissance} (Bismut/submersion condition) rather than the Chow-Rashevskii controllability condition. The driftless system is the structural choice that ensures that Chow-Rashevskii controllability is equivalent to H\"ormander-Bismut/submersion condition~\cite{arous1991decroissance}. Aside the fact that our settings smoothens the geometric singularities present in the OT-SR settings, the SBP setup is strictly convex and hence acts as a  selection principle up a to subsequence among potentially many deterministic SR-optimal solutions. Moreover, the dynamic regularity formulation transforms the continuity equation associated to~\eqref{eq:intro_det_dyn} 
to a parabolic  associated to~\eqref{eq:intro_sto_dyn}. This adds a numerical novelty in the paradigm of OT-SR setup. Aside the fact that the hypoelliptic forward-backward characterization improves 
numerical stability, it yields a practical algorithmic pipeline that mirrors (and extends) the classical entropic optimal transport/Sinkhorn 
paradigm to the sub-Riemannian regime. We show that the Schr\"odinger system associated to our problem can be solved via a Sinkhorn/IPFP 
fixed-point iteration. Throughout, our result covers the case where $\mu_0$ and $\mu_f$ have compact support and sub-Gaussian tails. We show that while the H\"ormander condition is sufficient for the compactly supported marginals regime, for the sub-Gaussian tails an additional quantitative bound on the transition density to control the tails is useful. Using this quantitative bound, we characterize the Schr\"odinger bridge, establish convergence of the Sinkhorn iterations of the Schr\"odinger scaling functions and convergence to a solution of the deterministic optimal transport problem~\eqref{eq:intro_det_dyn}-\eqref{eq: transport_plans}. 
%Finally, the hypoelliptic degeneracy introduces a genuine discretization challenge absent in uniformly elliptic settings: standard Euclidean  finite-difference discretizations may fail to preserve positivity and mass. Our schemes therefore, discretize the generator and its adjoint in a manner that is consistent with the underlying sub-Riemannian geometry. (see Figure~\ref{fig:quantitative_measurements}). 

We state here that an analysis from the Schr\"odinger bridge problem to the optimal transport problem for the classical \emph{isotropic} setting has been studied by  L\'eonard in~\cite{Leonard2012}. Aside the fact that our work treats a genuinely \emph{anisotropic} diffusion, L\'eonard's zero- noise limit approach (through large-deviation of the reference measure) does not apply to our case, in that L\'eonard's zero- noise limit is established through $\Gamma$-convergence and requires the final marginals vary along the sequence $\mu_{f,k}\rightharpoonup\mu_f$. In our case, $\mu_0,\mu_f$ are fixed throughout and we study the zero-limit of a single associated Schr\"odinger bridge problem.  Consequently, the sequence-based $\Gamma$-convergence in~\cite{Leonard2012} is not applicable. In our setting, we establish value convergence together with compactness of the sequence of minimizers of the Schrodinger bridge problem. For compactly supported marginals, we establish compactness of the minimizing laws through the variational formula for relative entropy~\cite[Section~1.4]{DupuisEllis1997} and for this, exponential tightness of the reference law of~\eqref{uncontrolled_dynamic_constraint} is needed. The latter fits the mixture framework of~\cite{biggins2004large} except that the sufficient condition for the pinned bridge to be exponentially tight does not apply in our hypoelliptic case. For the sub-Gaussian regime,  we directly establish tightness of the Schrodinger bridge.

%{\color{blue} We state here that convergence of entropic Schr\"odinger bridges to the deterministic Monge--Kantorovich problem has been studied systematically by L\'eonard in~\cite{Leonard2012} under the assumption that the reference diffusion is uniformly elliptic.} By contrast, our work introduces an \emph{anisotropic} entropic regularization: the noise acts only along the horizontal distribution, yielding a degenerate reference process whose law is supported on a null set for the full Wiener measure in~\cite{Leonard2012}. This horizontal regularization preserves the sub-Riemannian geometry and leads, in the zero-noise limit, to the sub-Riemannian optimal transport problem. %, removes the ellipticity hypothesis and establish $\Gamma$-convergence to the zero-noise limit on the path space.
%\lipsum[2-3]

% The outline is not required, but we show an example here.
The paper is organized as follows: In Section~\ref{sec: Stochastic Geometric Control Problem}, we discuss the Schr\"odinger bridge problem over the sub-Riemannian manifold and characterize the unique solution. In Section~\ref{sec:Eulerian} we derive the first-order condition of optimality in the Eulerian framework. In Section~\ref{sec: Limiting Section}, we show that the stochastic solution established in Section~\ref{sec:Eulerian} converges to the corresponding Eulerian solution in Section~\ref{sec: Stochastic Geometric Control Problem}. Section~\ref{sec: Numerical Method} is dedicated to numerical method and example.

\section{Schr\"odinger Bridge Formulation}\label{sec: Stochastic Geometric Control Problem}

All the analysis in this work is under the umbrella of the following Assumptions. %~\ref{eq:density_regularities} and~\ref{ass:Controllability condition}. 
\begin{assumption}\label{eq:density_regularities}
The given probability measures $\mu_i$ have densities, 
$\mu_i(dx) = \rho_i(x)\,dx$, $i\in\{0,f\}$, where $\rho_i$ is continuous and nonnegative and either:
\begin{enumerate}
\item \label{ass: compact support} the supports %$\rho_0,\rho_f \in L^1(\mathbb{R}^d)\cap L^\infty(\mathbb{R}^d)$ are continuous and nonnegative and  
$\mathrm K_0:=\operatorname{supp}\rho_0$ and 
$\mathrm K_f:=\operatorname{supp}\rho_f$ are compact sets in $\mathbb{R}^d$. %; in particular $\rho_0,\rho_f$ have finite entropy  $\int \rho_i \log \rho_i\,dx < \infty$ and we idenify $0 \log 0=0$.
\item\label{ass:two_tailed_Gaussian} There exist $x_0\in\mathbb{R}^d$ such that for $q(x):=d_{\mathrm{SR}}(x_0,x)$, where $d_{\mathrm{SR}}$ is the sub-Riemannian distance  associated with $(\mathbb{R}^d,\mathcal{D},g^{\mathrm{SR}})$, we have
\[
c_i \, e^{-A_i q(x)^2} \leq \rho_i(x) \leq C_i \, e^{-a_i q(x)^2}, 
\qquad i \in \{0, f\}, \quad x \in \mathbb{R}^d,
\]
for constants $0 < a_i \leq A_i < \infty$ and $0 < c_i \leq C_i < \infty$.
\end{enumerate}   
    \end{assumption}
%The first condition eliminates endpoint tails. The second condition allows noncompact supports but gives Gaussian-type control of those tails.
To ensure the problem is well-posed, we make the following assumptions.    
\begin{assumption}\label{ass:wellposed}
The coefficients $g_{1},\ldots ,g_{m}\in C^{\infty}(\R^{d};\R^{d})$ have at most linear growth. Moreover, every derivative of positive order is globally
bounded. 
\end{assumption}
The following will be useful.
\begin{assumption}\label{ass:Controllability condition}
The vector fields ${g_1, \dots, g_m}$ satisfy the bracket-generating condition. Specifically, we assume that the Lie algebra generated by $\{g_1, \dots, g_m\}$ spans the tangent space $T_y\mathbb{R}^d$ at every point $y \in \mathbb{R}^d$:
\[
\operatorname{Lie}(g_1, \dots, g_m)(y) = \mathbb{R}^d \quad \text{ for all $y \in \mathbb{R}^d$.}
\]
\end{assumption}
Rather than attempting to solve~\eqref{eq:intro_det_dyn}-\eqref{eq:intro_det_obj} directly as in \cite{agrachev2009optimal,figalli2010mass,elamvazhuthi2024benamou}, we first attempt to solve the relaxed problem~\eqref{eq:intro_sto_dyn}-\eqref{eq:intro_sto_obj}.
\begin{comment}
: %stochastic optimal control problem
\begin{equation}\label{eq: objective}
\inf_{u \in \mathcal{U}} \E \left[ \int_0^{1} \frac12 \|u(t,  X_t^u)\|_2^2  dt  \right]
\end{equation}
subject to 
\begin{align}\label{controlled_dynamics}
&d X_t^u = (b_{\epsilon}(X_t^u) + g(X_t^u) u(t,  X_t^u))  dt+ \sqrt{\epsilon} g(X_t^u)  dW_t,\\
& X_0^u\sim\mu_0\quad\text{and}\quad  X_{1}^u\sim{\color{blue}\mu_f}.\nonumber
\end{align}
Later we show that the solution for~\eqref{eq:intro_det_dyn}-\eqref{eq:intro_det_obj} is obtained as a limiting process of the relaxed problem. Problem~\eqref{eq:intro_sto_dyn}-\eqref{eq:intro_sto_obj} 
\end{comment}
This is equivalent to a Schr\"odinger bridge problem over a sub-Riemannian manifold $(\R^d,\mathcal{D},g^{\mathrm{SR}})$. To make this document self-contained and also emphasize the novelty of our work in that paradigm, we state the derivation below.
\begin{figure}[t]
    \centering
    
    % Left plot: Horizontally Constrained
    \begin{subfigure}[b]{0.32\textwidth}
        \centering
        \includegraphics[width=\textwidth]{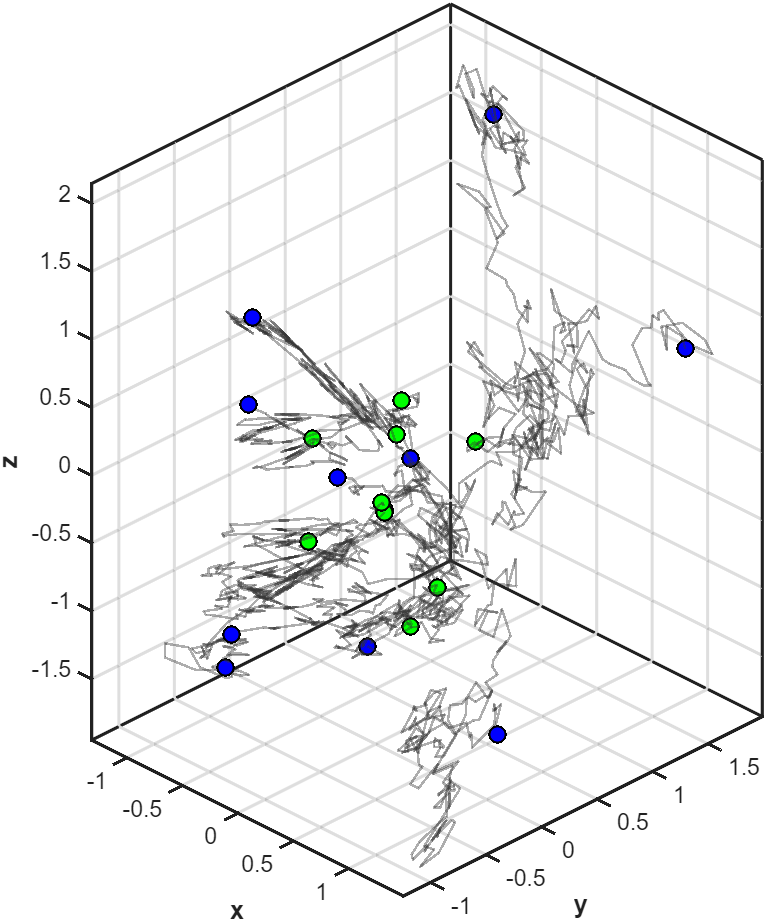}
        \caption{Horizontally constrained diffusion}
        \label{fig:constrained}
    \end{subfigure}
    \hfill
    % Middle plot: Purely Horizontal
    \begin{subfigure}[b]{0.32\textwidth}
        \centering
        \includegraphics[width=\textwidth]{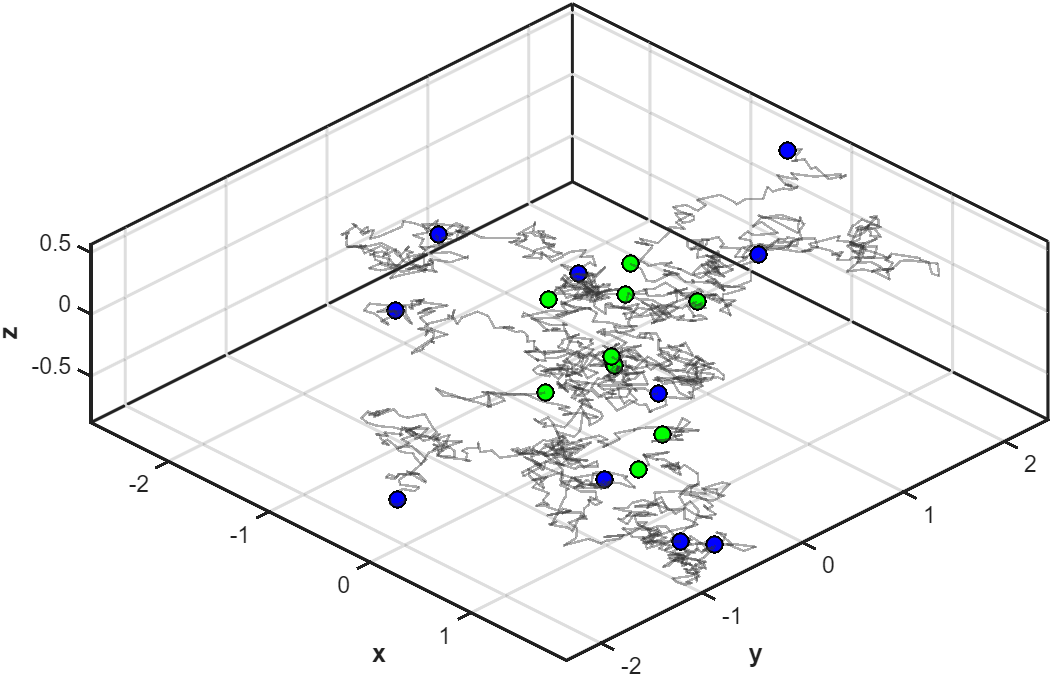}
        \caption{Purely horizontal diffusion}
        \label{fig:horizontal_only}
    \end{subfigure}
    \hfill
    % Right plot: Unconstrained Diffusion
    \begin{subfigure}[b]{0.32\textwidth}
        \centering
        \includegraphics[width=\textwidth]{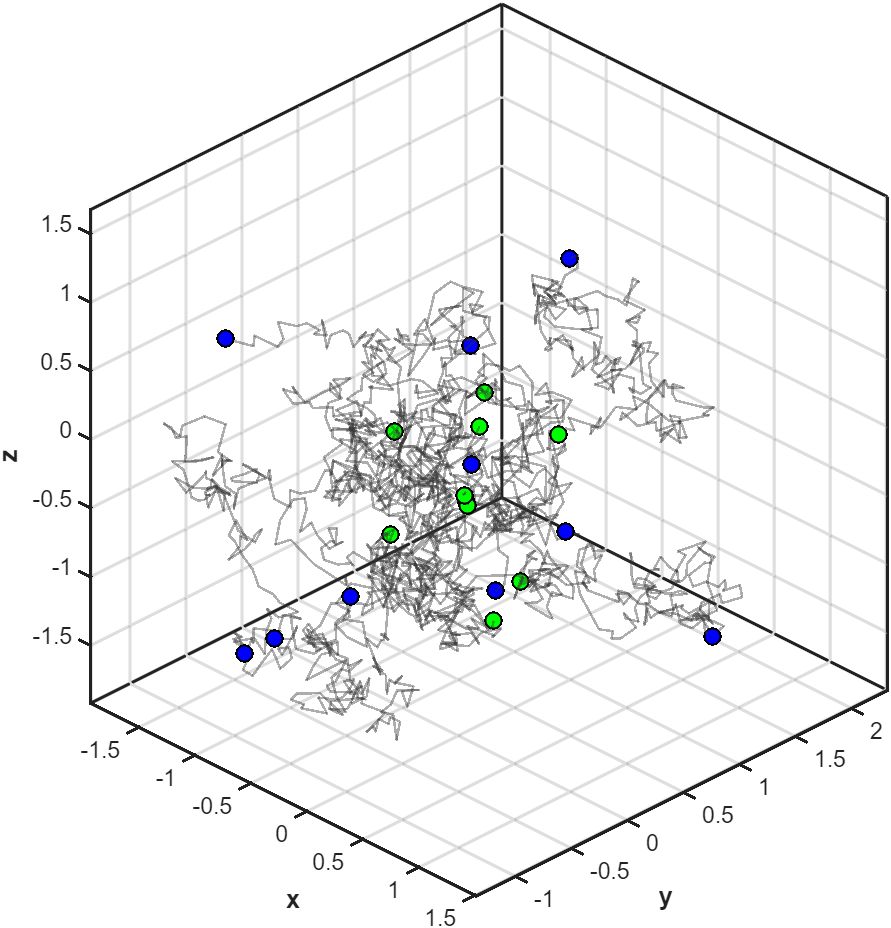}
        \caption{Isotropic diffusion}
        \label{fig:unconstrained}
    \end{subfigure}
    
    \caption{Plots~(a) and~(b) are anisotropic diffusion $dX_t=\sqrt{\epsilon} g(X_t) dW_t$ in $\mathbb{R}^3$ with $\epsilon=1$ as opposed to~(c) which is an isotropic diffusion $dX_t=\sqrt{\epsilon}dW_t$ with $\epsilon=1$ in $\mathbb{R}^3$. In all scenarios, the same $10$ initial samples from $X_0\sim\mu_0$ are used. The plot in~(a) is a horizontally constrained (hypoelliptic) diffusion of Heisenberg type, encoded by $g$ in Example~\ref{ex: unified_example}, where the noise acts only along the horizontal distribution and the vertical coordinate is generated indirectly through the noncommutative geometry. This creates a relatively flat pattern in the noise. The plot in~(b) corresponds to a process which is purely confined to horizontal planes (encoded by $g$).}
    \label{fig:diffusion_comparison_three}
\end{figure}
%\subsection{Schr\"odinger Bridge Formulation}
Let $\Omega:=C([0,1];\mathbb R^d)$ be the canonical space  with the uniform topology (and endowed with the Borel $\sigma$-algebra) of the uncontrolled process $\{ X_t\}_{0\leq t}$ in~\eqref{uncontrolled_dynamic_constraint}. Note that under Assumption~\ref{ass:wellposed}, we have that~\eqref{uncontrolled_dynamic_constraint} admits a unique strong solution on $[0,1]$ (see e.g.,~\cite{oksendal2003stochastic}). Let  $\mathrm{R}^{\epsilon,\mu_0}\in\mathcal P(\Omega)$, satisfying the pushforward \footnote{Given a measurable map $\map{T}{ (\Omega,\mathcal{F})}{(Y,\mathcal{G})}$ 
and a measure $\mu$ on $(\Omega,\mathcal{F})$, 
the \emph{pushforward measure} $T_\# \mu$ on $(Y,\mathcal{G})$ is defined by 
$T_\# \mu(B) := \mu(T^{-1}(B))$ for every $B \in \mathcal{G}$.} 
\begin{equation}\label{eq: law of unreflected process}
 ( \mathrm{ev}_0)_\# \mathrm{R}^{\epsilon,\mu_0} = \mu_0,   
\end{equation} 
 be the canonical law of~\eqref{uncontrolled_dynamic_constraint} on the sub-Riemannian manifold $(\R^d,\mathcal{D},g^{\mathrm{SR}})$. Here $\map{\mathrm{ev}_t}{\Omega}{\R^d}$ is the evaluation $\mathrm{ev}_t(X)=X_t$, for all $X\in\Omega$ at time $t\geq 0$. For a given admissible control $u \in \mathcal{U}$, let $\mathrm{P}^{u_\epsilon} \in \mathcal{P}(\Omega)$ be the corresponding probability distribution of the process $\{ X_t^u\}_{0\leq t\leq 1}\in \mathbb{R}^d$ in~\eqref{eq:intro_sto_dyn}. The Schr\"odinger bridge problem (SBP) is the infinite-dimensional optimization problem: 
\begin{equation}\label{sbp_precise}
\inf_{\mathrm{P}^{u_\epsilon} \in \mathcal C_{\mathrm{SB}}} \KL(\mathrm{P}^{u_\epsilon} \| \mathrm{R}^{\epsilon,\mu_0}):= \mathbb{E}_{\mathrm{P}^{u_\epsilon}} \Bigg[ \log \frac{d\mathrm{P}^{u_\epsilon}}{d\mathrm{R}^{\epsilon,\mu_0}} \Bigg] 
\end{equation}
where
\begin{equation}\label{eq: SB_constraint}
\mathcal C_{\mathrm{SB}}
:= \Big\{\mathrm{P}\ll\mathrm{R}^{\epsilon,\mu_0} :\ 
(\mathrm{ev}_0)_\#\mathrm{P}=\mu_0,\ (\mathrm{ev}_{1})_\#\mathrm{P}=\mu_f\Big\}.    
\end{equation}

By the Girsanov transformation~\cite[Chapter~8.6]{oksendal2003stochastic},  the candidate density process for a control $u \in \mathcal{U}$ is given by
\begin{equation}\label{rn_derivative_precise}
\frac{d\mathrm{P}^{u}}{d\mathrm{R}^{\epsilon,\mu_0}}\Big|_{\mathcal{F}_t} 
= \exp\Bigg( \int_0^t \frac{1}{\sqrt{\epsilon}} u(s,  X_s)^\top dW_s - \frac12 \int_0^t \frac{1}{\epsilon} \|u(s,  X_s)\|^2 ds \Bigg) 
\end{equation}
and is a local martingale. This implies that
\begin{equation}\label{eq: general_inequality}
\epsilon\KL(\mathrm{P}^{u} \| \mathrm{R}^{\epsilon,\mu_0}) \leq \mathbb{E}_{\mathrm{P}^{u}} \left[ \frac12 \int_0^{1} \|u(t,  X_t)\|^2 dt \right].     
\end{equation}
If there exists $\mathrm{P} \in \mathcal C_{\mathrm{SB}}$ such that $\epsilon\KL(\mathrm{P} \| \mathrm{R}^{\epsilon,\mu_0})<\infty$, then from~\cite[Theorem~2]{leonard2012girsanov}, we get that %there exists a minimal drift $\beta$ representing $\mathrm{P}$ that satisfies  $\mathbb{E}_{\mathrm{P}} \left[ \frac12 \int_0^{1} \|\beta(t)\|^2 dt\right]\leq \epsilon\KL(\mathrm{P} \| \mathrm{R}^{\epsilon,\mu_0})$  
%Therefore, we get that
\begin{equation}\label{kl_control_cost}
\inf_{\mathrm{P}^{u} \in \mathcal C_{\mathrm{SB}}}\epsilon\KL(\mathrm{P}^{u} \| \mathrm{R}^{\epsilon,\mu_0}) =  \inf_{u\in\mathcal U}\mathbb{E}_{\mathrm{P}^{u}} \left[ \frac12 \int_0^{1} \|u(t,  X_t)\|^2 dt \right].
\end{equation}
This shows that the Schr\"odinger bridge problem (SBP)~\eqref{sbp_precise} is equivalent to the stochastic control problem~\eqref{eq:intro_sto_dyn}-\eqref{eq:intro_sto_obj} on $(\R^d,\mathcal{D},g^{\mathrm{SR}})$ with given $\mu_0$ and $\mu_f$.%The equality in \eqref{kl_control_cost} is uniquely attained at the optimal bridge control $u^*$, whose density is a true martingale by construction. Consequently, minimizing $\KL$ in~\eqref{sbp_precise} is equivalent to finding the minimal-energy control~\eqref{eq:intro_sto_dyn}-\eqref{eq:intro_sto_obj} to steer $\mu_0$ to $\mu_f$.} % Furthermore, from~\eqref{kl_control_cost}, we have that finite KL-divergence is equivalent to finite expected drift energy. Minimizing $\KL$ in~\eqref{sbp_precise} is equivalent to finding the minimal-energy control~\eqref{eq:intro_sto_dyn}-\eqref{eq:intro_sto_obj} to steer $\mu_0$ to $\mu_f$.  

The SBP  over a sub-Riemannian manifold $(\R^d,\mathcal{D},g^{\mathrm{SR}})$ %\margin{we keep saying sub-R manifold, but we never say what we mean by this. We either should state what this means, or just say subject to the dynamics, which obviously is non-holonomic}
in~\eqref{sbp_precise} which serves as a stochastic counterpart to~\eqref{eq:intro_det_dyn}-\eqref{eq:intro_det_obj} deviates from the literature. More precisely, in the classical setting (e.g. see~\cite{dai1991stochastic}) the reference process $dX_t = b(X_t)dt + g(X_t)dW_t$ has a uniformly elliptic (ue) diffusion matrix $\xi^\top G(x) \xi \ge \lambda \|\xi\|^2$ for any $\xi,x\in\R^d$, where $G=gg^\top$ and for some constant $\lambda>0$. This is the situation in Figure~\ref{fig:diffusion_comparison_three}(c). Consequently its law $\mathrm{R}_{\mathrm{ue}}^{\epsilon}\in\mathcal P(\Omega)$ has full support and possesses a strictly positive smooth density. In contrast, $\mathrm{R}^{\epsilon,\mu_0}$ in~\eqref{eq: law of unreflected process} can be supported on the thin subset, a proper, null-measure set for $\mathcal{R}^{\epsilon}_{ue}$, reflecting the loss of ellipticity. This is the situation in Figure~\ref{fig:diffusion_comparison_three}(b) and the SBP formulated in~\eqref{sbp_precise} will be ill-posed.  We show that the driftless characterization in~\eqref{uncontrolled_dynamic_constraint} is right setup to guarantee that the reference path space measure $\mathrm{R}^{\epsilon,\mu_0}$ has full support and admit a strictly positive transition density, which is necessary to characterize the solution of the SBP problem in~\eqref{sbp_precise}. This is the situation in Figure~\ref{fig:diffusion_comparison_three}(a).

To establish the existence of a unique solution of~\eqref{sbp_precise}, the following result will be useful.
%\margin{I cannot figure this out, but the lemma environment is giving us an error. Why?}
\begin{lemma} The infinitesimal generator 
\begin{equation}\label{eq:Ito-generator}
\mathcal L_\epsilon 
= b_\epsilon\cdot\nabla_x +\frac{\epsilon}{2}\mathrm{tr}\big(G\nabla_x^2\big)
\end{equation}
associated to~\eqref{uncontrolled_dynamic_constraint} can be expressed as the sum of squares 
\begin{equation}\label{eq:sum-of-squares}
  \mathcal L_{\epsilon}= \tfrac12\sum_{i=1}^m V_i^2,
\end{equation}
where $V_i^2$ means that the operator
\begin{equation}\label{eq: trans_vec_field}
V_i:=\sqrt{\epsilon}  g_i\cdot\nabla_x,\quad i=1,\dots,m,
\end{equation}
applied twice.
\end{lemma}
\begin{proof}
For smooth test function $u=u(t,x)$, we have that $V_i u
= \sqrt{\epsilon}\sum_{k=1}^d g^{k}_i\partial_{ x_k}u$.
Applying $V_i$ a second time and using the product rule,
\[
\begin{aligned}
V_i^2 u
&= V_i\left( \sqrt{\epsilon}\sum_{\ell=1}^d g^{\ell}_i\partial_{ x_\ell} u \right)
= \sqrt{\epsilon}\sum_{k=1}^d g^{k}_i\partial_{ x_k}
\left( \sqrt{\epsilon}\sum_{\ell=1}^d g^{\ell}_i\partial_{ x_\ell} u \right)\\
&= \epsilon\sum_{k,\ell=1}^d g^{k}_i(\partial_{ x_k} g^{\ell}_i)\partial_{ x_\ell} u+\epsilon\sum_{k,\ell=1}^d g^{k}_ig^{\ell}_i\partial_{ x_k  x_\ell}^2 u.
\end{aligned}
\]
Summing over $i=1,\dots,m$ and multiplying by $\tfrac12$ gives
\[
\frac12\sum_{i=1}^m V_i^2 u
= \frac{\epsilon}{2}\sum_{k,\ell=1}^d
\Big(\sum_{i=1}^m g^{k}_i g^{\ell}_i\Big)\partial_{ x_k  x_\ell}^2 u
+\frac{\epsilon}{2}\sum_{\ell=1}^d
\Big(\sum_{i=1}^m\sum_{k=1}^d g^{k}_i\partial_{ x_k} g^{\ell}_i\Big)\partial_{ x_\ell} u.
\]
Recognizing the diffusion matrix $G:=g g^\top$ with entries
$G_{k\ell}=\sum_{i=1}^m g^{k}_i g^{\ell}_i$,
we can rewrite the previous identity compactly as
\begin{equation}\label{eq:sum-of-squares-expanded}
\frac12\sum_{i=1}^m V_i^2 u
=\frac{\epsilon}{2}\mathrm{tr}\big(G\nabla_x^2 u\big)
+\frac{\epsilon}{2}\Big(\sum_{i=1}^m \nabla_{g_i}g_i\Big)\cdot\nabla_x u=\mathcal Lu.
\end{equation} 
\end{proof}

\begin{proposition}\label{lem:kernel}
The reference measure $\mathrm{R}^{\epsilon,\mu_0}\in\mathcal P(\Omega)$ in~\eqref{eq: law of unreflected process} associated to the uncontrolled SDE in~\eqref{uncontrolled_dynamic_constraint} 
admits a smooth positive transition density function. 
\end{proposition}
\begin{proof}
Using the Lie-algebra isomorphism $\Phi:v\mapsto v\cdot\nabla$, we have that
\[
\mathrm{Lie}(g_1,\dots,g_m)(x)= \mathrm{Lie}(V_1,\dots,V_m)(x)
\qquad \forall x\in\mathbb{R}^d,
\]
where $V_i$ are defined in~\eqref{eq: trans_vec_field}. Therefore, under Assumption~\ref{ass:Controllability condition} and following from~\cite[Theorem~1.1]{Hormander1967} , we have that the operator $\partial_t-\mathcal L_{\epsilon}$, where $\mathcal L_{\epsilon}$ is in~\eqref{eq:sum-of-squares} is hypoelliptic. This implies that the PDE
\[
\partial_tu=\mathcal{L}_\epsilon^{\star}u,\qquad  \forall s<t\qquad\text{and}\quad u(s)=\delta_{x},
\]
where $\mathcal{L}_\epsilon^{\star}$ is the adjoint of $\mathcal{L}_\epsilon$ in~\eqref{eq:sum-of-squares}, $\delta_{x}$ is the Dirac distribution, admits a smooth transition density $p_{t-s,\epsilon}(x,y)$, for all $(x,y)\in\mathbb{R}^d\times\mathbb{R}^d$ and $t> s\geq 0$.   In our driftless case in~\eqref{eq:sum-of-squares}, controllable points can be reached using submersive controls~\cite[pp.~397]{arous1991decroissance} (controls that ensure the end-point map is a submersion). Hence under Assumption~\ref{ass:Controllability condition}, it follows that $p_{t,\epsilon}(x,y) > 0$ for all $x,y\in\R^d$ and $t>0$ (see~\cite[Main Theorem]{leandre2005positivity}). 
%or~\cite{arous1991decroissance}.
\end{proof} 
For later use, the following quantify the smoothness and positivity of $p_{t,\epsilon}$ in Proposition~\ref{lem:kernel}. This is useful only for the analysis under Assumption~\ref{eq:density_regularities}, item~\ref{ass:two_tailed_Gaussian}. 
\begin{assumption} \label{ass:global-lower-gaussian}
Under Assumption~\ref{eq:density_regularities}, item~\ref{ass:two_tailed_Gaussian}, we assume in addition that 
$p_{t,\epsilon}$ in Proposition~\ref{lem:kernel} satisfies: % the global two-sided Gaussian estimate
\begin{equation}\label{eq:global-gaussian}
\frac{a}{V(x,\sqrt{\epsilon t})}
\exp\left(
-A\frac{d_{\mathrm{SR}}(x,y)^2}{\epsilon t}
\right)
\le p_{t,\epsilon}(x,y)
\le
\frac{A}{V(x,\sqrt{\epsilon t})}
\exp\left(
-a\frac{d_{\mathrm{SR}}(x,y)^2}{\epsilon t}
\right).
\end{equation}
for all $x,y\in\mathbb{R}^d$ and $t>0$, where $\mathcal V(x,r):=|B(x,r)|$ denotes the Lebesgue volume of the sub-Riemannian ball $B(x,r):=\{y\in\R^d : d_{\mathrm{SR}}(x,y)<r\}$ of radius $r>0$ centred at $x$ and $0<a\le
A<\infty$. %For a fixed reference point $x_0$, we assume a polynomial volume growth $V(x,r)\leq c_V(1+d_{\mathrm{SR}}(x_0,x))^Nr^Q$, where $Q$ is the homogeneous dimension and $N\geq 0$ is the polynomial growth exponent {\color{blue}(later add the polynomial volume growth reference. This~\cite[Theorem~2.5]{bonfiglioli2002uniform} is specific.)}.
\end{assumption}
By Assumption~\ref{ass:Controllability condition}, we have that $\mathcal V(x,r)>0$, for all $x\in\R^d$ and $r>0$. 
The above quantitative estimate is an additional structural hypothesis on the reference diffusion, beyond Assumption~\ref{ass:wellposed}-\ref{ass:Controllability condition}. For the class of sub-Riemannian manifolds for which Assumption~\ref{eq:global-gaussian} is known to hold, see for example~\cite[Theorem~2.1]{dziubanski2019lie}.
%This estimate is available, in particular, when the underlying sub-Riemannian structure  $\mathcal{D}$ in~\eqref{eq: distribution} is a left-invariant  H\"ormander system on a Carnot group~\cite{bonfiglioli2002uniform}.
%That being said for densities satisfying item~\ref{ass:two_tailed_Gaussian} in Assumption~\ref{eq:density_regularities}, throughout,  we restrict the extension to Carnot-group sub-Laplacians for which Assumption~\ref{ass:wellposed} hold. 

Using the lower bound in~\eqref{eq:global-gaussian}, we get $1\geq
\int_{B_d(x,2r)}p_{r^2,1}(x,z)\,dz
\geq
ae^{-4A}\frac{V(x,2r)}{V(x,r)}$.
Consequently, $V(x,2r)\leq D V(x,r)$ for some $D\geq2$.
Setting $\kappa=\log_2D$, for $\lambda\geq1$, $ x,y\in\mathbb R^d$ and $r>0$, the standard consequences of
volume doubling (see~\cite{chen2021stability}) give
\begin{equation}\label{eq:bridge_volume_comparison}
V(x,\lambda r)\leq D\lambda^\kappa V(x,r),
\quad \text{and}\quad \frac{V(x,r)}{V(y,r)}
\leq
D\left(1+\frac{d_{\mathrm{SR}}(x,y)}r\right)^\kappa.
\end{equation}
Furthermore, one can show that, for a fixed $x_0\in\R^d$, the volume is bound 
\begin{equation}\label{eq: polynomial bound}
\frac{V(x_0,1)}{D^2}
\epsilon^\frac{\kappa}{2}(1+q(x))^{-\kappa}
\le V(x,\sqrt\epsilon)
\le
D V(x_0,1)(1+q(x))^\kappa, \quad\text{for all $x\in\R^d$.}   
\end{equation}
%we recall here that $q(x):=d_{\mathrm{SR}}(x_0,x)$. In this setting, global Gaussian heat-kernel estimates and polynomial volume growth provide the required entropy and tail-integrability bounds useful later.
We are now ready to characterize the Schr\"odinger bridge for~\eqref{sbp_precise}. % has been established, we proceed  to establish well-posedness of the SBP in~\eqref{eq:intro_sto_dyn}-\eqref{eq:intro_sto_obj} under degeneracy of the diffusion coefficient. 
%There exists a unique entropy-minimizing path measure $\mathrm P^{u^*}\in  \mathcal C_{\mathrm{SB}}$  . Consequently, problem~\eqref{eq:intro_sto_dyn}-\eqref{eq:intro_sto_obj} is proper: there exists an admissible $u$ that steers~\eqref{eq:intro_sto_dyn} from $ X_0^u\sim\mu_0$ to $ X_{1}^u\sim{\color{blue}\mu_f}$ and has finite energy in~\eqref{eq:intro_sto_obj}.
\begin{theorem}\label{thm:properness}
The unique Schr\"odinger bridge for~\eqref{sbp_precise} is 
\begin{equation}\label{eq: Schrodinger bridge}
d\mathrm{P}^{*}=d\mathrm{P}^{u_\epsilon^*}:=Zd\mathrm{R}^{\epsilon,\mu_0},\quad\text{with}\quad Z=\phi_{0,\epsilon}(X_0)\phi_{f,\epsilon}(X_1)    
\end{equation}
and
\begin{equation}\label{eq: finite_entropy=finite_energy}
\epsilon\KL(\mathrm{P}^{u_\epsilon^*} \| \mathrm{R}^{\epsilon,\mu_0}) = \mathbb{E}_{\mathrm{P}^{u_\epsilon^*}} \left[ \frac12 \int_0^{1} \|u^*(t,  X_t)\|^2 dt\right]<\infty,
\end{equation}
where $u^*=\epsilon\,g^\top\nabla\log\varphi_\epsilon$ minimizes~\eqref{eq:intro_sto_obj} and steers~\eqref{eq:intro_sto_dyn} from $X_0\sim\mu_0$ to $ X_{1}\sim\mu_f$. %is the optimal control for~\eqref{eq:intro_sto_dyn}-\eqref{eq:intro_sto_obj}. 
Here  $\varphi_\epsilon(t,x):=\mathbb{E}_{\mathrm{R}^{\epsilon,\mu_0}}\!\left[
\phi_{f,\epsilon}(X_{1})
\,\middle|\,
X_t=x
\right]=\int_{\R^d} \phi_{f,\epsilon}(y)p_{1-t,\epsilon}(x,y)\,dy$ and
$\map{\phi_{0,\epsilon},\phi_{f,\epsilon}}{\mathbb R^d}{[0,\infty)}$ are non-negative measurable functions satisfying 
the Schr\"odinger system
\begin{equation}\label{eq:Sch-multplication_system}
\phi_{0,\epsilon}(x)(\mathscr P_{0,1} \phi_{f,\epsilon})(x)=1,\qquad\text{for $\mu_0$-a.e $x$ and} \qquad (\mathscr P_{0,1}^{\star} \phi_{0,\epsilon}\rho_0)(y)\phi_{f,\epsilon}(y)=\rho_f(y),\qquad\text{for a.e $y$}
\end{equation}
where 
\begin{equation*}
(\mathscr P_{s,t}\phi_{f,\epsilon})(x):=\int \phi_{f,\epsilon}(y)p_{h,\epsilon}(x,y)dy\quad\text{and}\quad (\mathscr P_{s,t}^{\star}  \phi_{0,\epsilon}\rho_0)(y):=\int \phi_{0,\epsilon}(x)p_{h,\epsilon}(x,y)\rho_0(x)dx    
\end{equation*}
 with $p_{h,\epsilon}$ in Proposition~\ref{lem:kernel}, where $h=t-s$.

\end{theorem}
We state here that, for  $m=d$ and $g=I_d$, the reference process is
$dX_t=\sqrt{\epsilon}\,dW_t$, and the feedback reduces to
$u^*=\epsilon\nabla\log\varphi_\epsilon$, where $\partial_t\varphi_\epsilon
+\frac{\epsilon}{2}\Delta\varphi_\epsilon=0$ with $\varphi_\epsilon(1,\cdot)=\phi_{f,\epsilon}(\cdot)$. Thus the theorem recovers the Brownian Schr\"odinger
bridge formulation~\cite{Leonard2014Survey} associated to Figure~\ref{fig:unconstrained}. %When $G=gg^\top$ is degenerate, $b_\epsilon$ in~\eqref{uncontrolled_dynamic_constraint} is restricted to the horizontal distribution. This characterization is consistent with the finite-entropy Girsanov framework of~\cite{leonard2012girsanov}, which allows degenerate diffusion covariance matrices.}

\begin{proof}
%Given the law $\mathrm{R}^{\epsilon,\mu_0}$ of $\{ X_t\}_{0\leq t\leq 1}$  in~\eqref{uncontrolled_dynamic_constraint}, consider
First, under Assumption~\ref{eq:density_regularities} and given Proposition~\ref{lem:kernel}, there exists such a pair $(\phi_{0,\epsilon},\phi_{f,\epsilon})$ solving~\eqref{eq:Sch-multplication_system} and is unique up to scaling constants (see~\cite[Theorem~$1$]{Chen2016Hilbert}  or~\cite{Fortet1940}). We proceed to show that~\eqref{eq: Schrodinger bridge} is admissible. To this end, from~\eqref{eq:Sch-multplication_system}, since
\begin{equation*}
 \E_{\mathrm{P}^{*}}[\psi( x_0)]
=\E_{\mathrm{R}^{\epsilon,\mu_0}} \big[\psi( x_0)\phi_{0,\epsilon}( x_0)\phi_{f,\epsilon}( x_{1})\big]
 =\int_{\mathbb{R}^d} \psi(x)\phi_{0,\epsilon}(x)(\mathscr P_{0,1} \phi_{f,\epsilon})(x) \rho_0(x)dx =\int_{\mathbb{R}^d} \psi(x)\rho_0(x)dx   
\end{equation*}
and
\begin{equation*}
\E_{\mathrm{P}^{*}}[\psi( x_{1})]
=\E_{\mathrm{R}^{\epsilon,\mu_0}} \big[\psi( x_{1})\phi_{0,\epsilon}( x_0)\phi_{f,\epsilon}( x_{1})\big] =\int_{\mathbb{R}^d} \psi(y)(\mathscr P_{0,1}^{\star} \phi_{0,\epsilon}\rho_0)(y)\phi_{f,\epsilon}(y)dy
 =\int_{\mathbb{R}^d} \psi(y)\rho_f(y)dy    
\end{equation*}
hold for all bounded $\psi$, we have that $( \mathrm{ev}_0)_{\#}\mathrm{P}^{*}=\mu_0$ and $( \mathrm{ev}_{1})_{\#}\mathrm{P}^{*}=\mu_f$. We proceed to show that $\mathrm{P}^{*}$ in~\eqref{eq: Schrodinger bridge} is the unique solution for~\eqref{sbp_precise} and has finite entropy with respect to $\mathrm{R}^{\epsilon,\mu_0}$. To this end,   let $\mathrm{P} \in\mathcal C_{\mathrm{SB}}$. Then, since $\epsilon\KL(\mathrm{P} \| \mathrm{R}^{\epsilon,\mu_0})=\epsilon\KL(\mathrm{P} \| \mathrm{P}^{*})+\epsilon\mathbb{E}_{\mathrm{P}} \Bigg[ \log \frac{d\mathrm{P}^{*}}{d\mathrm{R}^{\epsilon,\mu_0}} \Bigg]$
and from~\eqref{eq: Schrodinger bridge}, we have $\mathbb{E}_{\mathrm{P}} \Bigg[ \log \frac{d\mathrm{P}^{*}}{d\mathrm{R}^{\epsilon,\mu_0}} \Bigg]= \mathbb{E}_{\mathrm{P}^{*}} \Bigg[\log \frac{d\mathrm{P}^{*}}{d\mathrm{R}^{\epsilon,\mu_0}} \Bigg]$
%$\mathbb{E}_{\mathrm{P}} \Bigg[ \log \frac{d\mathrm{P}^{*}}{d\mathrm{R}^{\epsilon,\mu_0}} \Bigg]=\mathbb{E}_{\mathrm{P}^{u_\epsilon}} \Bigg[ \log (\phi_{0,\epsilon}( X_0)\phi_{f,\epsilon}( X_{1})) \Bigg]=\mathbb{E}_{\mathrm{P}^{*}} \Bigg[ \log (\phi_{0,\epsilon}( X_0)\phi_{f,\epsilon}( X_{1})) \Bigg]= \mathbb{E}_{\mathrm{P}^{*}} \Bigg[\log \frac{d\mathrm{P}^{*}}{d\mathrm{R}^{\epsilon,\mu_0}} \Bigg]$, 
we get that $\epsilon\KL(\mathrm{P} \| \mathrm{R}^{\epsilon,\mu_0})=\epsilon\KL(\mathrm{P} \|\mathrm{P}^{*})+ \epsilon\KL(\mathrm{P}^{*}\| \mathrm{R}^{\epsilon,\mu_0})$.
Since $\KL(\mathrm{P} \| \mathrm{P}^{*})\geq 0$, we get that \begin{equation}\label{eq: optimal_bridge} 
\epsilon\KL(\mathrm{P} \| \mathrm{R}^{\epsilon,\mu_0})\geq \epsilon\KL(\mathrm{P}^{*}\| \mathrm{R}^{\epsilon,\mu_0}),\qquad\text{for all $\mathrm{P}\in\mathcal C_{\mathrm{SB}}$},
\end{equation}
with equality if and only if %$\KL(\mathrm{P} \| \mathrm{P}^{*})=0$ that is 
$\mathrm{P} =\mathrm{P}^{*}$. %Since $\mathrm{P}\mapsto \KL(\mathrm{P} \| \mathrm{R}^{\epsilon,\mu_0})$ is strictly convex, This proves uniqueness. 
To prove finite entropy~\eqref{eq: finite_entropy=finite_energy}, from~\eqref{eq: Schrodinger bridge},  %consider the reference endpoint coupling,
%where $K(x,dy)=p_{1,\epsilon}(x,y)dy$ is the final time transition. 
since %By disintegration of the relative entropy, we have that 
\begin{equation}\label{eq:KL-disintegration}
\KL(\mathrm{P}^{*} \| \mathrm{R}^{\epsilon,\mu_0})
\;=\;
\KL(\gamma^*_\epsilon \| r_\epsilon)
%\;+\;\int_{\R^d\times\R^d}\KL\big(\mathrm{P}^{\star,x,y} \| \R^{\epsilon,x,y}\big) \pi^*(dx,dy),
\end{equation}
where $\gamma^*_\epsilon\in \Pi(\mu_0,\mu_f)$ in~\eqref{eq: transport_plans} is the optimal endpoint coupling, whenever it exists and %$\mathrm{R}^{\epsilon,x,y}( \cdot )=\mathrm{R}^{\epsilon,\mu_0}( \cdot \mid \mathrm{ev}_0(\cdot)=x,\mathrm{ev}_{1}(\cdot)=y)$  and
\begin{equation}\label{eq: reference_endpoint_density}
r_\epsilon(dx,dy)=(\mathrm{ev}_0,\mathrm{ev}_{1})_{\#}\mathrm{R}^{\epsilon,\mu_0}=p_{1,\epsilon}(x,y)\rho_0(x)dxdy,%K(x,dy)\mu_0(dx),
\end{equation}
%and from~\eqref{eq: Schrodinger bridge}, we only tilt at end time $t=0$ and $t=1$ and hence $\mathrm{P}^{\star,x,y} = \mathrm{R}^{\epsilon,x,y}$, for $r_\epsilon$-a.e. $(x,y)$, 
we get that  $\KL(\mathrm{P}^{*} \| \mathrm{R}^{\epsilon,\mu_0})=\min_{\gamma\in\Pi(\mu_0,\mu_f)} \KL(\gamma \| r_\epsilon)$.
Assumption~\ref{eq:density_regularities} ensures that
\begin{multline}\label{eq: finite_energy_bound}
\epsilon\KL(\mathrm{P}^{u_\epsilon^*} \| \mathrm{R}^{\epsilon,\mu_0})\leq \epsilon\KL(\mu_0\otimes\mu_f \| r_\epsilon)\\=\epsilon\int_{\R^d}\rho_f(y)\log(\rho_f(y))dy-\int_{\R^d\times\R^d}\epsilon\log(p_{1,\epsilon}(x,y))\rho_0(x)\rho_f(y)dxdy<\infty.
\end{multline}
Indeed, under Assumption~\ref{eq:density_regularities}, item~\ref{ass: compact support}, we have $p_{1,\epsilon}(x,y) \ge c_{K_0,K_f,\epsilon} > 0
\quad\text{on } K_0 \times K_f$ the above is finite. 
Under Assumption~\ref{eq:density_regularities}, item~\ref{ass:two_tailed_Gaussian}, the first term becomes $\int\rho_f|\log\rho_f|
\le
|\log c_f|+|\log C_f|+A_f\int q(y)^2\,\rho_f(y)\,dy
<
\infty$. Using the lower bound in~\eqref{eq:global-gaussian}, we get 
\begin{multline*}%\label{eq: log_moment_kernel}
-\int_{\R^d\times\R^d}\epsilon\log(p_{1,\epsilon}(x,y))\rho_0(x)\rho_f(y)
\le
-\epsilon\log c
+\epsilon\int_{\R^d}\log V(x,\sqrt{\epsilon})\,\rho_0(x)\,dx
\\+C\int_{\R^d\times\R^d} d_{\mathrm{SR}}(x,y)^2\,\rho_0(x)\rho_f(y)\,dx\,dy<\infty.
\end{multline*}
For the second term, from~\eqref{eq: polynomial bound}, we get that $\epsilon|\log V(x,\sqrt\epsilon)|
\le C\bigl(1+q(x)^2\bigr)$ and integrability follows from $\int_{\R^d}C\bigl(1+q(x)^2\bigr)\rho_0(x)\,dx=C\left(1+\int_{\R^d}q(x)^2\rho_0(x)\,dx\right)<\infty$. Finiteness of the third term follows from $d_{\mathrm{SR}}(x,y)^2\le 2\,q(x)^2+2q(y)^2$. This proves that $\KL(\mathrm{P}^{*} \| \mathrm{R}^{\epsilon,\mu_0})<\infty$.

We proceed to construct the process for which $\mathrm{P}^{*}$ is realized. For this we apply~\cite[Theorem~1 and Theorem~2]{leonard2012girsanov}. To this end, let $\{ X_t\}_{0\leq t\leq 1}$ satisfy the SDE in~\eqref{uncontrolled_dynamic_constraint} and $\mathcal F_{t}=\sigma( X_s, 0\leq s\leq t)$ be the \emph{canonical filtration}, for all $t\in[0,1]$.
Consider
\begin{equation}\label{eq: martingale_conditional_expectation}
Z_t:=\mathbb{E}_{\mathrm{R}^{\epsilon,\mu_0}}\left[Z|\mathcal F_t\right]=\phi_{0,\epsilon}(X_0)\varphi_\epsilon(t, X_t),~\text{ where }~\varphi_\epsilon(t, x):=\mathbb{E}_{\mathrm{R}^{\epsilon,\mu_0}}\!\left[
\phi_{f,\epsilon}(X_{1})
\,\middle|\,
X_t=x
\right]. %,\qquad .    
\end{equation}
%Then $\varphi(1, \cdot)=\phi_{f,\epsilon}(\cdot)$. Let
By the first equation in~\eqref{eq:Sch-multplication_system}, we get $Z_0=\phi_{0,\epsilon}(X_0)\varphi_\epsilon(0, X_0)=1$, for $R^{\epsilon,\mu_0}$-a.s.
Moreover, for $0\le s\le t\le1$, using the Markov property of~\eqref{uncontrolled_dynamic_constraint} and the
semigroup property we get that $\E_{\mathrm{R}^{\epsilon,\mu_0}}[Z_t\mid\mathcal F_s]=Z_s$.
\begin{comment}
    \begin{align}
\E_{\mathrm{R}^{\epsilon,\mu_0}}[Z_t\mid\mathcal F_s]
&=
\phi_{0,\epsilon}(X_0)\E_{\mathrm{R}^{\epsilon,\mu_0}}[\varphi(t, X_t)\mid\mathcal F_s] \notag\\
&=
\phi_{0,\epsilon}(X_0)(\mathscr P_{s,t}\varphi(t, \cdot))(X_s) \notag\\
&=
\phi_{0,\epsilon}(X_0)(\mathscr P_{s,1}\phi_{f,\epsilon})(X_s) \notag\\
&=
Z_s. \label{eq:likelihood-martingale}
\end{align}
\end{comment}
Thus, $Z_t=\phi_{0,\epsilon}(X_0)\varphi_\epsilon(t, X_t)=\frac{\varphi(t, X_t)}{\varphi(0, X_0)}$, where $X_t$ satisfies~\eqref{uncontrolled_dynamic_constraint}, is a nonnegative mean-one martingale. Using Feynman–Kac formula for $\varphi_\epsilon$ in~\eqref{eq: martingale_conditional_expectation}, on the interior, the Ito formula yields 
\begin{equation}\label{eq: admissible control}
\frac{dZ_t}{Z_t}=\frac{1}{\sqrt{\epsilon}}u^*(t,X_t)dW_t,\qquad\text{where}\qquad u^*(t,x)=\epsilon  g(x)^\top\nabla\log\varphi(t,x).
\end{equation}
Since $Z_0=1$,  the solution of the martingale density $Z_t$ in~\eqref{eq: martingale_conditional_expectation} is
\[
Z_t
=
\exp\left(
\frac1{\sqrt{\epsilon}}
\int_0^t u^*(s,X_s)\cdot dW_s
-
\frac1{2\epsilon}
\int_0^t \|u^*(s,X_s)\|^2\,ds
\right),\quad t\in[0,1).
\]
From~\cite[Theorem~2]{leonard2012girsanov}, we have  $\mathbb{E}_{\mathrm{P}^{*}} \left[ \frac12 \int_0^{1} \|u^*(t,  X_t)\|^2 dt\right]\leq \epsilon\KL(\mathrm{P}^{*} \| \mathrm{R}^{\epsilon,\mu_0})<\infty$. Indeed, the reference covariance is $A(dt)=\epsilon Gdt$,
and \eqref{eq: admissible control} identifies the
Girsanov drift through $\beta=\nabla\log\varphi_\epsilon$, so that $A(dt)\beta=gu^*dt$ and $\beta^\top A(dt)\beta=\frac{|u^*|^2}{\epsilon}dt$, where $u^*=\epsilon\,g^\top\nabla\log\varphi_\epsilon$. 
 Therefore, from~\cite[Theorem~1]{leonard2012girsanov}, we get that 
$\mathrm{P}^{*}=\mathrm{P}^{u_\epsilon^*}$ in~\eqref{eq: Schrodinger bridge} is the law of
\begin{equation}\label{eq:unref-tiltedSDE}
d X_t=(b_{\epsilon}(X_t)+g(X_t)u^*(t,X_t)) dt+\sqrt{\epsilon}  g(X_t)  dW_t^{\mathrm{P}^{u_\epsilon^*}}
\end{equation}
where $\{W_t^{\mathrm{P}^{u_\epsilon^*}}\}$ is a Brownian motion. Under Assumption~\ref{ass:wellposed}, since SDE~\eqref{uncontrolled_dynamic_constraint} admits a unique solution, from~\cite[Theorem~2]{leonard2012girsanov}, we get~\eqref{eq: finite_entropy=finite_energy}. To show that $u^*=\epsilon\,g^\top\nabla\log\varphi_\epsilon$ minimizes~\eqref{eq:intro_sto_obj}, let $u$ be any control that generates~\eqref{eq:intro_sto_dyn} such that $\mathrm{P}^{*}$ is indistinguishable between~\eqref{eq:intro_sto_dyn} and~\eqref{eq:unref-tiltedSDE}. Then $u$ satisfies  $gu=gu^*$ and hence $u=u^*+k$, where $k\in \mathrm{ker}(g)$. Since $u^*\in\mathrm{ker}(g)^{\perp}$, we get that $\|u\|^2\geq \|u^*\|^2$ in~\eqref{eq:intro_sto_obj}. For general admissible $u$, from~\eqref{eq: general_inequality},~\eqref{eq: finite_entropy=finite_energy} and~\eqref{eq: optimal_bridge}, we get that $\mathbb{E}_{\mathrm{P}^{u_\epsilon^*}} \left[ \frac12 \int_0^{1} \|u^*(t,  X_t)\|^2 dt\right]=\epsilon\KL(\mathrm{P}^{u_\epsilon^*} \| \mathrm{R}^{\epsilon,\mu_0})\leq\epsilon\KL(\mathrm{P}^{u} \| \mathrm{R}^{\epsilon,\mu_0}) \leq \mathbb{E}_{\mathrm{P}^{u}} \left[ \frac12 \int_0^{1} \|u(t,  X_t)\|^2 dt \right]$. This completes the proof.
\end{proof}
Although the algebraic skeleton of the proof is similar to~\cite{Leonard2014Survey}, the degeneracy of $G = gg^\top$ in~\eqref{eq:unref-tiltedSDE} forces the control~\eqref{eq: admissible control} to lie in $\operatorname{span}\{g_1,\dots,g_m\}$. Moreover, the finite-energy bound in~\eqref{eq: finite_energy_bound} relies on properties of the hypoelliptic heat-kernel, which in turn depend on the sub-Riemannian distance on~\eqref{eq: distribution}.

\section{Convex Eulerian Formulation}\label{sec:Eulerian}
To facilitate the proof of convergence from the horizontal Schr\"odinger bridge problem in~\eqref{sbp_precise} to the sub-Riemannian optimal-transport problem~\eqref{eq:intro_det_dyn}-\eqref{eq:intro_det_obj}, we work in the Eulerian formulation.

Fix $\epsilon>0$ and let $Q:=(0,1)\times\R^d$. The time marginals of the reference law $\mathrm{R}^{\epsilon,\mu_0}\in\mathcal P(\Omega)$ in~\eqref{uncontrolled_dynamic_constraint} admit Lebesgue densities. Consequently, for every finite-energy controlled law $\mathrm{P}^{u_\epsilon}\in \mathcal C_{\mathrm{SB}}$ in~\eqref{eq: SB_constraint},
 its time marginals $\mu_{\epsilon,t}:=(\operatorname{ev}_t)_\#\mathrm{P}^{u_\epsilon}$ also admit densities $\rho_\epsilon(t,\cdot)$. Define the associated spacetime and momentum measures by
\begin{equation}\label{eq: space_time_measures}
\boldsymbol\mu_\epsilon(dt,dx)
:=dt\,\mu_{\epsilon,t}(dx),\qquad
\mathbf m_\epsilon:=u_\epsilon\boldsymbol\mu_\epsilon.    
\end{equation}
Then $\mu_{\epsilon,\cdot}\in C([0,1];\mathcal P(\R^d))$ is narrowly continuous satisfying the endpoint trace $\mu_{\epsilon,t=0}=\mu_0$, $\mu_{\epsilon,t=1}=\mu_f$ and
\[
\mathcal J(\boldsymbol\mu_\epsilon,\mathbf m_\epsilon):=
\int_0^1\int_{\R^d} \frac12\Big\|\frac{d\mathbf m_{\epsilon}(t,x)}{d\boldsymbol\mu_{\epsilon}(t,x)}\Big\|^2 d\boldsymbol\mu_{\epsilon}
=\frac12\E_{\mathrm{P}^{u_\epsilon}}\int_0^1
\|u_\epsilon(t,X_t)\|^2\,dt.%=\epsilon\KL(\mathrm{P}^{u_\epsilon} \| \mathrm{R}^{\epsilon,\mu_0}).
\]
Itô’s formula for~\eqref{eq:intro_sto_dyn} yields the Fokker–Planck constraint, motivating the following Eulerian minimization problem:
%For a fixed $\epsilon>0$, since the time marginals of $\mathrm{R}^{\epsilon,\mu_0}\in\mathcal P(\Omega)$ in~\eqref{uncontrolled_dynamic_constraint} admit Lebesgue densities and the probability distribution $\mathrm{P}^{u_\epsilon}$  of the process $\{ X_t^u\}_{0\leq t\leq 1}\in \mathbb{R}^d$ in~\eqref{eq:intro_sto_dyn} is absolutely continuous with respect to $\mathrm{R}^{\epsilon,\mu_0}$, we have that $\mathrm{P}^{u_{\epsilon}}$  admits a density function which we denote as $\rho_{\epsilon}$. Let $\boldsymbol\mu_{\epsilon}\in \mathcal P(Q)$, where $Q:=(0,1)\times\R^d$, be the associated spacetime measure of~\eqref{eq:intro_sto_dyn} having density $\rho_{\epsilon}$ with $t\mapsto \boldsymbol\mu_{\epsilon}|_t=\rho_{\epsilon}(t,\cdot)dx\in C([0,1];\mathcal P(\R^d))$, where the continuity is narrow continuity, satisfying the marginal constraint $\rho_{\epsilon}(0,\cdot)=\rho_0$ and $\rho_{\epsilon}(1,\cdot)=\rho_f$.
%Let $\mathbf m_{\epsilon} \in\mathcal{M}(Q;\R^m)$  Radon measure with an $\R^m$-valued density $m_{\epsilon}$, that is $\mathbf m_{\epsilon}\ll \boldsymbol\mu_{\epsilon}$ and $u_{\epsilon}=\frac{d\mathbf m_{\epsilon}}{d\boldsymbol\mu_{\epsilon}}\in L^2(\boldsymbol\mu_{\epsilon};\R^m)$. 
%Then we can express~\eqref{eq:intro_sto_dyn}-\eqref{eq:intro_sto_obj} as the optimization problem, \in \mathcal{P}(\overline Q)\times \mathcal{M}(\overline Q) 
\begin{equation}\label{measure_theoretic_objective}
\inf_{(\boldsymbol\mu_{\epsilon}, \mathbf m_{\epsilon})\in\mathcal C} 
\mathcal{J}(\boldsymbol\mu_{\epsilon}, \mathbf m_{\epsilon}),%(dt,dx)
\end{equation}
%\margin{Very important one: I have been confused throughout by when we use $ \boldsymbol\mu_{\epsilon}$ vs $\boldsymbol\mu_\epsilon$ to refer to the solution of the above. You have $ \mathcal{J}$ in here, and so perhaps it made sense to use $\boldsymbol\mu_\epsilon$. Now, if you look at say (4.3), there you use $\boldsymbol\mu_\epsilon$. This is one of those changes that if done, needs to be done very carefully, and consistently. }
where
\begin{equation*}
\mathcal C: =\bigl\{ (\boldsymbol\mu, \mathbf m): \boldsymbol\mu(dt,dx)
:=dt\,\mu_{t}(dx)\,~ \mu_{\cdot}\in C([0,1];\mathcal P(\R^d))~\text{and}~ \mathbf m \in\mathcal{M}(Q;\R^m),~\,\mathbf m\ll \boldsymbol\mu,\ 
\frac{d\mathbf m}{d\boldsymbol\mu}\in L^2(\boldsymbol\mu;\R^m) \bigr\}    
\end{equation*}
subject to
\begin{subequations}
\begin{align}
   \partial_t \boldsymbol\mu_{\epsilon}
+ \nabla_x \cdot \left(b_{\epsilon}\boldsymbol\mu_{\epsilon}+g\mathbf m_{\epsilon}\right)
- \frac{\epsilon}{2} \sum_{i,j=1}^d 
\partial_{ x_i  x_j}^2 (G_{ij} \boldsymbol\mu_{\epsilon} )
&= 0\quad\text{in } \bigl(C_c^\infty(Q;\R)\bigr)^{*},\label{eq: fb_measure_constraint}\\
    \boldsymbol\mu_{\epsilon}|_{t=0} = \mu_0, 
    \qquad 
    \boldsymbol\mu_{\epsilon}|_{t=1} &= \mu_f,\quad\text{in } \mathcal P(\R^d),\label{eq:endpoint_measure_constraints}
\end{align}
\end{subequations}
where $\bigl(C_c^\infty(Q;\R)\bigr)^{*}$ is the dual of  the space of compactly supported smooth functions $C_c^\infty(Q;\R)$.  
%{\color{blue}Since the constraint is linear} in $(\boldsymbol{\mu}_\epsilon,\mathbf{m}_\epsilon)$ and $(\boldsymbol{\mu}_\epsilon,\mathbf{m}_\epsilon)\mapsto \mathcal{J}(\boldsymbol{\mu}_\epsilon,\mathbf{m}_\epsilon)$ is {\color{blue}jointly convex} (see~\cite{BenamouBrenier2000}), using Theorem~\eqref{thm:properness} and Theorem~\eqref{prop: Strong duality and attainment} we can establish the existence and uniqueness result of Problem~\eqref{measure_theoretic_objective}.
We have the following result.
\begin{theorem}\label{thm: existence and uniqueness}
Problem~\eqref{measure_theoretic_objective} subject to~\eqref{eq: fb_measure_constraint}-\eqref{eq:endpoint_measure_constraints} admits a unique minimizer $(\boldsymbol\mu_{\epsilon}^*, \mathbf m_{\epsilon}^*)$.  
\end{theorem} 
Before proving the theorem, we establish the following consequence of the weighted superposition principle and the entropy--energy
estimate.
\begin{lemma}\label{lem: main_link_to_density_problem}
Let $(\boldsymbol\mu_\epsilon,\mathbf m_\epsilon)$ in~\eqref{eq: space_time_measures} with $\mathcal J(\boldsymbol\mu_\epsilon,\mathbf m_\epsilon)<\infty$, where $(\mu_{\epsilon,\cdot},u_\epsilon)$ satisfy % and suppose that
\begin{subequations}
\begin{align}
 \partial_t\mu_{\epsilon,t}
+\nabla\cdot\bigl((b_\epsilon+gu_\epsilon)\mu_t\bigr)
-\frac{\epsilon}{2}\sum_{i,j}\partial_{ij}(G_{ij}\mu_{\epsilon,t})&=0\quad\text{in } \bigl(C_c^\infty(\R^d;\R)\bigr)^{*}\label{eq: weak_sol_space_measure}\\
    \mu_{\epsilon,0} = \mu_0,% \quad\text{in } \mathcal P(\R^d).
    \qquad 
    \mu_{\epsilon,1}& = \mu_f,\quad\text{in } \mathcal P(\R^d),\label{eq: trace_condition}
\end{align}
\end{subequations}
Then there exists a law  $\mathrm{P}^{u_\epsilon}\in\mathcal P(\Omega)$ such that $(\operatorname{ev}_t)_\#\mathrm{P}^{u_\epsilon}=\mu_{\epsilon,t}$, for every $t\in[0,1]$ $\epsilon\KL(\mathrm{P}^{u_\epsilon} \| \mathrm{R}^{\epsilon,\mu_0})\le
\frac12\int_Q|u_\epsilon(t,x)|^2\,d\boldsymbol\mu
=
\mathcal J(\boldsymbol\mu_\epsilon,\mathbf m_\epsilon)$. Moreover, $\mathrm{P}^{u_\epsilon}$ is realized as the law of the weak solution to~\eqref{eq:intro_sto_dyn}.
\end{lemma}
\begin{proof}
We verify the hypothesis of~\cite[Theorem~1.1]{bogachev2021ambrosio}. To this end, let $a(x):=\frac{\epsilon}{2}G(x)$ and $B(t,x):=b_\epsilon(x)+g(x)u(t,x)$. Then, under Assumption~\ref{ass:wellposed}, we get that $\|a(x)\|
\le C_\epsilon(1+|x|)^2$ and $|\langle B(t,x),x\rangle| \le C_\epsilon(1+|x|)^2)(1
+|u(t,x)|)$ and hence %$\frac{\|a(x)\|+|\langle B(t,x),x\rangle|}{(1+|x|)^2}\le C_\epsilon(1+|u_\epsilon(t,x)|)$. Therefore, 
$\int_Q
\frac{\|a(t,x)\|+|\langle B(t,x),x\rangle|}
{(1+|x|)^2}\,d\boldsymbol\mu(t,x)\leq C_\epsilon\int_Q (1+|u_\epsilon(t,x)|)\,d\boldsymbol\mu(t,x)\leq  C_\epsilon(1+\sqrt{2\mathcal J(\boldsymbol\mu_\epsilon,\mathbf m_\epsilon})$. Finally, since $g$,
$b_\epsilon$, and $G$ are locally bounded and
$u_\epsilon\in L^1(\boldsymbol\mu_\epsilon)$, we have that $a$ and $B$ are also locally integrable
with respect to $\boldsymbol\mu_\epsilon$. Therefore, from ~\cite[Theorem~1.1]{bogachev2021ambrosio}, there exists a law  $\mathrm{P}^{u_\epsilon}\in\mathcal P(\Omega)$ such that its time-marginal matches $(\operatorname{ev}_t)_\#\mathrm{P}^{u_\epsilon}=\mu_{\epsilon,t}$, for every $t\in[0,1]$ in~\eqref{eq: weak_sol_space_measure}-\eqref{eq: trace_condition}. Furthermore, $\mathrm{P}^{u_\epsilon}$ is realized as the law of the weak solution to~\eqref{eq:intro_sto_dyn}. 

To establish the entropy bound, we adapt the localization argument in~\cite[Proposition~1]{lehec2013representation} to our degenerate case. Let $U_t=u_\epsilon(t,X_t)$ and define the stopping time  $\tau_n
:=
\inf\left\{
t\in[0,1]:
\int_0^t|U_s|^2\,ds\ge n
\right\}\wedge1$ and the truncated control $U_t^n:=U_t\,\mathbf1_{\{t<\tau_n\}}$. Let $\mathbb P$ denote the underlying probability measure carrying the weak solution $(X,W)$, so $\mathrm{P}^{u_\epsilon}=\operatorname{Law}_{\mathbb P}(X)$, where $X$ is in~\eqref{eq:intro_sto_dyn}. Then, under Assumption~\ref{ass:wellposed}, we have that 
$dX_t^n = \big(b_\epsilon(X_t^n) + g(X_t^n)U_t^n\big)\,dt + \sqrt{\epsilon}\,g(X_t^n)\,dW_t$, with $X_0^n=X_0\sim\mu_0$,
is well-defined and $X_t^n=X_t$, for all $t<\tau_n\leq 1$, in distribution. Let  $\mathrm P_n=\operatorname{Law}_{\mathbb P}(X^n)$. Since $\exp\left(\int_0^{\tau_n} |U_s|^2 \, ds\right) \le e^n < \infty$, we have that the exponential process $Z_n
=
\exp\left(
-\frac1{\sqrt\epsilon}\int_0^1U_t^n\cdot dW_t
-\frac1{2\epsilon}\int_0^1|U_t^n|^2dt
\right)$ defines a probability measure $d\mathbb Q_n=Z_n\,d\mathbb P$. Hence by the Girsanov Theorem and data processing inequality, $\epsilon\KL(\mathrm P_n \| \mathrm{R}^{\epsilon,\mu_0})\le
\frac{1}{2} \mathbb{E} \int_0^{\tau_n} |U_t|^2 dt$. Since $ \int_0^1 |U_t|^2dt <\infty $ almost surely, we have that $X^n=X$
on $[0,1]$ for all sufficiently large $n$, almost surely, and hence 
Thus $\mathrm P_n\rightharpoonup\mathrm P^{u_\epsilon}$. Hence
lower semicontinuity of relative entropy yields $\epsilon\operatorname{KL}
\bigl(\mathrm P^{u_\epsilon}\|
\mathrm R^{\epsilon,\mu_0}\bigr)
\le
\frac12\E\int_0^1|U_t|^2\,dt
=
\mathcal J(\boldsymbol\mu_\epsilon,\mathbf m_\epsilon)$. 
\end{proof}
%State and proof that optimizers of the above problem admits densities.

\begin{comment}  

Then,
\begin{equation}\label{eq: time_varying density}
 \rho_{\epsilon}(t,\cdot) \in L^1\big(\mathbb{R}^d\big), \quad \rho_{\epsilon} \geq 0,\quad \int_{\R^d} \rho_{\epsilon}(t,x)  dx = 1 \quad \text{for a.e. } t
\end{equation}
and
\end{comment}
From the above result, we have that $\mathcal J(\boldsymbol\mu_\epsilon,\mathbf m_\epsilon)<\infty$ implies that $\mathrm P^{u_\epsilon}\ll\mathrm R^{\epsilon,\mu_0}$ and hence $
\mu_{\epsilon,t}\ll dx$. Thus, $\boldsymbol\mu_\epsilon=\rho_\epsilon\,dt\,dx$, $\mathbf m_\epsilon=m_\epsilon\,dt\,dx$ and $m_\epsilon=\rho_\epsilon u_\epsilon$. Thus $\mathcal J(\boldsymbol\mu_\epsilon,\mathbf m_\epsilon)=\mathcal J(\rho_\epsilon, m_\epsilon)$. Therefore, the proof of Theorem~\ref{thm: existence and uniqueness} reduces to establishing existence and uniqueness of the density problem:
%Motivated by~\cite{BenamouBrenier2000}, rather than working with the time-marginal measures in~\eqref{measure_theoretic_objective}-\eqref{eq:endpoint_measure_constraints}, we first establish existence and uniqueness  directly for the following primal problem corresponding to~\eqref{eq:intro_sto_dyn}-\eqref{eq:intro_sto_obj}: 
\begin{equation}\label{eq: convexified_stochastic}
\inf_{(\rho_{\epsilon}, m_{\epsilon})} \mathcal{J}(\rho_{\epsilon},m_{\epsilon}):=
\int_0^{1} \int_{\mathbb R^d}
\frac{\| m_{\epsilon}(t,x) \|_2^2}{2   \rho_{\epsilon}(t,x)}   dx   dt.
%\quad \text{subject to~\eqref{eq: fb_constraint}-\eqref{eq:endpoint_constraints}.}
\end{equation}
subject to
\begin{subequations}
\begin{align}
   \partial_t \rho_{\epsilon}
+ \nabla_x \cdot \left(b_{\epsilon}\rho_{\epsilon}+gm_{\epsilon}\right)
- \frac{\epsilon}{2} \sum_{i,j=1}^d 
\partial_{x_i x_j}^2 \Big(G_{ij} \rho_{\epsilon} \Big)
&= 0\quad\text{in }\quad \bigl(C_c^\infty(Q;\R)\bigr)^{*}\label{eq: fb_constraint}\\
    \rho_{\epsilon}(0, \cdot) = \rho_0, 
    \qquad 
    \rho_{\epsilon}(1, \cdot) &= \rho_f,\quad\quad\text{in\quad $\mathcal P(\R^d)$.} 
    \label{eq:endpoint_constraints}
\end{align}
\end{subequations}
\begin{comment}
. To this end,
 the cost functional in~\eqref{eq:intro_sto_obj} becomes %the expectation (with respect to the law $\rho(t,x)$ of the state) of the control cost for the original optimal control problem~\eqref{eq:intro_sto_dyn}:
\begin{equation}\label{eq: function_objective}
\mathbb{E} \Big[ \frac12 \int_0^{1} \tfrac{1}{2} \| u_{\epsilon}(t,  X_t^u) \|^2 \Big] =  \frac12 \int_0^{1}\int_{\mathbb{R}^d} \frac12 \| u_{\epsilon}(t,x) \|^2    \rho_{\epsilon}(t,x)    dx.
\end{equation}
for the momentum
\begin{equation}\label{eq: momentum_pointwise}
m_{\epsilon}(t,x) = u_{\epsilon}(t,x)\rho_{\epsilon}(t,x)  \quad \text{for some } u_{\epsilon} \in L^2(\rho_{\epsilon}; \mathbb{R}^m).
\end{equation}
and from~\eqref{eq: function_objective}, we have
\[
\frac12\int_0^{1}\int_{\R^d}\|u_{\epsilon}(t,x)\|^2 \rho_{\epsilon}(t,x) dxdt
=
\int_0^1\int_{\R^d} \frac{\|m_{\epsilon}(t,x)\|^2}{2\rho_{\epsilon}(t,x)} dx dt,
\]
and the function pair $(\rho_{\epsilon},m_{\epsilon})$ solves the controlled Fokker--Planck equation
reads:    
\end{comment}
To provide a proof of existence and uniqueness for the above problem, the following setup will be useful. 
For $\mathrm{P}^{u_{\epsilon}}\in\mathcal C_{\mathrm{SB}}$, define $J(u_{\epsilon})
:=
\frac12 \E_{\mathrm{P}^{u_{\epsilon}}}\int_0^1 \|u_{\epsilon}(t,X^u_t)\|^2\,dt$,
where $\{ X_t^u\}_{0\leq t\leq 1}\in \mathbb{R}^d$ is in~\eqref{eq:intro_sto_dyn}. Then, from~\eqref{measure_theoretic_objective} and~\eqref{eq: convexified_stochastic}, we get that $J(u_{\epsilon})=\mathcal{J}(\rho_{\epsilon},m_{\epsilon})$ and from~\eqref{kl_control_cost} and~\eqref{eq: optimal_bridge} we get $J^*_\epsilon:=J(u^*_{\epsilon})\leq J(u_{\epsilon})$, for all $\mathrm{P}^{u_\epsilon}\in\mathcal C_{\mathrm{SB}}$, where $u^*$ is in~\eqref{eq: admissible control}. Along this unique bridge $\mathrm{P}^{*}=\mathrm{P}^{u^*_{\epsilon}}$ in~\eqref{eq: Schrodinger bridge}, we define the set:
\begin{comment}
    \begin{multline*}
\Lambda(\mathrm{P}^{u^*_{\epsilon}})
:=\bigl\{
\lambda\in C^{1,2}(Q;\R):\;
H_\lambda\le0 \text{ on }Q,\text{ there exists}~\lambda_0\in L^1(\mu_0),\\
\lambda_1\in L^1(\mu_f)
\text{ such that}~\lim_{t\downarrow0}\lambda(t,x)=\lambda_0(x),
~\mu_0\text{-a.e.}\text{ and}~\\
\lim_{t\uparrow1}\lambda(t,x)=\lambda_1(x)
~\mu_f\text{-a.e. and }~\E_{\mathrm{P}^{u^*_{\epsilon}}}\int_0^1
\left(
|H_\lambda(t,X_t)|
+
|v_\lambda(t,X_t)|^2
\right)dt<\infty
\bigr\}
\end{multline*}
or 
\begin{multline*}
\Lambda(\mathrm{P}^{u^*_{\epsilon}})
:=\Bigl\{
\lambda\in C^{1,2}(Q;\R):\;
H_\lambda\le0 \text{ on }Q,\text{ there exists}~\lambda_0\in L^1(\mu_0),\\
\lambda_1\in L^1(\mu_f)
\text{ such that}~\lambda(t,X_t)\to\lambda_0(X_0)
\text{ in }L^1(\mathrm{P}^{u^*_{\epsilon}})
\text{ as }t\downarrow0\text{ and}~\\
\lambda(t,X_t)\to\lambda_1(X_1)
\text{ in }L^1(\mathrm{P}^{u^*_{\epsilon}})
\text{ as }t\uparrow1~
\text{and }~\E_{\mathrm{P}^{u^*_{\epsilon}}}\int_0^1
\|A_\epsilon\lambda(t,X_t)\|dt<\infty
\Bigr\},
\end{multline*}
\end{comment}
\begin{multline*}
\Lambda(\mathrm{P}^{*})
:=\Bigl\{
\lambda\in C^{1,2}(Q;\R):\;
H_\lambda\le0 \text{ on }Q,\text{ there exists}~\lambda_0\in L^1(\mu_0),
\lambda_1\in L^1(\mu_f)
\text{ such that}~\\\lim_{t\downarrow0}
\E_{\mathrm{P}^{*}}
|\lambda(t,X_t)-\lambda_0(X_0)|=0\text{ and}~
\lim_{t\uparrow1}
\E_{\mathrm{P}^{*}}
|\lambda(t,X_t)-\lambda_1(X_1)|=0\text{ and }~\E_{\mathrm{P}^{*}}\int_0^1
\|A_\epsilon\lambda(t,X_t)\|dt<\infty
\Bigr\},
\end{multline*}
where $H_\lambda
:=
A_\epsilon\lambda+\frac12\|v_\lambda\|^2$ with $v_\lambda:=g^\top\nabla\lambda$ and  $A_\epsilon\lambda
:=
\partial_t\lambda+b_\epsilon\cdot\nabla\lambda
+\frac{\epsilon}{2}\operatorname{tr}(G\nabla^2\lambda)$.
%For the full-support two-sided-tail case in Assumption~\ref{eq:density_regularities}, the endpoint traces are finite almost everywhere. In a compact-support case, one can allow extended-real traces such as $-\infty$.

We proceed to characterize the optimal solution for~\eqref{eq: convexified_stochastic}-\eqref{eq:endpoint_constraints} as follows.
\begin{theorem}\label{prop: Strong duality and attainment}
Define the dual functional
$\map{\mathcal D}{\Lambda}{\R}$ as
\begin{equation}\label{eq:dual-func-thm}
\mathcal D(\lambda):=\int_{\R^d}\lambda_1(x)\mu_f(dx)
-\int_{\R^d}\lambda_0(x)\mu_0(dx).
\end{equation}
 Then strong duality holds: 
\[ 
\inf_{\substack{(\rho_{\epsilon}, m_{\epsilon}) \\ \text{\eqref{eq: fb_constraint}-\eqref{eq:endpoint_constraints}}}} 
\mathcal{J}(\rho_{\epsilon}, m_{\epsilon}) = \sup_{\lambda_{\epsilon}\in\Lambda(\mathrm{P}^{*})}\mathcal D(\lambda_{\epsilon}).
\]
The supremum is attained at
\begin{equation}\label{eq:optimal-dual}
\lambda_\epsilon^*:=\epsilon\log\varphi_\epsilon\in C^\infty(Q;\R).
\end{equation}
Moreover, the primal infimum is uniquely attained at
\begin{align}\label{eq:optimal-primal}
&\rho_\epsilon^*:=\varphi_\epsilon\hat\varphi_\epsilon\in C^\infty(Q;\R),\nonumber\\\
&m_\epsilon^*:=\rho_\epsilon^* g^\top\nabla\lambda_\epsilon^* = \epsilon \hat\varphi_\epsilon g^\top\nabla\varphi_\epsilon\in C^\infty(Q;\R^m),
\end{align}
where $\varphi_\epsilon,\hat\varphi_\epsilon
\in C^\infty(Q;\R)$ be the  mild solutions of: %the Schr\"odinger system
\begin{subequations}\label{eq:Schr-system-thm}
\begin{align}
\partial_t\varphi_\epsilon + \mathcal L_\epsilon \varphi_\epsilon =0,\qquad \partial_t\hat\varphi_\epsilon - \mathcal{L}_\epsilon^{\star} \hat\varphi_\epsilon &=0\label{eq:Schr-fbwd}\\
\varphi_\epsilon(0,\cdot)\hat\varphi_\epsilon(0,\cdot)=\rho_0,\qquad
\varphi_\epsilon(1,\cdot)\hat\varphi_\epsilon(1,\cdot)&=\rho_f,\label{eq:Schr-bc}
\end{align}
\end{subequations}
with continuous endpoint traces, where $\mathcal L_\epsilon$ is in~\eqref{eq:Ito-generator}.
\end{theorem}
Before we provide the proof, we state here that the mild solution of~\eqref{eq:Schr-system-thm} implies that the solutions for~\eqref{eq:Schr-fbwd}-\eqref{eq:Schr-bc} is characterized as semigroups
\[
\varphi_\epsilon(t,\cdot) = e^{(1-t)\mathcal{L}_\epsilon}\varphi_{1}, \qquad 
\hat{\varphi}_\epsilon(t,\cdot) = e^{t\mathcal{L}_\epsilon^{\star}}\hat{\varphi}_0,
\]
with $\hat{\varphi}_0 = \rho_0 / \varphi(0,\cdot)$ and $\varphi_{1}=\phi_{f,\epsilon}$ in~\eqref{eq:Sch-multplication_system},  where the forward and backward semigroups $P_t = e^{t\mathcal{L}_\epsilon^{\star}}$ and $
Q_t = e^{t\mathcal{L}_\epsilon}$ are defined through the actions
\begin{equation}\label{eq: action_semigroup}
(P_t \varphi_0)(y) = \int_{\R^d} p_{t,\epsilon}(x,y) \varphi_0(x)   dx\quad\text{and}\quad (Q_t \varphi_f)(x) = \int_{\R^d} p_{t,\epsilon}(x,z) \varphi_f(z)   dz    
\end{equation}
respectively, where $p_{t,\epsilon}$ in Proposition~\ref{lem:kernel}.
\begin{proof}
Let $0<\delta<\frac{1}{2}$ and $\lambda\in \Lambda(\mathrm{P}^{u^*_{\epsilon}})$, by Ito formula, we get 
\begin{equation*}%\label{eq: Ito_lambda}
\lambda(1-\delta,X_{1-\delta})-\lambda(\delta,X_\delta)=\int_\delta^{1-\delta}F_\lambda(t,X_t)dt+\sqrt{\epsilon}\int_\delta^{1-\delta}v_\lambda(t,X_t)dW^{\mathrm{P}^{u^*_{\epsilon}}},
\end{equation*}
where $F_\lambda
:=
A_\epsilon\lambda+u_\epsilon^*\cdot v_\lambda$. For the right side, since $|v_\lambda|^2\leq 2|A_\epsilon\lambda|$ we get that $\E_{\mathrm{P}^{u^*_{\epsilon}}}
\int_\delta^{1-\delta}v_\lambda(t,X_t)dW^{\mathrm{P}^{u^*_{\epsilon}}}=0$. 
\begin{comment}
and  $\E_{\mathrm{P}^{u^*_{\epsilon}}}
\int_0^1\|(u_\epsilon^*\cdot v_\lambda)(t,X_t)\|\,dt\leq \sqrt{2\epsilon}\left(\KL(\mathrm{P}^{u_\epsilon^*} \| \mathrm{R}^{\epsilon,\mu_0})
\right)^{1/2}
\left(
\E_{\mathrm{P}^{u^*_{\epsilon}}}\int_0^1\|A_\epsilon\lambda(t,X_t)\|dt
\right)^{1/2}$

\begin{multline*}
\E_{\mathrm{P}^{u^*_{\epsilon}}}
\int_0^1\|F_\lambda(t,X_t)\|\,dt\leq \left(
\E_{\mathrm{P}^{u^*_{\epsilon}}}\int_0^1\|A_\epsilon\lambda(t,X_t)\|dt
\right)^{1/2}\\\left(\left(
\E_{\mathrm{P}^{u^*_{\epsilon}}}\int_0^1\|A_\epsilon\lambda(t,X_t)\|dt
\right)^{1/2}+ \sqrt{2\epsilon}\left(\KL(\mathrm{P}^{u_\epsilon^*} \| \mathrm{R}^{\epsilon,\mu_0})
\right)^{1/2}\right),    
\end{multline*}    \E_{\mathrm{P}^{u^*_{\epsilon}}}\int_\delta^{1-\delta}F_\lambda(t,X_t)dt=
\end{comment}
Thus, 
\begin{equation}\label{eq: limit_of_Ito}
\lim_{\delta\downarrow0}\E_{\mathrm{P}^{u^*_{\epsilon}}}(\lambda(1-\delta,X_{1-\delta})-\lambda(\delta,X_\delta))= \lim_{\delta\downarrow0}\E_{\mathrm{P}^{u^*_{\epsilon}}}\int_0^1\mathbbm{1}_{[\delta,1-\delta]}F_\lambda(t,X_t)dt.    
\end{equation}
Since $|F_\lambda|\leq 2|A_\epsilon\lambda|+\frac{1}{2}\|u_\epsilon^*\|^2$, from~\eqref{eq: finite_energy_bound}, we get  $\E_{\mathrm{P}^{u^*_{\epsilon}}}
\int_0^1|F_\lambda(t,X_t)|\,dt<\infty$. Therefore, by dominating convergence, the right side of~\eqref{eq: limit_of_Ito}  yields
\[
\lim_{\delta\downarrow0} \E_{\mathrm{P}^{u^*_{\epsilon}}}\int_0^1\mathbbm{1}_{[\delta,1-\delta]}F_\lambda(t,X_t)dt=\E_{\mathrm{P}^{u^*_{\epsilon}}}\int_0^{1}F_\lambda(t,X_t)dt.
\]
For the left side of~\eqref{eq: limit_of_Ito},  the $L^1(P^*)$
endpoint conditions in $\Lambda(\mathrm{P}^{u^*_{\epsilon}})$ implies that
\begin{comment}
since $\E_{\mathrm{P}^{u^*_{\epsilon}}}(\lambda(1-\delta,X_{1-\delta})-\lambda(\delta,X_\delta))=\E_{\mathrm{P}^{u^*_{\epsilon}}}(\lambda(1-\delta,X_{1-\delta}))-\E_{\mathrm{P}^{u^*_{\epsilon}}}(\lambda(\delta,X_\delta))$ and for $\E_{\mathrm{P}^{u^*_{\epsilon}}}(\lambda(\delta,X_\delta))$, we have that
\begin{multline*}
\E_{\mathrm{P}^{u^*_{\epsilon}}}|\lambda(s,X_s)-\lambda(s',X_s')|\leq \E_{\mathrm{P}^{u^*_{\epsilon}}}
\int_{s\wedge s'}^{s\vee s'}
|F_\lambda(r,X_r)|\,dr\\+\left(
2\epsilon\,
\E_{\mathrm{P}^{u^*_{\epsilon}}}
\int_{s\wedge s'}^{s\vee s'}
|A_\epsilon\lambda(r,X_r)|^2\,dr
\right)^{1/2}
\end{multline*}
holds for small $0<s'<s$, we get that $\{\lambda(s,X_s)\}$ is Cauchy in $L^1(\mathrm{P}^{u^*_{\epsilon}})$. Therefore, there exists $\lambda(0,X_0)\in L^1(\mathrm{P}^{u^*_{\epsilon}})$ such that  $\lim_{\delta\downarrow0}\E_{\mathrm{P}^{u^*_{\epsilon}}}(\lambda(\delta,X_\delta))=\E_{\mathrm{P}^{u^*_{\epsilon}}}(\lambda(0,X_0))$. Similar arguments yield $\lim_{\delta\downarrow0}\E_{\mathrm{P}^{u^*_{\epsilon}}}(\lambda(1-\delta,X_{1-\delta}))=\E_{\mathrm{P}^{u^*_{\epsilon}}}(\lambda(1,X_1))$.    
\end{comment}
\[
\lim_{\delta\downarrow0}\E_{\mathrm{P}^{u^*_{\epsilon}}}(\lambda(1-\delta,X_{1-\delta})-\lambda(\delta,X_\delta))=D(\lambda)=\E_{\mathrm{P}^{u^*_{\epsilon}}}\int_0^{1}F_\lambda(t,X_t)dt.
\]
Furthermore, since $F_\lambda=
H_\lambda
-\frac12|v_\lambda|^2
+
u_\epsilon^*\cdot v_\lambda=
H_\lambda
-\frac12|v_\lambda-u_\epsilon^*|^2
+
\frac12|u_\epsilon^*|^2$, we get that 
\begin{equation}\label{eq:weak-fp-duality}
J_\epsilon^*-D(\lambda)
=
\E_{\mathrm{P}^{u^*_{\epsilon}}}\int_0^1(\frac12|v_\lambda(t,X_t)-u_\epsilon^*(t,X_t)|^2-
H_\lambda(t,X_t))dt\geq 0
\end{equation}
holds. Let $(\rho_\epsilon,m_\epsilon)$ be an arbitrary feasible density pair such that $J(\rho_\epsilon,m_\epsilon)<\infty$. Then, $\mathcal J(\boldsymbol\mu_\epsilon,\mathbf m_\epsilon)=\mathcal J(\rho_\epsilon, m_\epsilon)$, where $\boldsymbol\mu_\epsilon=\rho_\epsilon\,dt\,dx$, $\mathbf m_\epsilon=m_\epsilon\,dt\,dx$ and $m_\epsilon=\rho_\epsilon u_\epsilon$. From Lemma~\ref{lem: main_link_to_density_problem}, there exists a law  $\mathrm{P}^{u_\epsilon}\in\mathcal P(\Omega)$ satisfying $\mathrm{P}^{u_\epsilon}_t=\rho_\epsilon(t,\cdot)dx$, for every $t\in[0,1]$, with the prescribed endpoint traces $\mathrm{P}^{u_\epsilon}_0=\rho_0(\cdot)dx$ and $\mathrm{P}^{u_\epsilon}_1=\rho_f(\cdot)dx$ and $\epsilon\KL(\mathrm{P}^{u_\epsilon} \| \mathrm{R}^{\epsilon,\mu_0})\le
\mathcal J(\rho_\epsilon,m_\epsilon)$. From~\eqref{kl_control_cost},~\eqref{eq: optimal_bridge} and~\eqref{eq:weak-fp-duality}, we get that $D(\lambda)\leq J_\epsilon^*=\epsilon\KL(\mathrm{P}^{u_\epsilon^*} \| \mathrm{R}^{\epsilon,\mu_0})\le\epsilon\KL(\mathrm{P}^{u_\epsilon} \| \mathrm{R}^{\epsilon,\mu_0})\le
\mathcal J(\rho_\epsilon,m_\epsilon)$. This proves weak duality $\mathcal D(\lambda)\leq\mathcal J(\rho_\epsilon,m_\epsilon)$.

We next verify dual attainment. For~\eqref{eq:optimal-dual}, a direct computation gives
%$\lambda_\epsilon^*=\epsilon\log\varphi_\epsilon$, the backward equation in~\eqref{eq:Schr-system-thm} and $G=gg^\top$ We now verify that our candidate of optimality satisfy \[ m^*_\epsilon=\rho^*_\epsilon g^\top\nabla\lambda^*,\qquad\text{where}\qquad H_\lambda^*=0. \]
%For the dual problem, consider~\eqref{eq:optimal-dual}.
\begin{multline*}
H_{\lambda_\epsilon^*}
=
\epsilon\frac{\partial_t\varphi_\epsilon+\mathcal L_\epsilon\varphi_\epsilon}
{\varphi_\epsilon}
-\frac{\epsilon^2}{2}
\frac{\langle\nabla\varphi_\epsilon,G\nabla\varphi_\epsilon\rangle}
{\varphi_\epsilon^2}
+\frac{\epsilon^2}{2}
\frac{\langle\nabla\varphi_\epsilon,G\nabla\varphi_\epsilon\rangle}
{\varphi_\epsilon^2}=0,\\ v_{\lambda_\epsilon^*}
=\epsilon g^\top\nabla\log\varphi_\epsilon
=u_\epsilon^*\text{ and }|A_\epsilon\lambda_\epsilon^*|=\frac12\|u_\epsilon^*\|^2.
\end{multline*}
Therefore, from~\eqref{eq: finite_entropy=finite_energy} $\E_{\mathrm{P}^{u^*_{\epsilon}}}\int_0^1
|A_\epsilon\lambda_\epsilon^*(t,X_t)|\,dt
=
\frac12\E_{\mathrm{P}^{u^*_{\epsilon}}}\int_0^1
|u_\epsilon^*(t,X_t)|^2\,dt
<\infty$. Also, by the continuous trace in~\eqref{eq:Schr-system-thm}, we get that $\lim_{t\downarrow0}\lambda_\epsilon^*(t,x)=\epsilon\log\varphi_\epsilon(0,x)$ and $\lim_{t\uparrow1}\lambda_\epsilon^*(t,x)=\epsilon\log\varphi_\epsilon(1,x)$, $\mathrm{P}^{u^*_{\epsilon}}$-a.s. Thus, for $\varphi_\epsilon(1,x)=\phi_{f,\epsilon}(x)$, where $\phi_{0,\epsilon},\phi_{f,\epsilon}$ are in~\eqref{eq:Sch-multplication_system}, following from~\eqref{eq:Schr-system-thm}, we get that $\varphi_\epsilon(0,x)=\frac{1}{\phi_{0,\epsilon}(x)}$. 
To show that $\log\phi_{0,\epsilon}\in L^1(\mu_0)$ and $\log\phi_{f,\epsilon}\in L^1(\mu_f)$, we verify the hypothesis in~\cite[Theorem~2.1(b)]{nutz2021introduction}. To this end, since, from~\eqref{eq: finite_energy_bound}, we have $\int_{\R^d\times\R^d}\log\frac{d(\mu_0\otimes\mu_f)}{r_\epsilon(dx,xy)}d(\mu_0\otimes\mu_f)=\KL(\mu_0\otimes\mu_f \| r_\epsilon)<\infty$ and the negative part$\int_{\R^d\times\R^d}\left(\log\frac{d(\mu_0\otimes\mu_f)}{r_\epsilon(dx,xy)}\right)^{-}d(\mu_0\otimes\mu_f)<\int_{\R^d\times\R^d} r_\epsilon(dx,dy)=1$, we get that $\log\frac{d(\mu_0\otimes\mu_f)}{r_\epsilon(dx,xy)}\in L^1(\mu_0\otimes\mu_f)$. Therefore, by ~\cite[Theorem~2.1(b)]{nutz2021introduction}, we conclude that  $\log\phi_{0,\epsilon}\in L^1(\mu_0)$ and $\log\phi_{f,\epsilon}\in L^1(\mu_f)$.  Taking the log of~\eqref{eq: martingale_conditional_expectation}, we get that~\eqref{eq: finite_energy_bound} simplifies to 
\[
\lambda_\epsilon^*(t,X_t)=\epsilon\log Z_t-\epsilon\log\phi_{0,\epsilon}(X_0),
\]
and hence  
\begin{multline*}
\E_{\mathrm{P}^{u^*_{\epsilon}}}
|\lambda_\epsilon^*(t,X_t)+\epsilon\log\phi_{0,\epsilon}(X_0)|
=
\epsilon\E_{\mathrm{P}^{u^*_{\epsilon}}}|\log Z_t|\quad \text{ and } \\\E_{\mathrm{P}^{u^*_{\epsilon}}}
|\lambda_\epsilon^*(t,X_t)-\epsilon\log\phi_{f,\epsilon}(X_1)|
=
\epsilon\E_{\mathrm{P}^{u^*_{\epsilon}}}|\log Z_t-\log Z|.    
\end{multline*}
Thus, the $L^1(\mathrm{P}^{u^*_{\epsilon}})$ convergence of $\lambda_\epsilon^*(t,X_t)$ reduces to establishing uniform integrability (U.I) of the family $\{\log Z_t\}$. For the negative part $(\log Z_t)^-$, let $K>0$. Then, since $(\log Z_t)^->K$ implies that $0<Z_t<e^{-K}$, we get that 
\[
\E_{\mathrm{P}^{u^*_{\epsilon}}}
\left[(\log Z_t)^-\mathbf 1_{\{(\log Z_t)^->K\}}\right]=\E_{\mathrm{R}^{\epsilon,\mu_0}}
\left[Z_t(-\log Z_t)\mathbf 1_{\{0<Z_t<e^{-K}\}}\right]\leq Ke^{-K},
\]
for all $t\in[0,1]$. Therefore, $\{(\log Z_t)^-\}$ is U.I. For $(\log Z_t)^+$, since $\E_{\mathrm{P}^{u^*_{\epsilon}}}\log Z=\KL(\mathrm{P}^{u^*_{\epsilon}}\|R)<\infty$ and $\E_{\mathrm{P}^{u^*_{\epsilon}}}(\log Z)^-=\E_{\mathrm{R}^{\epsilon,\mu_0}}
\left[Z(-\log Z)\mathbf 1_{\{0<Z<1\}}\right]<1$, we have that $\log Z\in L^1(\mathrm{P}^{u^*_{\epsilon}})$. Using $z\mapsto z\log z$, by the conditional Jensen and conditional Bayes formula, since $(\log Z_t)^+
\le
\E_{\mathrm{P}^{u^*_{\epsilon}}}[(\log Z)^+\mid\mathcal F_t]$ and $(\log Z)^+\in L^1(\mathrm{P}^{u^*_{\epsilon}})$ from~\cite[Section~13.4]{williams1991probability}, we get that $\{\E_{\mathrm{P}^{u^*_{\epsilon}}}[(\log Z)^+\mid\mathcal F_t]\}$ is U.I and hence $\{(\log Z_t)^+\}$ is U.I. Thus, $\{\log Z_t\}$ is U.I. By Martingale convergence Theorem, since $\lim_{t\rightarrow 0}\log Z_t=0$ and  $\lim_{t\rightarrow 1}\log Z_t=\log Z:=\log(\phi_{0,\epsilon}( X_0)\phi_{f,\epsilon}( X_{1}))$,  $\mathrm{P}^{u^*_{\epsilon}}$-a.s and $\{\log Z_t\}$ is U.I, we get that
\begin{multline*}
 \lim_{t\rightarrow 0}\E_{\mathrm{P}^{u^*_{\epsilon}}}
|\lambda_\epsilon^*(t,X_t)+\epsilon\log\phi_{0,\epsilon}(X_0)|
=
\lim_{t\rightarrow 0}\epsilon\E_{\mathrm{P}^{u^*_{\epsilon}}}|\log Z_t|=0\quad \text{and} \\\lim_{t\rightarrow 1}\E_{\mathrm{P}^{u^*_{\epsilon}}}
|\lambda_\epsilon^*(t,X_t)-\epsilon\log\phi_{f,\epsilon}(X_1)|
=
\lim_{t\rightarrow 1}\epsilon\E_{\mathrm{P}^{u^*_{\epsilon}}}|\log Z_t-\log Z|=0.   
\end{multline*}
Therefore,  $\lambda_\epsilon^*\in\Lambda(\mathrm{P}^{u^*_{\epsilon}})$.

Consider~\eqref{eq:optimal-primal}. Using \eqref{eq:Schr-system-thm}, the pair $(\rho_{\epsilon}^*,m_{\epsilon}^*)$ satisfies
\begin{align*}
\partial_t\rho_{\epsilon}^*
&= \hat\varphi_{\epsilon}\,\partial_t\varphi_{\epsilon} + \varphi_{\epsilon}\,\partial_t\hat\varphi_{\epsilon}
= -\hat\varphi_{\epsilon}\,\mathcal L_\epsilon\varphi_{\epsilon} + \varphi_{\epsilon}\,\mathcal L_\epsilon^\ast\hat\varphi_{\epsilon} \\
&= -\hat\varphi_{\epsilon}\Big(b_\epsilon\cdot\nabla\varphi_{\epsilon} + \frac{\epsilon}{2}\tr(G\nabla^2\varphi_{\epsilon})\Big)
\;+\;\varphi_{\epsilon}\Big(-\nabla\cdot(b_\epsilon\hat\varphi_{\epsilon})+\frac{\epsilon}{2}\sum_{i,j=1}^d\partial_{ij}(G_{ij}\hat\varphi_{\epsilon})\Big) \\
&= -\nabla\cdot(b_\epsilon\,\rho_{\epsilon}^*)
\;+\;\frac{\epsilon}{2}\sum_{i,j=1}^d \partial_{ij}\big(G_{ij}\rho_{\epsilon}^*\big)
\;-\;\epsilon\,\nabla\cdot\big(\hat\varphi_{\epsilon}\,G\nabla\varphi_{\epsilon}\big),
\end{align*}
where the cross terms combine to yield precisely $\nabla\cdot(gm_{\epsilon}^*) = \nabla\cdot(\epsilon\hat\varphi_{\epsilon} G\nabla\varphi_{\epsilon})$. Hence the pair $(\rho_{\epsilon}^*,m_{\epsilon}^*)$ solves~\eqref{eq: fb_constraint}. For finite energy, since
\[
\int_0^1\int_{\R^d} \frac{\|m_{\epsilon}^*\|^2}{2\rho_{\epsilon}^*} dxdt = \int_0^1\int_{\R^d} \hat{\varphi_{\epsilon}}  \varphi_{\epsilon} \frac{\|\epsilon g^\top \nabla(\log\varphi_{\epsilon})\|^2}{2}   dxdt=\int_0^1\int_{\R^d}  \frac12\|u^*\|^2\rho_{\epsilon}^*   dxdt,
\]
from~\eqref{kl_control_cost} and the energy estimates in~\eqref{eq: finite_energy_bound} we conclude that
\[
\mathcal D(\lambda^*)=\mathcal J(\rho^*_\epsilon,m^*_\epsilon)=\int_0^1\int_{\R^d} \frac{\|m_{\epsilon}^*\|^2}{2\rho_{\epsilon}^*} dxdt<\infty.
\]

For uniqueness, let $(\rho_\epsilon,m_\epsilon)$ be any
primal minimizer. By
Lemma~\ref{lem: main_link_to_density_problem}, there exists   $\mathrm{P}^{u_\epsilon}\in\mathcal P(\Omega)$ satisfying $\mathrm{P}^{u_\epsilon}_t=\rho_\epsilon(t,\cdot)dx$, for every $t\in[0,1]$, with the prescribed endpoint traces $\mathrm{P}^{u_\epsilon}_0=\rho_0(\cdot)dx$ and $\mathrm{P}^{u_\epsilon}_1=\rho_f(\cdot)dx$ and $J_\epsilon^*
=\mathcal J(\rho_\epsilon,m_\epsilon)
\ge
\epsilon\KL(\mathrm{P}^{u_\epsilon}\|\mathrm R^{\epsilon,\mu_0})
\ge
\epsilon\KL(
\mathrm P^{u_\epsilon^*}\|\mathrm R^{\epsilon,\mu_0})=J_\epsilon^*
$. Thus, $\mathrm{P}^{u_\epsilon}$ minimizes $\mathrm{P}\mapsto\KL(\mathrm{P} \| \mathrm{R}^{\epsilon,\mu_0})$ over~\eqref{eq: SB_constraint}. From~\eqref{eq: optimal_bridge} we get that  $\mathrm{P}^{u_\epsilon}=\mathrm P^{u_\epsilon^*}$ uniquely. This implies that its time-marginal density in $\mathrm{P}^{u_\epsilon}_t=\rho_\epsilon(t,\cdot)dx$ is unique $\rho_\epsilon=\rho_\epsilon^*$ $dtdx-a.e.$ Let $m_\epsilon=\rho^*_\epsilon u_\epsilon$ and $m^*_\epsilon=\rho^*_\epsilon u_\epsilon^*$. By linearity of the Fokker--Planck constraint,
$(\rho_\epsilon^*,(m_\epsilon+m_\epsilon^*)/2)$ is feasible.
Moreover, $\mathcal J\left(
\rho_\epsilon^*,\frac{m_\epsilon+m_\epsilon^*}{2}
\right)
=
J_\epsilon^*
-\frac18\int_Q
|u_\epsilon-u_\epsilon^*|^2\rho_\epsilon^*\,dt\,dx\geq J_\epsilon^*$. The last inequality is due to  
minimality and this forces $\frac18\int_Q
|u_\epsilon-u_\epsilon^*|^2\rho_\epsilon^*\,dt\,dx=0$ and hence $u_\epsilon=u_\epsilon^*$ and $m_\epsilon=m_\epsilon^*$ almost everywhere,
proving uniqueness.
\end{proof}
From this the optimal control is characterized as
\begin{equation}\label{eq:optimal-control}
u_\epsilon^*:=\frac{m_{\epsilon}^*}{\rho_{\epsilon}^*}=g^\top\nabla\lambda_{\epsilon}^*=\epsilon  g^\top\nabla\log\varphi_{\epsilon},
\end{equation}
where $\varphi_{\epsilon}$ is a strictly positive functions and together with $\hat\varphi_{\epsilon}$ solves the Schr\"odinger system~\eqref{eq:Schr-fbwd}-\eqref{eq:Schr-bc}.

We are now ready to provide a proof for Theorem~3.1.
\begin{proof}[Proof of
Theorem~\ref{thm: existence and uniqueness}]
By Theorem~\ref{prop: Strong duality and attainment}, the
density problem admits a unique minimizer
$(\rho_\epsilon^*,m_\epsilon^*)$. Define $\boldsymbol\mu_\epsilon^*
:=\rho_\epsilon^*\,dt\,dx$ and $\mathbf m_\epsilon^*
:=m_\epsilon^*\,dt\,dx$. Then, $(\boldsymbol\mu_\epsilon^*,\mathbf m_\epsilon^*)$ satisfy~\eqref{eq: fb_measure_constraint}-\eqref{eq:endpoint_measure_constraints} and $\mathcal J(\boldsymbol\mu_\epsilon^*,\mathbf m_\epsilon^*)=\mathcal J(\rho_\epsilon^*,m_\epsilon^*)<\infty$. Let $(\boldsymbol\mu_\epsilon,\mathbf m_\epsilon)$ satisfy~\eqref{eq: fb_measure_constraint}-\eqref{eq:endpoint_measure_constraints} and $J(\boldsymbol\mu_\epsilon,\mathbf m_\epsilon)<\infty$. By
Lemma~\ref{lem: main_link_to_density_problem}, there exists
a path law $\mathrm{P}^{u_\epsilon}$ with marginal curve $\mu_{\epsilon,t}$ and $\epsilon\KL(\mathrm{P}^{u_\epsilon}\|\mathrm R^{\epsilon,\mu_0})
\le
\mathcal J(\boldsymbol\mu_\epsilon,\mathbf m_\epsilon)
<\infty$.
Hence $\mathrm{P}^{u_\epsilon}\ll\mathrm R^{\epsilon,\mu_0}$.
Since the reference time marginals admit Lebesgue densities
for $0<t<1$, so do $\mu_{\epsilon,t}$.
Consequently,
$\boldsymbol\mu_\epsilon=\rho_\epsilon\,dt\,dx$.
Writing
$u_\epsilon=d\mathbf m_\epsilon/d\boldsymbol\mu_\epsilon$,
we also obtain $\mathbf m_\epsilon=m_\epsilon\,dt\,dx$ and $m_\epsilon=\rho_\epsilon u_\epsilon$, where $(\rho_\epsilon,m_\epsilon)$ satisfies~\eqref{eq: fb_constraint}-\eqref{eq:endpoint_constraints}
and $\mathcal J(\boldsymbol\mu_\epsilon,\mathbf m_\epsilon)
=
\mathcal J(\rho_\epsilon,m_\epsilon)
\ge
\mathcal J(\rho_\epsilon^*,m_\epsilon^*)
=
\mathcal J(\boldsymbol\mu_\epsilon^*,\mathbf m_\epsilon^*)$.
Thus $(\boldsymbol\mu_\epsilon^*,\mathbf m_\epsilon^*)$
is a minimizer.

If $(\boldsymbol\mu_\epsilon,\mathbf m_\epsilon)$ is another
minimizer, its density pair also minimizes the density
problem. Uniqueness in
Theorem~\ref{prop: Strong duality and attainment}
therefore gives
$\rho_\epsilon=\rho_\epsilon^*$ and
$m_\epsilon=m_\epsilon^*$ almost everywhere.
Hence the measure pairs coincide. Their narrowly continuous
marginal curves also coincide at every time, proving uniqueness.
\end{proof}

\section{From Schr\"odinger Bridge to Optimal Transport}\label{sec: Limiting Section}
In this section we analyze the zero-noise limit $\epsilon \downarrow 0$ of the entropic (horizontal) Schr\"odinger bridge and show that the sequence of solutions $(\boldsymbol{\mu}_\epsilon^*, \mathbf{m}_\epsilon^*)$ of Problem~\eqref{measure_theoretic_objective} converges, up to subsequences, to a solution of the sub-Riemannian Benamou--Brenier optimal transport problem for~\eqref{eq:intro_det_dyn}-\eqref{eq:intro_det_obj}.
\begin{theorem}\label{thm:Gamma}
Consider the problem
\begin{equation}\label{measure_theoretic_objective2}
\inf_{(\boldsymbol\mu, \mathbf m)\in\mathcal C} 
\mathcal{J}(\boldsymbol\mu, \mathbf m):=
\int_0^1\int_{\R^d} \frac12\Big\|\frac{d\mathbf m(t,x)}{d\boldsymbol\mu(t,x)}\Big\|^2 d\boldsymbol\mu%(dt,dx)
\end{equation}
subject to the continuity equation
\begin{subequations}\label{eq:CE_measure_unref}
\begin{align}
\partial_t \boldsymbol\mu + \nabla_x \cdot \left(g\mathbf m\right) &= 0 \quad \text{in } \bigl(C_c^\infty(Q)\bigr)^*, \label{eq:CE_measure_constraint}\\
\boldsymbol\mu|_{t=0} = \mu_0, \qquad \boldsymbol\mu|_{t=1} &= \mu_f \quad \text{in } \mathcal{P}(\mathbb{R}^d). \label{eq:CE_measure_endpoint}
\end{align} 
\end{subequations}
Let $(\boldsymbol{\mu}_\epsilon^*,\mathbf m_\epsilon^*)$
be the unique minimizer of
Problem~\eqref{measure_theoretic_objective}
for each $\epsilon>0$.
For every sequence $\epsilon_n\downarrow0$, there exists
a subsequence $\epsilon_{n_k}$ and an admissible pair
$(\boldsymbol{\mu}^*,\mathbf m^*)$ satisfying
\eqref{eq:CE_measure_unref} such that $(\boldsymbol{\mu}_{\epsilon_{n_k}}^*,
 \mathbf m_{\epsilon_{n_k}}^*)
\rightharpoonup
(\boldsymbol{\mu}^*,\mathbf m^*)$
in the weak topology of finite measures, componentwise
for the momentum measures.
Every such limit pair minimizes
\eqref{measure_theoretic_objective2}
subject to \eqref{eq:CE_measure_unref}.
Moreover,
\begin{equation}\label{eq:aux1}
\lim_{\epsilon\downarrow0}
\mathcal J(\boldsymbol{\mu}_\epsilon^*,\mathbf m_\epsilon^*)
=
\mathcal J(\boldsymbol{\mu}^*,\mathbf m^*)
=
\min_{\substack{(\boldsymbol\mu,\mathbf m)\in\mathcal C\\
                 \text{\eqref{eq:CE_measure_unref}}}}
\mathcal J(\boldsymbol\mu,\mathbf m).
\end{equation}
 %Let $(\boldsymbol{\mu}_\epsilon^*, \mathbf{m}_\epsilon^*)$ be the unique minimizer of Problem~\eqref{measure_theoretic_objective} for each $\epsilon > 0$. Then along the convergent subsequence, we have that
       % \begin{equation}\label{eq:aux1}
       %     \lim_{\epsilon \to 0} \mathcal{J}(\boldsymbol{\mu}_\epsilon^*, \mathbf{m}_\epsilon^*) = \mathcal{J}(\boldsymbol{\mu}^*, \mathbf{m}^*) = \min_{(\boldsymbol\mu,\mathbf m)\text{ satisfying \eqref{eq:CE_measure_constraint}--\eqref{eq:CE_measure_endpoint}}} \mathcal J(\boldsymbol\mu,\mathbf m).
     %   \end{equation}
\end{theorem}
We first provide a proof under Assumption~\ref{eq:density_regularities}, item~\ref{ass: compact support} and later discuss the proof under item~\ref{ass:two_tailed_Gaussian}. For this, the following will be useful.
\begin{proposition}\label{prop: exponential_tightness of reference law}
Under Assumption~\ref{eq:density_regularities}, item~\ref{ass: compact support}, the family of the probability distribution $\{\mathrm{R}^{\epsilon,\mu_0}\}_{0<\epsilon\leq1}$  in~\eqref{eq: law of unreflected process}
is exponentially tight on $\Omega$, equipped with the uniform topology.  
\end{proposition}
\begin{proof}
Using $\mathrm{R}^{\epsilon,\mu_0}(d\omega)=\int_{\R^d}\mathrm{R}^{\epsilon,x}(d\omega)\mu_0(dx)$, the exponential tightness of reduces to showing that for any $\delta>0$ there exists compact set $\mathcal K_\delta\subset\Omega$ such that 
\begin{equation}\label{eq: uniform exponentially tight}
\sup_{x\in\mathrm K_0}\mathrm{R}^{\epsilon,x}(\mathcal K_\delta^c)\leq Ce^{-\frac{c\delta}{\epsilon}}
\end{equation}
To this end, consider $V(x)=\log(1+|x|^2)$. Then, under $\mathrm{R}^{\epsilon,x}$, the Ito formula for $V(X_t)$, where $X_t$ satisfies~\eqref{uncontrolled_dynamic_constraint}, is $V(X_t) = V(x) + \int_0^t \mathcal L_\epsilon V(X_s)\,ds+M_t^\epsilon$, where $\mathcal L_\epsilon$ is in~\eqref{eq:Ito-generator}, with $\nabla V(x)=\frac{2x}{1+|x|^2}$ and  $\nabla_x^2V(x)=\frac{2I}{1+|x|^2} - \frac{4xx^\top}{(1+|x|^2)^2}$ and $M_t^\epsilon :=  \sqrt{\epsilon}\int_0^t\, \nabla V(X_s)^\top g(X_s)dW_s$. Under Assumption~\ref{ass:wellposed}, since $|\mathcal L_\epsilon V(x)|\leq c_0$, we get $V(X_t) \le v_{\mathrm K_0}+c_0+M_t^\epsilon$, with $v_{\mathrm K_0}:=\max_{x\in\mathrm K_0}\log(1+|x|^2)$. Hence, since $\{|X_t|\geq R\}\subset \{|M_t^\epsilon|\geq z_R\}$, where $z_R:=\log(1+R^2)-v_{\mathrm K_0}-c_0$  and $|\nabla V(X_s)^\top g(X_s)|\leq c_1$, following from~\cite[pp~153]{revuz2013continuous}, we get that 
\begin{equation}\label{eq: high probability of uniform_bounded}
\sup_{x\in\mathrm K_0}\mathrm{R}^{\epsilon,x}\left(\sup_{t\le1}|X_t|\ge R\right)
\le e^{-a_R/\epsilon},    
\end{equation}
where $a_R:=\frac{z_R^2}{2c^2_1}$. Thus, for $\max_{x\in\mathrm K_0}|x|<R$, we define the localization $Y_t := X_{t\wedge\tau_R}$, where $\tau_R = \inf\{t\ge0 : |X_t|\ge R\}$ is the stopping time. Then $\{Y_t\}$ has bounded drift $b_R:=\sup_{|x|\le R}|b_\epsilon(x)|$ and diffusion coefficient $g_R = \max\left\{1, \sup_{|x|\le R}\|g(x)\|\right\}$, where the norm on $g$ is the Hilbert-Schmidt norm. Suppose $r>2b_Rh$, where $h=t-s>0$, then we get incremental estimate to be  $\sup_{x\in\mathrm K_0}\mathrm{R}^{\epsilon,x}\left(|Y_t-Y_s|\ge r\right)
\le
2d\exp\left(-\frac{r^2}{8d\epsilon g_R^2h}\right)$. Consider $0,\;2^{-n},\;2\cdot 2^{-n},\;\dots,\;1$, where $n=1,2,\dots$ and for each $n\geq 1$, the adjacent intervals $[k2^{-n},(k+1)2^{-n}]$, where $k=0,\ldots,2^n-1$. Then, setting $r = L2^{-n/4}$, we get that 
\begin{equation}\label{eq: high probability of oscillation}
\sup_{x\in\mathrm K_0}\mathrm{R}^{\epsilon,x}\left(\bigcup_{n=1}^{\infty}
\bigcup_{k=0}^{2^n-1}|Y_{(k+1)2^{-n}}^\epsilon - Y_{k2^{-n}}^\epsilon| > L2^{-n/4}\right) \le C_d e^{-c_RL^2/(2\epsilon)},    
\end{equation}
for some positive constants $c_R,C_d$. The standard dyadic chaining argument (see the proof of~\cite[Theorem~2]{bell2015kolmogorov}) shows that a continuous path outside $\bigcup_{n=1}^{\infty}
\bigcup_{k=0}^{2^n-1}|Y_{(k+1)2^{-n}}^\epsilon - Y_{k2^{-n}}^\epsilon| > L2^{-n/4}$ satisfies $|Y_t - Y_s| \le CL|t-s|^{1/4}$, for all $s,t\in[0,1]$. Hence we consider the compact set
\[
\mathcal K_{R,L} = \left\{\omega\in\Omega: \omega_0\in \mathrm K_0,\ \sup_{t\le1}|\omega_t|\le R,\ |\omega_t-\omega_s| \le CL|t-s|^{1/4}\ \forall s,t \right\}.
\]
Then, from~\eqref{eq: high probability of uniform_bounded} and~\eqref{eq: high probability of oscillation} $\sup_{x\in\mathrm K_0}\mathrm{R}^{\epsilon,x}(\mathcal K_{R,L}^c) \le e^{-a_R/\epsilon} + C_d e^{-c_RL^2/(2\epsilon)}$. Hence for any $\delta>0$, choose $R$ large such that $a_R>2\delta$ and for such $R$ an $L$ sufficiently large enough such that $c_RL^2/2 \ge 2\delta$. Then by setting $\mathcal K_{\delta}=\mathcal K_{R,L}$, we get~\eqref{eq: uniform exponentially tight}. %This completes the proof.
\end{proof}

\begin{proof}[Proof of Theorem~\ref{thm:Gamma}]
Under Assumption~\ref{eq:density_regularities}, item~\ref{ass: compact support}, we show this in three steps; Compactness of the sequence of minimizers, the limit sequence satisfying continuity equation and the value convergence in~\eqref{eq:aux1}. To this end, since the unique minimizer $(\boldsymbol{\mu}_\epsilon^*, \mathbf{m}_\epsilon^*)$ of Problem~\eqref{measure_theoretic_objective} is characterized as $d\boldsymbol{\mu}_{\epsilon}^* = \rho_{\epsilon}^* dt dx$, $d\mathbf{m}_{\epsilon}^* = m_{\epsilon}^* dt dx$, where $(\rho_{\epsilon}^*,m_{\epsilon}^*)$ is in~\eqref{eq:optimal-primal}, from~\eqref{kl_control_cost}, we have that
\[
\int_0^1\int_{\R^d} \frac12\Big\|\frac{d\mathbf m_{\epsilon}^*(t,x)}{d\boldsymbol\mu^*_\epsilon(t,x)}\Big\|^2 d\boldsymbol\mu^*_\epsilon=\int_0^1\int_{\R^d}  \frac{\|m_{\epsilon}^*(t,x)\|^2}{2\rho_\epsilon^*(t,x)} dx\, dt =\epsilon\KL(\mathrm{P}^{u_\epsilon^*} \| \mathrm{R}^{\epsilon,\mu_0})
\]
holds, for all $\epsilon\in (0,1]$, where $u_\epsilon^*$ is in~\eqref{eq:optimal-control}. Thus, from~\eqref{eq: finite_energy_bound} we have that 
\begin{equation*}
\int_0^1\int_{\R^d} \frac12\Big\|\frac{d\mathbf m^*_\epsilon(t,x)}{d\boldsymbol\mu^*_\epsilon(t,x)}\Big\|^2 d\boldsymbol\mu^*_\epsilon\leq \epsilon\int_{\R^d}\rho_f(y)\log(\rho_f(y))dy-\int_{\R^d\times\R^d}\epsilon\log(p_{1,\epsilon}(x,y))\rho_0(x)\rho_f(y)dxdy,    
\end{equation*}
 for all $\epsilon\in (0,1]$. By Assumption~\ref{ass:wellposed}, since $d_{\mathrm{SR}}\bigl(x,\{y:|y|\ge R\}\bigr)
\ge
\frac1{C_g}\log\frac{1+R}{1+|x|}
\longrightarrow\infty$, we get that every compact endpoint set satisfies their strong localization condition (see~\cite[Definition~1.1]{neel2026uniform}). Therefore, from~\cite[Theorem~1.2]{neel2026uniform}, we have that  
 \begin{equation}\label{eq: uniform convergence of the log of the heat kernel}
 \lim_{\epsilon\to 0}-\epsilon\log(p_{1,\epsilon}(x,y))=\tfrac{1}{2} \dSR(x,y)^2\quad\text{ uniformly in }\quad(x,y)\in\mathrm K_0\times \mathrm K_f.
 \end{equation}
Thus,
 \begin{equation}
    \label{eq:uniform_energy_explicit}
    \int_0^1\int_{\R^d} \frac12\Big\|\frac{d\mathbf m^*_\epsilon(t,x)}{d\boldsymbol\mu^*_\epsilon(t,x)}\Big\|^2 d\boldsymbol\mu^*_\epsilon=\int_0^1\int_{\R^d}  \frac{\|m_{\epsilon}^*(t,x)\|^2}{2\rho_\epsilon^*(t,x)} dx\, dt \leq C,
\end{equation}
where $C$ is independent of $\epsilon$. Hence, the Cauchy--Schwarz inequality, for any Borel set $B \subset Q$ yields:
\begin{equation}
    |\mathbf m^*_\epsilon|(B) = \int_B \Big\|\frac{d\mathbf m^*_\epsilon(t,x)}{d\boldsymbol\mu^*_\epsilon(t,x)}\Big\| d\boldsymbol\mu^*_\epsilon \leq \left( 2C \right)^{1/2} \boldsymbol\mu^*_\epsilon(B)^{1/2}.
\end{equation}
Thus tightness of $\{\boldsymbol\mu^*_\epsilon\}_\epsilon$ implies tightness of $\{\mathbf m^*_\epsilon\}_\epsilon$. Since $Q=[0,1]\times\R^d$ and $\boldsymbol\mu^*_\epsilon(dt,dx)=\mu^*_{\epsilon,t}(dx)dt$, it is enough to show tightness of $\mu^*_{\epsilon,t}$. To this end, let $\mathrm{P}^{u_{\epsilon}^*}$ be the probability distribution of the optimal process $ X^{u^*_{\epsilon}}\in \Omega$ in~\eqref{eq:intro_sto_dyn}, where $u^*_\epsilon$ is in~\eqref{eq:optimal-control}.
Using the variational formula for relative entropy~\cite[\S6.6]{DemboZeitouni1998},\cite[\S 1.4]{DupuisEllis1997},
\[
D_{\mathrm{KL}}(\mathrm{P}^{u_{\epsilon}^*} \| \mathrm{R}^{\epsilon,\mu_0}) = \sup_{f} \left\{ \int f d\mathrm{P}^{u_{\epsilon}^*} - \log \int e^{f} d\mathrm{R}^{\epsilon,\mu_0} \right\},
\]
we get that 
\[
\mathrm{P}^{u_{\epsilon}^*}(K_\delta^c)\ \le\frac{ \KL(\mathrm{P}^{u_{\epsilon}^*}\|\mathrm{R}^{\epsilon,\mu_0})+\log 2}{\log\big(1/\mathrm{R}^{\epsilon,\mu_0}(K_\delta^c)\big)},\qquad\text{for all $0<\epsilon\leq 1$}
\]
%{\color{blue}
where $\{K_\delta\}_{\delta\in\mathbb{N}}\subset\Omega$ are compact sets.  From Propisition~\ref{prop: exponential_tightness of reference law} and~\eqref{eq:uniform_energy_explicit},
%from~\cite[Theorem~4.3]{BoueDupuis1998}, INCORRECT REFERENCE $\mathrm{R}^{\epsilon,x}$ is uniformly exponential tight (THIS MUST BE PROVED, NOT IN THE LITERATURE ACCORDING TO~\cite[Proposition~6]{biggins2004large}) (see for example~\cite[pp~8]{DemboZeitouni1998} for definition of exponential tightness), 
we have that $\mathrm{P}^{u_{\epsilon}^*}(K_\delta^c)\ \le\frac{C+\log 2}{\delta}$ for all $0<\epsilon\leq 1$ and hence 
\begin{equation}\label{eq: tightness_under_compact_support}
\sup_{0<\epsilon\leq 1} \mathrm{P}^{u_{\epsilon}^*}(K_\delta^c)\ \le\ \frac{C+\log 2}{\delta}\xrightarrow[\delta\rightarrow \infty]{}0    
\end{equation}
Hence, by Prokhorov Theorem~\cite[Theorem~5.1]{Billingsley1995}, we get that $\{\mathrm{P}^{u_{\epsilon}^*}\}_{0<\epsilon\leq 0}$ is tight, which concludes that $\mu^*_{\epsilon,t}=\mathrm{ev}_{t\#}\mathrm{P}^{u_{\epsilon}^*}$ is tight. Therefore, by Prokhorov's theorem, $(\boldsymbol{\mu}_\epsilon^*, \mathbf{m}_\epsilon^*)$ is precompact. Thus, there exists a subsequence (not relabeled) $(\boldsymbol{\mu}_\epsilon^*, \mathbf{m}_\epsilon^*) \in \mathcal{M}_+(\overline{Q}) \times \mathcal{M}(\overline{Q}; \mathbb{R}^m)$ satisfying % the Fokker-Planck system in the weak sense as space-time measures:
\begin{multline}\label{eq:FP-weak}
\int_0^1\int_{\R^d}\Big(
-\partial_t\phi(t,x)\boldsymbol\mu^*_\epsilon(dt,dx)
- \nabla_x\phi(t,x)\cdot \big(b_\epsilon(x)\boldsymbol\mu^*_\epsilon(dt,dx) + g(x)\,\mathbf m^*_\epsilon(dt,dx)\big) \\
- \tfrac{\epsilon}{2}\sum_{i,j}\partial^2_{x_ix_j}\phi(t,x) G_{ij}(x)\,\boldsymbol\mu^*_\epsilon(dt,dx)
\Big) -\int_{\mathbb R^d}\phi(0,x)\,\mu_0(dx)+\int_{\mathbb R^d}\phi(1,x)\,\mu_f(dx)=0    
\end{multline}
for all $\phi\in C_c^\infty([0,1]\times\mathbb R^d)$ and
\begin{equation}
    \label{eq:weak_convergence}
    (\boldsymbol{\mu}_\epsilon^*, \mathbf{m}_\epsilon^*) \rightharpoonup (\boldsymbol{\mu}^*, \mathbf{m}^*)\qquad \text{ in }\qquad \mathcal{M}_+(\overline{Q})\times\mathcal{M}(\overline{Q}; \mathbb{R}^m).
\end{equation}
Since $b_\epsilon \to 0$ locally uniformly and $\phi$ has compact spatial support, passing the limit in the integral yields that $(\boldsymbol{\mu}^*, \mathbf{m}^*)$ satisfies~\eqref{eq:CE_measure_unref}. Since $\mathcal{J}$ is  convex and lower semicontinuous with respect to weak$^*$ convergence for every $\epsilon\geq 0$, we have that
\begin{equation}
    \label{eq:lsc_cost}
   \mathcal{J}(\boldsymbol{\mu}^*, \mathbf{m}^*)  \leq \liminf_{\epsilon \to 0}\mathcal{J}(\boldsymbol{\mu}_\epsilon^*, \mathbf{m}_\epsilon^*).
\end{equation}

We proceed to prove the corresponding upper bound for the limit superior. To this end, let $\gamma_\epsilon^*$ be the unique coupling to $\min_{\gamma \in \Pi(\mu_0,\mu_f)}
\epsilon\KL(\gamma \| r_\epsilon)$, where $r_\epsilon$ is in~\eqref{eq: reference_endpoint_density}. Then,  $\mathrm{P}^{u_{\epsilon}^*}= \int_{\R^d\times\R^d} \mathbf{R}^{\epsilon,x,y} \gamma_\epsilon^*(dx,dy)$, where $\mathbf{R}^{\epsilon,x,y}$ is the bridge of the
process in~\eqref{uncontrolled_dynamic_constraint} from $x$ to $y$ and $\epsilon\KL(\mathrm{P}^{u_\epsilon^*} \| \mathrm{R}^{\epsilon,\mu_0})=\epsilon\KL(\gamma_\epsilon^*\| r_\epsilon)=\min_{\gamma \in \Pi(\mu_0,\mu_f)}
\epsilon\KL(\gamma \| r_\epsilon)=:V_\epsilon$. Let $W_\epsilon:=\min_{\gamma \in \Pi(\mu_0,\mu_f)} \epsilon D_{\mathrm{KL}}(\gamma \| \mu_0\otimes\mu_f)+\int_{\R^d\times\R^d} c(x,y) d\gamma$, where $c=\tfrac{1}{2} \dSR^2$. Then since $c$ is continuous, $c(x,y)\leq q^2(x)+q^2(y)$, for all $(x,y)\in\R^d\times\R^d$ and for compactly supported marginals $q^2\in L^1(\mu_0)$ and $q^2\in L^1(\mu_f)$, from~\cite[Theorem~5.10]{nutz2021introduction}, we get that $\lim_{\epsilon\rightarrow 0} W_\epsilon=W_0:=\min_{\gamma \in \Pi(\mu_0,\mu_f)} \int_{\R^d\times\R^d} c(x,y) d\gamma$. Since
\begin{equation}\label{eq: KL_for_entropy_num}
\epsilon D_{\mathrm{KL}}(\gamma \| r_\epsilon)=\epsilon D_{\mathrm{KL}}(\gamma \| \mu_0\otimes\mu_f)+\int_{\R^d\times\R^d} c_\epsilon(x,y) d\gamma+\epsilon\int_{\R^d} \log\rho_f(y) d\mu_f(y),
\end{equation}
where
$c_\epsilon := -\epsilon \log p_{1,\epsilon}$, we get that $|V_\epsilon-W_\epsilon-\epsilon H_f|\leq \sup_{\gamma \in \Pi(\mu_0,\mu_f)}
\left|\int_{\R^d\times\R^d} (c_\epsilon - c) d\gamma\right|$.
Here $H_f:=\int_{\R^d} \log\rho_f(y) d\mu_f(y)$. From~\eqref{eq: uniform convergence of the log of the heat kernel}, we have that, 
\begin{equation}\label{eq: convergence to sub-Riemannian metric}
\lim_{\epsilon \to 0} \sup_{\gamma \in \Pi(\mu_0,\mu_f)}
\left|\int_{\R^d\times\R^d} (c_\epsilon - c) d\gamma\right| = 0
\end{equation}
and hence 
\begin{equation}\label{eq: link_to_reverse_inequality}
\lim_{\epsilon\to 0} \mathcal{J}(\boldsymbol{\mu}_\epsilon^*, \mathbf{m}_\epsilon^*)=\lim_{\epsilon\to 0}
\min_{\gamma\in \Pi(\mu_0,\mu_f)} \epsilon D_{\mathrm{KL}}(\gamma \| r_\epsilon)
=
\min_{\gamma \in \Pi(\mu_0,\mu_f)}
\int_{\R^d\times\R^d} \tfrac{1}{2}\dSR^2(x,y)d\gamma.    
\end{equation}
%The stability of the minimizers follows from~\cite[Proposition~3.2]{bernton2022entropic}.

Let $(\tilde{\boldsymbol{\mu}}, \tilde{\mathbf{m}})$, where $\tilde{\boldsymbol{\mu}}(dt,dx)=\tilde{\mu}_t(dx)dt$, be any admissible pair satisfying \eqref{eq:CE_measure_unref} with $\mathcal{J}(\tilde{\boldsymbol{\mu}}, \tilde{\mathbf{m}}) < \infty$. Let $u:=\frac{d\tilde{\mathbf{m}}}{d\tilde{\boldsymbol{\mu}}}\in L^2(\tilde{\boldsymbol{\mu}};\R^m)$. Then, under Assumption~\ref{ass:wellposed}, since $\int_0^1\int_{\R^d}\frac{|g(x)u(t,x)|}{1+|x|}d\boldsymbol{\mu}\leq C\int_0^1\int_{\R^d}|u(t,x)|d\boldsymbol{\mu}\leq C\left(\int_0^1\int_{\R^d}|u(t,x)|^2d\boldsymbol{\mu}\right)^{\frac{1}{2}}$ and hence $\int_0^1\int_{\R^d}\frac{|g(x)u(t,x)|}{1+|x|}d\boldsymbol{\mu}\leq C\sqrt{2\mathcal{J}(\tilde{\boldsymbol{\mu}}, \tilde{\mathbf{m}})}<\infty$, from~\cite[pp~2]{bogachev2021ambrosio}, there exists a probability measure $\mathrm{\tilde P} \in \mathcal{P}(\Omega)$  such that 
\begin{equation}\label{eq: path-space_representation}
(\mathrm{ev}_t)_\# \mathrm{\tilde P} = \tilde{\boldsymbol{\mu}}_t,
\qquad
\mathcal{J}(\tilde{\boldsymbol{\mu}}, \tilde{\mathbf{m}})\geq \int_\Omega C(\omega) \mathrm{\tilde P}(d\omega),
\end{equation}
where
\begin{equation*}%\label{eq: rate_function}
\mathrm{C}(\omega):=
\begin{cases}
\displaystyle 
\int_{0}^{1}
\frac12\|\dot\omega_t\|_{G^\dagger}^2 dt,
&
\omega \in \mathrm{AC}([0,1]; \mathbb{R}^d), \; \dot{\omega}_t \in \mathcal{D}_{\omega_t} \text{ a.e.}, \\[10pt]
+\infty, & \text{otherwise}.
\end{cases}
\end{equation*}
Let $\tilde\gamma:=(\mathrm{ev}_0, \mathrm{ev}_1)_\# \mathrm{\tilde P} \in \Pi(\mu_0, \mu_f)$. Then, from~\eqref{eq: Riemannian metric}, since 
\[
\int_{\R^d\times\R^d} \tfrac{1}{2}\dSR^2(x,y) d\tilde\gamma=\int_{\R^d\times\R^d} \tfrac{1}{2}\dSR(\omega_0, \omega_1)^2 \mathrm{\tilde P}(d\omega)
\le \int_\Omega C(\omega) \mathrm{\tilde P}(d\omega)
\leq \mathcal{J}(\tilde{\boldsymbol{\mu}}, \tilde{\mathbf{m}})
\]
holds, from~\eqref{eq: link_to_reverse_inequality}, we get that $\lim_{\epsilon\to 0} \mathcal{J}(\boldsymbol{\mu}_\epsilon^*, \mathbf{m}_\epsilon^*)\leq \mathcal{J}(\tilde{\boldsymbol{\mu}}, \tilde{\mathbf{m}})$, for all admissible pair $(\tilde{\boldsymbol{\mu}}, \tilde{\mathbf{m}})$ and hence $\lim_{\epsilon\to 0} \mathcal{J}(\boldsymbol{\mu}_\epsilon^*, \mathbf{m}_\epsilon^*)=\limsup_{\epsilon \to 0} \mathcal{J}(\boldsymbol{\mu}_\epsilon^*, \mathbf{m}_\epsilon^*)
\le \min_{(\boldsymbol{\mu},\mathbf{m})} \mathcal{J}(\boldsymbol{\mu},\mathbf{m})$,
thus,
\begin{equation}
    \mathcal{J}(\boldsymbol{\mu}^*, \mathbf{m}^*) \leq \liminf_{\epsilon \to 0} \mathcal{J}(\boldsymbol{\mu}_\epsilon^*, \mathbf{m}_\epsilon^*) \leq \limsup_{\epsilon \to 0} \mathcal{J}(\boldsymbol{\mu}_\epsilon^*, \mathbf{m}_\epsilon^*) \leq \mathcal{J}(\boldsymbol{\mu}^*, \mathbf{m}^*).
\end{equation}
This implies that all inequalities are equalities. Combined with precompactness and the fact that every limit point is a minimizer, this gives
convergence of the minimizers up to a subsequence.
\end{proof}
Under Assumption~\ref{eq:density_regularities}, item~\ref{ass:two_tailed_Gaussian}, two steps of the
above proof require modification. First, $\mu_0$ is not compactly supported and hence $\mathrm{R}^{\epsilon,\mu_0}(X_0\in B^c)=\mu_0(B^c)>0$, for every compact set $B\in\R^d$. Hence the argument leading to
\eqref{eq: tightness_under_compact_support} cannot be
applied directly. Second, for noncompact marginals, some mass lies outside every compact set. Hence we still have to justify~\eqref{eq: convergence to sub-Riemannian metric}. The following results address these issues.  
\begin{proposition}\label{prop: optimal law is tight implies their marginal is tight}
Under Assumption~\ref{eq:density_regularities}, item~\ref{ass:two_tailed_Gaussian}, the family of the probability distribution $\{\mathrm{P}^{u_{\epsilon}^*}\}_{0<\epsilon\leq1}$ of the optimal process $ X^{u^*_{\epsilon}}\in \Omega$, where $u^*_\epsilon$ is in~\eqref{eq:optimal-control} in~\eqref{eq:intro_sto_dyn}% . Then,
is tight on $\Omega$, equipped
with the uniform topology.  
\end{proposition}
\begin{proof}
Consider $\mathrm{P}^{u_{\epsilon}^*}= \int_{\R^d\times\R^d} \mathbf{R}^{\epsilon,x,y} \gamma_\epsilon^*(dx,dy)$, where $\mathbf{R}^{\epsilon,x,y}$ is the bridge of the reference
process in~\eqref{uncontrolled_dynamic_constraint} and $\gamma_\epsilon^*$ be the unique coupling to $\min_{\gamma \in \Pi(\mu_0,\mu_f)}
\epsilon\KL(\gamma \| r_\epsilon)$, and $r_\epsilon$ is in~\eqref{eq: reference_endpoint_density}. Then since $\mathrm P^{u_\epsilon^*}(\mathcal K_\delta^c)\leq
\sup_{x\in K_0,\;y\in K_f}
\mathrm R^{\epsilon,x,y}(\mathcal K_\delta^c)
+\gamma_\epsilon^*((K_0\times K_f)^c)$
uniformly for $0<\epsilon\leq1$, tightness decomposes to tightness of $\{\gamma_\epsilon^*\}_{0<\leq \epsilon\leq 1}$ and uniformly tightness of $\{\mathrm R^{\epsilon,x,y}\}_{0<\leq \epsilon\leq 1}$, for all $x$ and $y$ in compact sets compact sets $K_0$ and $K_f$ in $\R^d$, respectively. 

For  $\{\gamma_\epsilon^*\}_{0<\leq \epsilon\leq 1}$, since the given marginals $\mu_0,\mu_f\in\mathcal P(\R^d)$ are tight, we have that for compact sets $K_0,K_f$ in $\R^d$, $\gamma_\epsilon^*((K_0\times K_f)^c) \leq\mu_0(K_0^c)+\mu_f(K_f^c)$, for all $\epsilon$. Hence $\{\gamma_\epsilon^*\}_{0<\epsilon\leq1}$ is tight on $\R^d\times \R^d$. 

We proceed to show that for any compact sets $K_0, K_f \subset \mathbb{R}^d$ and any $\delta > 0$, there exists a compact set 
$\mathcal{K}_{\delta} \subset C([0,1];\mathbb{R}^d)$ such that $\sup_{0 < \epsilon \leq 1} \sup_{x \in K_0, y \in K_f} \mathbf{R}^{\epsilon,x,y}(\mathcal{K}_{\delta}^c) < \delta$. For this, we first estimate the marginal
\begin{equation}\label{eq: two-time_marginal_estimate}
 \mathbf{R}^{\epsilon,x,y}(d_{\mathrm{SR}}(X_s,X_t)>r)
 =\int_{\{(z,w):d(z,w)>r\}}\frac{p_{s,\epsilon}(x,z)p_{t-s,\epsilon}(z,w)
       p_{1-t,\epsilon}(w,y)}
      {p_{1,\epsilon}(x,y)}d z d w.
\end{equation}
The following three estimates will be useful. First, for $d(z,w)>r$, the upper bound in~\eqref{eq:global-gaussian} gives 
\[
\int_{d(z,w)>r}p_{h,\epsilon}(z,w) dw\leq A\left (e^{-ar^2/(2\epsilon h)}\right)I_{\frac{a}{2}}(z,\sqrt{\epsilon h}),
\]
 with $I_{\frac{a}{2}}(z,\sqrt{\epsilon h})=\frac1{V(z,\sqrt{\epsilon h})}
 \int_{\R^d}e^{-a d(z,w)^2/2\epsilon h}d w
\leq
1+D\sum_{j=0}^{\infty}
2^{(j+1)\kappa}e^{-(a/2)4^j}
<\infty$. 
Thus,
\begin{equation}\label{eq: middle_estimate_for_dw}
\int_{d(z,w)>r}p_{h,\epsilon}(z,w) dw\leq C_0 e^{-c_0r^2/(\epsilon h)}.
\end{equation}
Second, we also need integration in the initial variable $z$. Thus, by~\eqref{eq:global-gaussian} and~\eqref{eq:bridge_volume_comparison}, we get the recentering estimate
\begin{equation}\label{eq: recentreing_cost_quadratic in distance}
p_{h,\epsilon}(z,w)\leq
 \frac{AD}{V(w,\sqrt{\epsilon h})}
 \left(1+\frac{d(z,w)}{\sqrt{\epsilon h}}\right)^\kappa
 e^{-a d(z,w)^2/\epsilon h}\leq
 \frac{\widetilde A}{V(w,\sqrt{\epsilon h})}
 e^{-\widetilde a d(z,w)^2/\epsilon h},
\end{equation}
and hence
\begin{equation}\label{eq: middle_estimate_for_dz}
\int_{d(z,w)>r}p_{h,\epsilon}(z,w) dz\leq C_0 e^{-c_0r^2/(\epsilon h)}.
\end{equation}
Third, we bound the ratio $\frac{p_{1-t,\epsilon}(w,y)}{p_{1,\epsilon}(x,y)}$. To this end, since $M:=\max_{\substack{x\in K_0, y\in K_f}}d(x,y)<\infty$, for $p_{1,\epsilon}(x,y)$, the lower bound in~\eqref{eq:global-gaussian}, yields $\frac{1}{p_{1,\epsilon}(x,y)}\leq V(x,\sqrt{\epsilon})\frac{e^\frac{AM^2}{\epsilon}}{a}$. For $p_{1-t,\epsilon}(w,y)$, using the upper bound in~\eqref{eq:global-gaussian} and  recentering in~\eqref{eq: recentreing_cost_quadratic in distance}, we get that  $p_{1-t,\epsilon}(w,y)\leq \frac{\widetilde A}{V(y,\sqrt{(1-t)\epsilon})}$. Hence $\frac{p_{1-t,\epsilon}(w,y)}{p_{1,\epsilon}(x,y)}\leq \left(\frac{V(x,\sqrt{\epsilon})}{V(y,\sqrt{(1-t)\epsilon})}\right)\frac{\widetilde ADe^\frac{AM^2}{\epsilon}}{a}$. For $0\leq s<t\leq\frac{1}{2}$, we get that $1-t\geq\frac{1}{2}$ and hence $\sqrt{\epsilon}\leq 2\sqrt{(1-t)\epsilon}$. Thus, by the volume doubling property, since $V(y,\sqrt{\epsilon})\leq V(y,2\sqrt{(1-t)\epsilon})\leq DV(y,\sqrt{(1-t)\epsilon})$, we get that $\frac{p_{1-t,\epsilon}(w,y)}{p_{1,\epsilon}(x,y)}\leq \left(\frac{V(x,\sqrt{\epsilon})}{V(y,\sqrt{\epsilon})}\right)\frac{\widetilde ADe^\frac{AM^2}{\epsilon}}{a}$. Applying~\eqref{eq:bridge_volume_comparison} to get $\frac{p_{1-t,\epsilon}(w,y)}{p_{1,\epsilon}(x,y)}\leq\left(1+\frac M{\sqrt\epsilon}\right)^\kappa\frac{\widetilde ADe^\frac{AM^2}{\epsilon}}{a}$ and using $\left(1+\frac M{\sqrt\epsilon}\right)^\kappa
 \leq(1+M)^\kappa\epsilon^{-\kappa/2}
 \leq(1+M)^\kappa e^{\kappa/(2\epsilon)}$, we get that
\begin{equation}\label{eq: end_ratio_estimate_for_dz}
\frac{p_{1-t,\epsilon}(w,y)}{p_{1,\epsilon}(x,y)}\leq C_2e^{B_\kappa/\epsilon},\quad\text{for $0\leq s<t\leq\frac{1}{2}$},
\end{equation}
where $B_\kappa:=AM^2+\kappa/2$. Without repeating approach, we also get that 
\begin{equation}\label{eq: begin_ratio_estimate_for_dw}
\frac{p_{s,\epsilon}(x,z)}{p_{1,\epsilon}(x,y)}\leq C_2e^{B_\kappa/\epsilon},\quad\text{for $\frac{1}{2}\leq s<t\leq 1$}.
\end{equation}
Using \eqref{eq: middle_estimate_for_dw} and
\eqref{eq: end_ratio_estimate_for_dz} in
\eqref{eq: two-time_marginal_estimate} for
$0\le s<t\le \frac12$, and
\eqref{eq: middle_estimate_for_dz} and
\eqref{eq: begin_ratio_estimate_for_dw} for
$\frac12\le s<t\le1$, we obtain
\begin{equation}\label{eq: estimate_of_increments}
\mathbf R^{\epsilon,x,y}
\bigl(d_{\mathrm{SR}}(X_s,X_t)>r\bigr)
\le
C\exp\left(
\frac{B_\kappa}{\epsilon}
-\frac{cr^2}{\epsilon(t-s)}
\right),
\end{equation}
for some constants $c,C>0$, whenever $s,t$ belong to
the same half of $[0,1]$.
Here we also use the uniform bound
$\int_{\mathbb R^d}p_{h,\epsilon}(z,y)\,dz
\le C_{\mathrm{col}}$,
which follows from the recentered Gaussian estimate
and the annulus calculation. At $s=0$ and $t=1$, replace the first and third transition functions by Diracs $\delta_x$ and $\delta_y$, respectively.

We consider $0,\;2^{-n},\;2\cdot 2^{-n},\;\dots,\;1$, where $n=1,2,\dots$. Then, for each $n\geq 1$, any adjacent $[k2^{-n},(k+1)2^{-n}]$, where $k=0,\ldots,2^n-1$, lies in one half of $[0,1]$. Consider $ E_{n,k}(L)=
 \left\{\omega\in \Omega:
 d_{\mathrm{SR}}(\omega_{k2^{-n}},\omega_{(k+1)2^{-n}})
 >L2^{-n/4}\right\}$. Then, from~\eqref{eq: estimate_of_increments}, we get that $\mathrm R^{\epsilon,x,y}(E_{n,k}(L))
\leq
C\exp\left(
\frac{B_\kappa-cL^2 2^{n/2}}{\epsilon}
\right)$, for all $k=0,\ldots,2^n-1$ and $n\geq 1$. Choose $L_0>0$ so that $cL_0^2\sqrt2\geq2B_\kappa$. Then for $L\geq L_0$, we get 
\begin{equation}\label{eq:bridge_dyadic_probability}
\sup_{\substack{0<\epsilon\leq1\\x\in K_0,\;y\in K_f}}
\mathrm R^{\epsilon,x,y}
\left(
\bigcup_{n=1}^{\infty}
\bigcup_{k=0}^{2^n-1}E_{n,k}(L)
\right)
\leq
S(L):=C\sum_{n=1}^{\infty}
2^n e^{-(c/2)L^2 2^{n/2}}
.
\end{equation}
The series converges, and $S(L)\to0$ as $L\to\infty$. The standard dyadic chaining argument (see the proof of~\cite[Theorem~2]{bell2015kolmogorov}) shows that a continuous path outside $\bigcup_{n=1}^{\infty}
\bigcup_{k=0}^{2^n-1}E_{n,k}(L)$ satisfies $d(\omega_s,\omega_t)
\leq C_{\frac{1}{4}}L|t-s|^{1/4}$, for all $s,t\in[0,1]$. Hence $\mathbf R^{\epsilon,x,y}(\mathcal K_L)\ge1-S(L)$, uniformly in $0<\epsilon\le1$ and $(x,y)\in K_0\times K_f$, where $\mathcal K_L:=\{\omega\in\Omega:\;
\omega_0\in K_0,\ \omega_1\in K_f,d_{\mathrm{SR}}(\omega_s,\omega_t)
\leq C_{\frac{1}{4}}L|t-s|^{1/4}
\text{ for all }s,t\in[0,1]\}$. To show compactness, paths in $\mathcal K_L$ are equicontinuous. For equiboundedness, since $d_{\mathrm{SR}}(\omega_0,\omega_t)
\leq C_{\frac{1}{4}}L$, for all $t\in[0,1]$, from~\eqref{eq: Riemannian metric} and Assumption~\ref{ass:wellposed}, using Gr\"onwell's inequality, we get that $\|\omega_t\|\leq (1+|\omega_0|)e^{C(C_{\frac{1}{4}}L+1)}$, Furthermore, since $y\mapsto d_{\mathrm{SR}}(x,y)$, for all $x\in K_0$, is Euclidean continuous, by Heine–Borel the sub-Riemannian ball $\mathrm B_{d_{\mathrm{SR}}}(x,M_x+C_{\frac{1}{4}}L)$, where $M_x:=\max_{\omega_0\in K_0}d_{\mathrm{SR}}(x,\omega_0)$, is compact. For $\omega\in\mathcal K_L$, by triangle inequality $d_{\mathrm{SR}}(x,\omega_t)\leq d_{\mathrm{SR}}(x,\omega_0)+d_{\mathrm{SR}}(\omega_0,\omega_t)\leq M_x+C_{\frac{1}{4}}L$, we get that $\omega_t\in\mathrm B_{d_{\mathrm{SR}}}(x,M_x+C_{\frac{1}{4}}L)$, for all $t\in[0,1]$. By Arzela-Ascoli, $\mathcal K_L$ is compact. Thus, for any $\delta>0$, choose $L$ large such that $S(L)<\delta$ and $\mathcal K_\delta:=\mathcal K_L$. Then, from~\eqref{eq:bridge_dyadic_probability}, we get $\sup_{0 < \epsilon \leq 1} \sup_{x \in K_0, y \in K_f} \mathbf{R}^{\epsilon,x,y}(\mathcal{K}_{\delta}^c) < \delta$.
This completes the proof. 
\end{proof}
The last ingredient needed for the non-compact case is the following.
\begin{proposition}\label{prop: EOT convergence to OT}
Under Assumption~\ref{eq:density_regularities}, item~\ref{ass:two_tailed_Gaussian}, let $c_\epsilon=-\epsilon\log p_{1,\epsilon}$, where $p_{1,\epsilon}$ satisfies~\eqref{eq:global-gaussian} and $c=\frac{1}{2}d^2_{\mathrm{SR}}$. Then, we have~\eqref{eq: convergence to sub-Riemannian metric}. 
\end{proposition}
\begin{proof}
Let $F_\epsilon:=c_\epsilon-c$.  Then, using the triangle inequality $d_{\mathrm{SR}}(x,y)^2
\le2q(x)^2+2q(y)^2$~\eqref{eq: polynomial bound} and~\eqref{eq:global-gaussian} for $t=1$, we get that $|F_\epsilon(x,y)|
\le C\bigl(1+q(x)^2+q(y)^2\bigr)$, for all $0<\epsilon\le1$. Fix $R>0$ and consider $K_R:=\{x:q(x)\le R\}$. Then, for every coupling $\gamma\in\Pi(\mu_0,\mu_f)$, since $\left|\int_{\R^d\times\R^d} F_\epsilon(x,y)\,d\gamma\right|\leq
\sup_{K_R\times K_R}|F_\epsilon(x,y)|+CT(R)$, where  $T(R):=
\mu_0(K_R^c)+\mu_f(K_R^c)
+2\int_{K_R^c}q^2(x)\,d\mu_0
+2\int_{K_R^c}q^2(y)\,d\mu_f$, from~\cite[Theorem~1.2]{neel2026uniform}, we get that $\lim_{\epsilon \to 0} \sup_{\gamma \in \Pi(\mu_0,\mu_f)}\left|\int_{\R^d\times\R^d} F_\epsilon(x,y)\,d\gamma\right|\leq CT(R)$. Under Assumption~\ref{eq:density_regularities}, item~\ref{ass:two_tailed_Gaussian}, since $\int_{\R^d} q(x)^2\,d\mu_0(x)<\infty$ and $\int_{\R^d} q(y)^2\,d\mu_f(y)<\infty$ and $\mu_0,\mu_f$ are tight, taking $R\rightarrow\infty$, we get that $T(R)\rightarrow 0$. 
\end{proof}
Therefore, the proof of Theorem~\ref{thm:Gamma} takes the same skeleton as above under Assumption~\ref{eq:density_regularities}, item~\ref{ass:two_tailed_Gaussian}  using~\eqref{eq: finite_energy_bound}, Proposition~\ref{prop: optimal law is tight implies their marginal is tight} and~\ref{prop: EOT convergence to OT}.

\section{Numerical Method and Examples}\label{sec: Numerical Method}
In this section, we develop numerical methods for solving the optimality conditions~\eqref{eq:Schr-fbwd}-\eqref{eq:Schr-bc}. 
Eliminating the initial factor in~\eqref{eq:Sch-multplication_system}, we arrive at the operator %\map{T_{\epsilon}}{\mathcal{S}_{\mathrm{com}}}{\mathcal{S}_{\mathrm{com}}}$ as
\begin{equation}\label{eq: contraction operator}
(T_{\epsilon}\varphi_{f,\epsilon})(y)
:=
\frac{\rho_f(y)}{\displaystyle\int_{\mathbb R^d} p_{1,\epsilon}(x,y)
\left[ \frac{\rho_0(x)}{(Q_{\epsilon,1}\varphi_{f,\epsilon})(x)} \right] dx },
\end{equation}
where $Q_{\epsilon,t}$ is defined in~\eqref{eq: action_semigroup}. We now establish the convergence of the Sinkhorn iteration for the above under both regimes of Assumption~\ref{eq:density_regularities}. 
\begin{proposition}
 Under Assumption~\ref{eq:density_regularities}, item~\ref{ass: compact support}, the iteration 
$\varphi_{f,\epsilon}^{k+1} = T_{\epsilon}(\varphi_{f,\epsilon}^{k})$ in~\eqref{eq: contraction operator} with $\varphi_{f,\epsilon}^{0}=\rho_f$, converges exponentially to the unique
fixed point $\varphi_f$, where $T_{\epsilon}(\varphi_{f,\epsilon})= \varphi_{f,\epsilon}$.
\end{proposition}
\begin{proof}
Under Assumption~\ref{eq:density_regularities}, item~\ref{ass: compact support}, we consider $\mathcal{S}_{\mathrm{com}} := \big\{\rho_fe^v:  v \in L^{\infty}(\mu_f)  \big\}$ and the metric $d_{\infty}(\varphi,\tilde \varphi)
  =
  \log\!\left\| \frac{\varphi}{\tilde \varphi} \right\|_{L^\infty(\mu_f)}
  +
  \log\!\left\| \frac{\tilde \varphi}{\varphi} \right\|_{L^\infty(\mu_f)}$, for interior $\varphi,\tilde \varphi \in \mathcal{S}_{\mathrm{com}}$.   From Proposition~\ref{lem:kernel}, since
\[
0<a_\epsilon
:=
\min_{((x,y))\in\mathrm K_0\times\mathrm K_f}p_{1,\epsilon}(x,y)
\le
p_{1,\epsilon}(x,y)
\le
\max_{((x,y))\in\mathrm K_0\times \mathrm K_f}p_{1,\epsilon}(x,y)
=:b_\epsilon<\infty,
\]
from~\cite{Chen2016Hilbert}, for each $\epsilon>0$, the map $T_{\epsilon}$ in~\eqref{eq: contraction operator} is a contraction map. Hence, the iteration 
$\varphi_{f,\epsilon}^{k+1} = T_{\epsilon}(\varphi_{f,\epsilon}^{k})$ with $\varphi_{f,\epsilon}^{0}=\rho_f$, converges exponentially to the unique
fixed point $\varphi_f$, where $T_{\epsilon}(\varphi_{f,\epsilon})= \varphi_{f,\epsilon}$.
\end{proof}
 Under Assumption~\ref{eq:density_regularities},
item~\ref{ass:two_tailed_Gaussian}, the Hilbert metric $d_{\infty}$ is not useful and we demonstrate Sinkhorn convergence in total variation of the couplings in the sense of~\cite{nutz2023stability}. For this part, the relevant function space  is $
\mathcal E_f
:=
\left\{e^v: v\in L^1(\mu_f)
\right\}$ with  $\int_{\mathbb R^d}\log\varphi\,d\mu_f=0$. Here $T_\epsilon$ in~\eqref{eq: contraction operator} is not a contraction on $\mathcal E_f$. We now state the result.
\begin{proposition}\label{prop:sinkhorn_sg}
Let Assumption~\ref{eq:density_regularities},
item~\ref{ass:two_tailed_Gaussian} hold. For each fixed $\epsilon > 0$, consider %the Sinkhorn iterates $\gamma_\epsilon^k$ generated by the factors in~\eqref{eq: contraction operator}, that is 
\[
\gamma_\epsilon^{k}(dx,dy) = \varphi_{0,\epsilon}^{k}(x)\,\varphi_{f,\epsilon}^{k}(y)\,p_{1,\epsilon}(x,y)\,dx\,dy
\]
where
\begin{equation}\label{eq: Sinkhorn_iteration}
\varphi_{0,\epsilon}^k(x) = \frac{\rho_0(x)}{\int_{\R^d} p_{1,\epsilon}(x,y)\,\varphi_{f,\epsilon}^k(y)\,dy}
\quad\text{and}\quad
\varphi_{f,\epsilon}^{k+1}(y) = \frac{\rho_f(y)}{\int_{\R^d} p_{1,\epsilon}(x,y)\,\varphi_{0,\epsilon}^k(x)\,dx}.
\end{equation}
Then,  $\bigl\|\gamma_\epsilon^{k} - \gamma_\epsilon^{*}\bigr\|_{\mathrm{TV}} \longrightarrow 0$ as $k \to \infty$ (converge in total variation), where \begin{equation}\label{eq: optimal_coupling_for_unshifted_cost}
\gamma_\epsilon^*(dx_0,dx_f)=\phi_{0,\epsilon}(x_0)\phi_{f,\epsilon}(x_f)p_{1,\epsilon}(x_0, x_f)\rho_0(x_0) dx_0dx_f,    
\end{equation}
and $\phi_{0,\epsilon},\phi_{f,\epsilon}$ satisfies~\eqref{eq:Sch-multplication_system} is the optimal coupling for the cost $c_\epsilon := -\epsilon \log p_{1,\epsilon}$ in~\eqref{eq: KL_for_entropy_num}.
\end{proposition}
%The From~\eqref{eq:KL-disintegration} and~\eqref{eq: finite_energy_bound} we have that 
 %Thus, $\phi_{f,\epsilon}$ satisfies $T_{\epsilon}(\phi_{f,\epsilon})= \phi_{f,\epsilon}$. 
%We proceed to show that the iteration  $\varphi_{f,\epsilon}^{k+1} = T_{\epsilon}(\varphi_{f,\epsilon}^{k})$ in~\eqref{eq: contraction operator} converges to $\phi_{f,\epsilon}$. For this, we verify the exponential moment condition in~\cite[Corollary~3.2]{nutz2023stability}. To this end, 
A direct computation shows that $\varphi_{f,\epsilon}^k$ satisfy $\varphi_{f,\epsilon}^{k+1} = T_{\epsilon}(\varphi_{f,\epsilon}^{k})$ in~\eqref{eq: contraction operator}.  
\begin{proof}
Consider the shifted cost 
\begin{equation}\label{eq: shifted_cost}
\tilde{c}_\epsilon(x_0,x_f):=c_\epsilon(x_0,x_f)+a_\epsilon(1+q(x_0)^2+q(x_f)^2)\quad\text{and}\quad k_\epsilon(x_0,x_f):=e^{-\frac{\tilde{c}_\epsilon(x_0,x_f)}{\epsilon}},
\end{equation}
for $a_\epsilon>0$. Then, from~\eqref{eq:global-gaussian}, since $|c_\epsilon(x_0,x_f)|\leq a_\epsilon(1+q(x_0)^2+q(x_f)^2)$, we have that 
\begin{equation}\label{eq: positive_shifted_cost}
0<\tilde{c}_\epsilon(x_0,x_f)\leq 2a_\epsilon(1+q(x_0)^2+q(x_f)^2).
\end{equation}
Let $\eta$ be such that $2a_\epsilon \eta<\min\{a_0,a_f\}$, where $a_i$, with $i\in\{0,f\}$ is in Assumption~\ref{eq:density_regularities},
item~\ref{ass:two_tailed_Gaussian}. Then, from~\eqref{eq: positive_shifted_cost}, we get that 
\[
\int_{\mathbb{R}^d\times\mathbb{R}^d}
e^{\eta\,\tilde{c}_\epsilon(x,y)}
(\mu_0\otimes\mu_f)(dx,dy)
\le
C_\epsilon
\int_{\mathbb{R}^d\times\mathbb{R}^d}
e^{-(a_i-2a_\epsilon\eta)\bigl(q(x)^2+q(y)^2\bigr)}
dxdy<\infty.
\]
Hence, $e^{\eta \tilde{c}_\epsilon}\in L^1(\mu_0\otimes\mu_f)$.
Thus, from~\cite[Corollary~3.2]{nutz2023stability}, for each $\epsilon>0$, we have that there exists a unique optimal coupling $\tilde\gamma_\epsilon^{*}$ for the cost shifted $\tilde c_\epsilon$ in~\eqref{eq: shifted_cost} used in place of $c_\epsilon$ in~\eqref{eq: KL_for_entropy_num}. In particular, since $a_\epsilon(1+q(x_0)^2+q(x_f)^2)$ in~\eqref{eq: KL_for_entropy_num} is separable, we get that $\tilde\gamma_\epsilon^{*}=\gamma_\epsilon^{*}$ in~\eqref{eq: optimal_coupling_for_unshifted_cost}. Furthermore, the sequence
\begin{equation}\label{eq: Nutz_iteration_coupling}
\tilde{\gamma}_\epsilon^k(dx,dy)
= \tilde{\varphi}_{0,\epsilon}^k(x)\,\tilde{\varphi}_{f,\epsilon}^k(y)\,k_\epsilon(x,y)\,\rho_0(x)\rho_f(y)\,dx\,dy,
\end{equation}
where $\tilde{\varphi}_{0,\epsilon}^k(x) = \frac{1}{\int_{\R^d} k_\epsilon(x,y)\,\tilde{\varphi}_{f,\epsilon}^k(y)\,\rho_f(y)\,dy}$ and $\tilde{\varphi}_{f,\epsilon}^{k+1}(y) = \frac{1}{\int_{\R^d} k_\epsilon(x,y)\,\tilde{\varphi}_{0,\epsilon}^k(x)\,\rho_0(x)\,dx}$
 converges to $\tilde\gamma_\epsilon^{*}$ in total variation, that is $\bigl\|\tilde\gamma_\epsilon^{k} - \gamma_\epsilon^{*}\bigr\|_{\mathrm{TV}} \longrightarrow 0$ as $n \to \infty$ as $k\rightarrow\infty$.

Let 
\begin{equation}\label{eq: proposed_scaling_function}
\varphi_{0,\epsilon}^{k}(x) := \rho_0(x)\,\tilde\varphi_{0,\epsilon}^{k}(x)\,e^{-a_\epsilon(1+q(x)^2)/\epsilon}
\qquad\text{and}\qquad
\varphi_f^{k}(y) := \rho_f(y)\,\tilde\varphi_f^{k}(y)\,e^{-a_\epsilon q(y)^2/\epsilon}    
\end{equation}
with 
\begin{equation}\label{eq: Nutz_iterate}
\tilde\varphi_0^{k}(x) = \frac{1}{\int_{\R^d} k_\epsilon(x,y)\,\tilde\varphi_f^{k}(y)\,\rho_f(y)dy}\quad \text{and}\quad \tilde\varphi_f^{k+1}(y) = \frac{1}{\int_{\R^d} k_\epsilon(x,y)\,\tilde\varphi_0^{k}(x)\,\rho_0(x)dx}.
\end{equation}
Then, we have that~\eqref{eq: Nutz_iteration_coupling} simplifies to
\[
\gamma_\epsilon^{k}(dx,dy) = \varphi_{0,\epsilon}^{k}(x)\,\varphi_{f,\epsilon}^{k}(y)p_{1,\epsilon}(x,y)\,dx\,dy=\tilde{\gamma}_\epsilon^k(dx,dy)
\]
and hence $\bigl\|\gamma_\epsilon^{k} - \gamma_\epsilon^{*}\bigr\|_{\mathrm{TV}} \longrightarrow 0$ as $k \to \infty$ holds. Furthermore, by eliminating $\tilde\varphi_0^{k}$ in~\eqref{eq: Nutz_iterate} from $\varphi_{0,\epsilon}^{k}(x)$ in~\eqref{eq: proposed_scaling_function} and eliminating $\tilde{\varphi}_f^{k+1}$ from $\varphi_f^{k+1}(y) := \rho_f(y)\,\tilde\varphi_f^{k+1}(y)\,e^{-a_\epsilon q(y)^2/\epsilon}$ in~\eqref{eq: proposed_scaling_function}
  we get that $\varphi_{0,\epsilon}^{k},\varphi_{f,\epsilon}^{k}$ satisfy~\eqref{eq: Sinkhorn_iteration}.  
\end{proof}
 %Because the map $\tilde\gamma_\epsilon \mapsto \gamma_\epsilon$ defined by the separable weight above is a bounded linear isomorphism on the space of couplings (with bounds depending only on $\epsilon$), we obtain $\bigl\|\gamma_\epsilon^{k} - \gamma_\epsilon^{*}\bigr\|_{\mathrm{TV}} \longrightarrow 0$, for each fixed $\epsilon > 0$.
\begin{comment}
For regularity, note that the solution for~\eqref{eq:Schr-fwd}-\eqref{eq:Schr-bc} is characterized as
\[
\varphi_\epsilon(t,\cdot) = e^{(1-t)\mathcal{L}_\epsilon}\varphi_{1}, \qquad 
\hat{\varphi}_\epsilon(t,\cdot) = e^{t\mathcal{L}_\epsilon^{\star}}\hat{\varphi}_0,
\]
with $\hat{\varphi}_0 = \rho_0 / \varphi(0,\cdot)$ and $\varphi_{1}=\phi_{f,\epsilon}$ in~\eqref{eq:Sch-multplication_system},  where the forward and backward semigroups $P_t = e^{t\mathcal{L}_\epsilon^{\star}}$ and $
Q_t = e^{t\mathcal{L}_\epsilon}$ are defined through the actions
\begin{equation}\label{eq: action_semigroup}
(P_t \varphi_0)(y) = \int_{\R^d} p_{t,\epsilon}(x,y) \varphi_0(x)   dx\quad\text{and}\quad (Q_t \varphi_f)(x) = \int_{\R^d} p_{t,\epsilon}(x,z) \varphi_f(z)   dz    
\end{equation}
\end{comment}

Once $\varphi_f$ in~\eqref{eq: contraction operator} is obtained, the full space-time solution for~\eqref{eq:Schr-fbwd}-\eqref{eq:Schr-bc} is recovered by
\[
\varphi_\epsilon(t,x) = (Q_{\epsilon,1-t} \varphi_f)(x), \qquad\text{and}\qquad
{\widehat{\varphi}_\epsilon(t,x)} = (P_{t,\epsilon}\big( \rho_0 / (Q_{\epsilon,1}\varphi_f) \big))(x),
\]
where $P_{\epsilon,t}$ and $Q_{\epsilon,t}$ is defined in~\eqref{eq: action_semigroup}. We state here that, for $\epsilon>0$, the transition density $p_{t,\epsilon}(x,y)$
is the fundamental solution of the forward Kolmogorov equation in the $y$-variable:
\[
\partial_t p_{t,\epsilon}(x,y) - (\mathcal{L}_{\epsilon,y}^{\star} p_{t,\epsilon})(x,y)=0,\qquad t>0,
\qquad
p_{0,\epsilon}(x,y)=\delta_x(y),
\]
where $\mathcal{L}_{\epsilon,y}^{\star}$ denotes the adjoint $\mathcal{L}_\epsilon^{\star}$ acting on the $y$-variable.
%Equivalently, $p_{t,\epsilon}$ satisfies the backward Kolmogorov equation in the $x$-variable:
%\[
%\partial_t p_{t,\epsilon}(x,y) + (\mathcal L_{\epsilon,x} p_{t,\epsilon})(x,y)=0,\qquad t>0\qquad
%p_{0,\epsilon}(x,y)=\delta_y(x).
%\]
\begin{figure}[t]
    \centering
    
    \begin{subfigure}[b]{0.45\textwidth}
        \centering
        \includegraphics[width=0.4\textwidth]{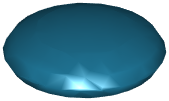}
        \caption{Initial density $\rho_0$}
        \label{fig:rho0}
    \end{subfigure}
    \hfill
    \begin{subfigure}[b]{0.45\textwidth}
        \centering
        \includegraphics[width=0.4\textwidth]{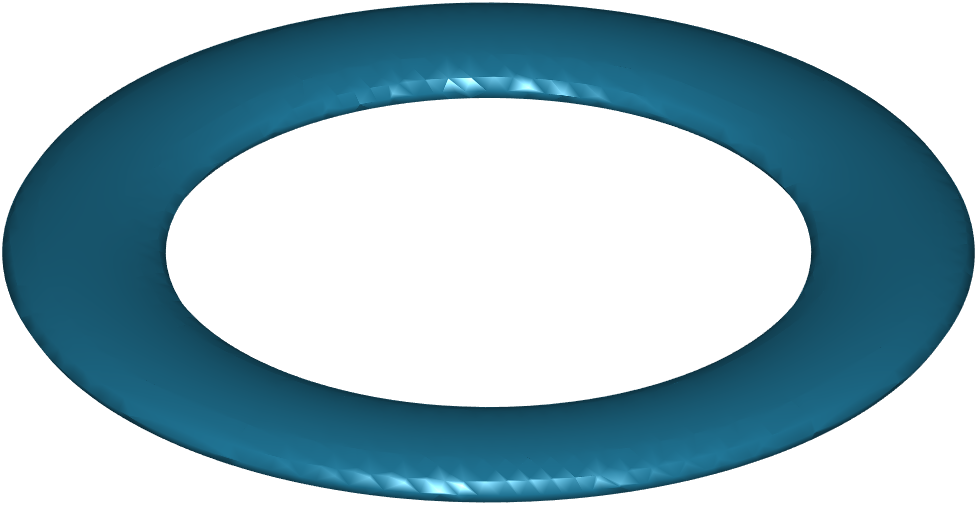} % [width=0.5\textwidth] [width=\linewidth]
        \caption{Final ring $\rho_f$}
        \label{fig:rhof}
    \end{subfigure}
    \caption{The above is the plot of the  isosurfaces of the initial and final densities defined in~\eqref{eq: initial_final density} with parameters $\alpha=.25, \beta=2$ and $R_0=2.5$.}
    \label{fig:boundary_densities}
\end{figure}
%\begin{figure}[t]
%    \centering
%    \includegraphics[width=\linewidth]{OT_horizontal_geodesics} %[width=0.45\textwidth]
%    \caption{TBW.}
%    \label{fig:OT_morphing}
%\end{figure}
%%%
\begin{figure}[t]
    \centering

    \begin{subfigure}[b]{\linewidth}
        \centering
        \includegraphics[width=0.8\linewidth]{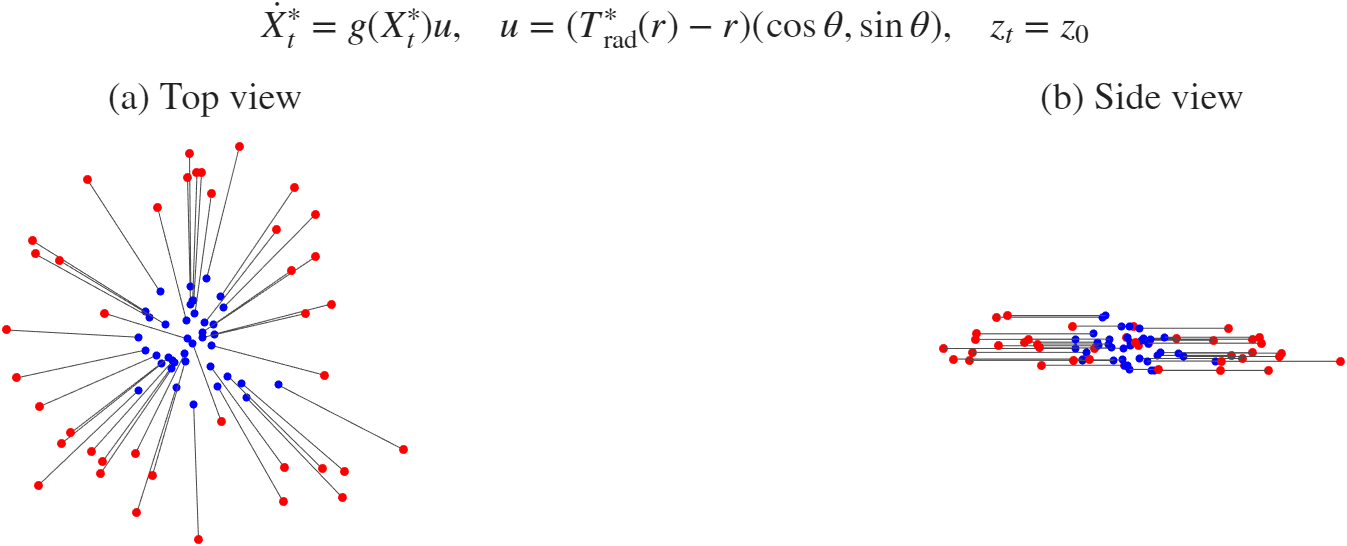}
        \caption{The above is the deterministic optimal transport trajectories on the
        Heisenberg group for
        \eqref{eq:intro_det_dyn}--\eqref{eq:intro_det_obj}
        with~\eqref{eq: diff_coefficient} and the marginals in
        Figure~\ref{fig:boundary_densities}.
        Blue and red markers indicate initial and terminal sample
        points, respectively. Each particle follows a constant-speed
        minimizing horizontal geodesic with radius $r_t$
        given by~\eqref{eq: flow_map}, while its angular coordinate
        and height remain unchanged.}
        \label{fig:OT_determ_traj}
    \end{subfigure}

    \medskip

    \begin{subfigure}[b]{\linewidth}
        \centering
        \includegraphics[width=\linewidth]{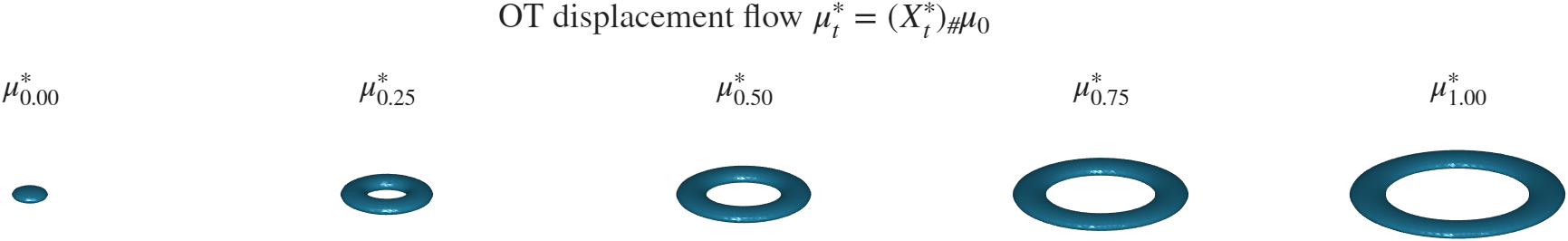}
        \caption{The above are snapshots of the optimal transport flow
        in~\eqref{eq: optimal_flow}, induced by the trajectories
        in Figure~(\subref{fig:OT_determ_traj}), between the marginals
        in Figure~\ref{fig:boundary_densities}.}
        \label{fig:OT_morphing}
    \end{subfigure}

    \caption{Figures~(\subref{fig:OT_determ_traj}) and~(\subref{fig:OT_morphing}) demonstrate the solution of the deterministic optimal transport on  $\mathbb{H}_{0.25}$ in Example~\ref{ex: unified_example} to be approximated by the Schrodinger bridge.}
    \label{fig:OT_combined}
\end{figure}
\begin{algorithm}[t]
\footnotesize
\caption{Static Sinkhorn iteration for the degenerate Schrödinger bridge}
\label{alg:sinkhorn}
\begin{enumerate}
\item \textbf{Input:} Densities $\rho_0,\rho_f$, tolerance $\mathrm{tol}$, maximum iterations $M_{\max}$, noise levels $\epsilon_0>\cdots>\epsilon_K>0$, kernels $p_{1,\epsilon_k}$.

\item \textbf{For each noise level $\epsilon_k$:}
\begin{enumerate}
\item \textbf{Initialize:} If $k=0$, set $\varphi_f^{\epsilon_0,0} = \rho_f$; if $k>0$, warm-start with $\varphi_f^{\epsilon_k,0} = \varphi_f^{\epsilon_{k-1},\mathrm{num}}$. Set the factor to zero outside $\{\rho_f>0\}$.

\item \textbf{Sinkhorn Iteration:} For $m=0,\ldots,M_{\max}$:
\begin{enumerate}
\item Compute forward factor: 
\[
\widehat{\varphi}_0^{\epsilon_k,m}(x)
= \frac{\rho_0(x)}{\int p_{1,\epsilon_k}(x,y)\,\varphi_f^{\epsilon_k,m}(y)\,dy}.
\]
\item Compute current terminal marginal and its TV residual:
\[
\rho_{\epsilon_k,m}(1,y)
= \varphi_f^{\epsilon_k,m}(y)\int p_{1,\epsilon_k}(x,y)\,\widehat{\varphi}_0^{\epsilon_k,m}(x)\,dx%\text{ and }\mathrm{res}_m = \frac12\int |\rho_{\epsilon_k,m}(1,y) - \rho_f(y)|\,dy. 
\]
\[
\mathrm{res}_m = \frac12\int |\rho_{\epsilon_k,m}(1,y) - \rho_f(y)|\,dy.
\]
\item If $\mathrm{res}_m<\mathrm{tol}$, store $\varphi_f^{\epsilon_k,\mathrm{num}} = \varphi_f^{\epsilon_k,m}$ and \textbf{break}.
\item Otherwise, compute and \textbf{normalize} the update:
\[
\widetilde{\varphi}_f^{\epsilon_k,m+1}(y)
= \frac{\rho_f(y)}{\int p_{1,\epsilon_k}(x,y)\,\widehat{\varphi}_0^{\epsilon_k,m}(x)\,dx},\text{ and  } \varphi_f^{\epsilon_k,m+1}(y)
= \frac{\widetilde{\varphi}_f^{\epsilon_k,m+1}(y)}{\int \widetilde{\varphi}_f^{\epsilon_k,m+1}(z)\,dz}
\]
%\[
%\varphi_f^{\epsilon_k,m+1}(y)= \frac{\widetilde{\varphi}_f^{\epsilon_k,m+1}(y)}{\int \widetilde{\varphi}_f^{\epsilon_k,m+1}(z)\,dz}
%\]
%(set to zero where $\rho_f=0$).
\end{enumerate}
\end{enumerate}
\item \textbf{Reconstruct factors:} Solve the backward and forward PDEs on $(0,1)\times\mathbb{R}^d$:
\[
\partial_t\varphi + \mathcal{L}_{\epsilon_k}\varphi = 0, \quad \varphi(1,\cdot) = \varphi_f^{\epsilon_k,\mathrm{num}}\quad \text{and}\quad  \partial_t\widehat{\varphi} - \mathcal{L}_{\epsilon_k}^*\widehat{\varphi} = 0, \quad \widehat{\varphi}(0,x) = \frac{\rho_0(x)}{\varphi(0,x)}
\]
\item \textbf{Compute optimal quantities:} For $0<t<1$,
\begin{equation*}
\rho_{\epsilon_k}^{\mathrm{num}}(t,x) = \widehat{\varphi}(t,x)\varphi(t,x),\quad
u_{\epsilon_k}^{\mathrm{num}}(t,x) = \epsilon_k g(x)^\top \nabla\log\varphi(t,x).
\end{equation*}
%\end{enumerate}

%\item \textbf{Verification (Optional):} Simulate $dX_t = (b_\epsilon + g u^{\mathrm{num}})dt + \sqrt{\epsilon_k}g\,dW_t$ and compare the empirical terminal distribution with $\mu_f$.
\end{enumerate}
\end{algorithm}
\begin{figure}[t]
    \centering
    \vspace{-0.3cm}
    % Epsilon = 1
    \includegraphics[width=0.9\textwidth]{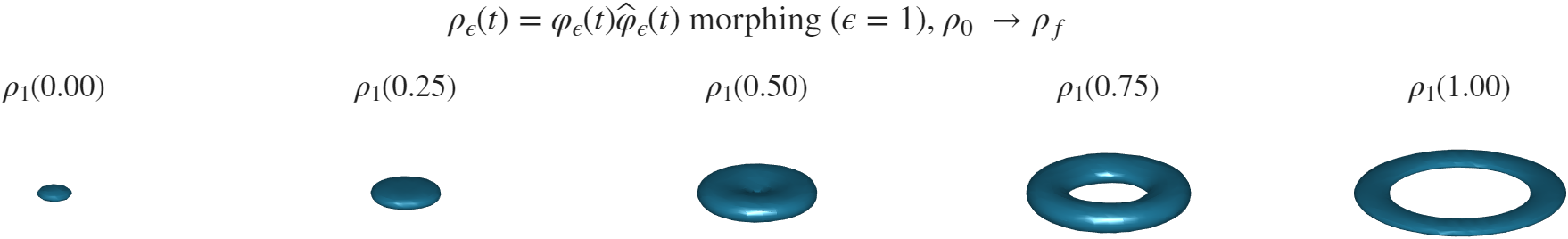}
    
   % \vspace{-0.3cm}
    
    % Epsilon = 0.01
    \includegraphics[width=0.9\textwidth]{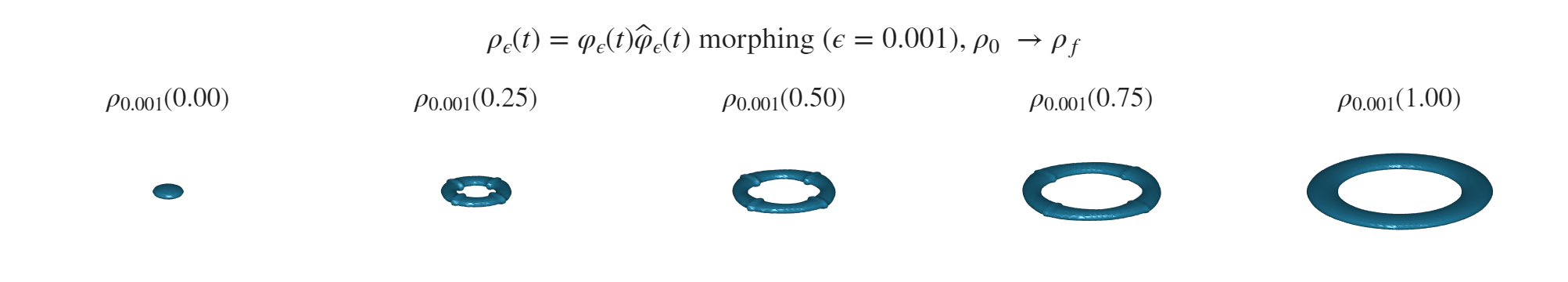}
    
    \caption{The plots above are the snapshots of the isosurfaces of the optimal density $\rho^*_\epsilon$ in~\eqref{eq: optimal_density}, where $\epsilon=0.001, 1$ (for demonstration purposes) obtained from the Sinkhorn iterations of the scaling functions $\varphi,\hat{\varphi}$. This aims to demonstrate a smooth and continuous morphing from $\rho_0$ to $\rho_f$ in~\eqref{eq: initial_final density}, for any $\epsilon>0$.} %Hence, as $\epsilon \to 0$, the Schr\"{o}dinger bridge approaches deterministic optimal transport.}
    \label{fig:all_eps_morphing}
\end{figure}
% In the preamble:
% \usepackage{graphicx}
% \usepackage{subcaption}

\begin{figure}[t]
    \centering

    % Top row
    \begin{subfigure}[b]{0.45\textwidth}
        \centering
        \includegraphics[width=\linewidth]{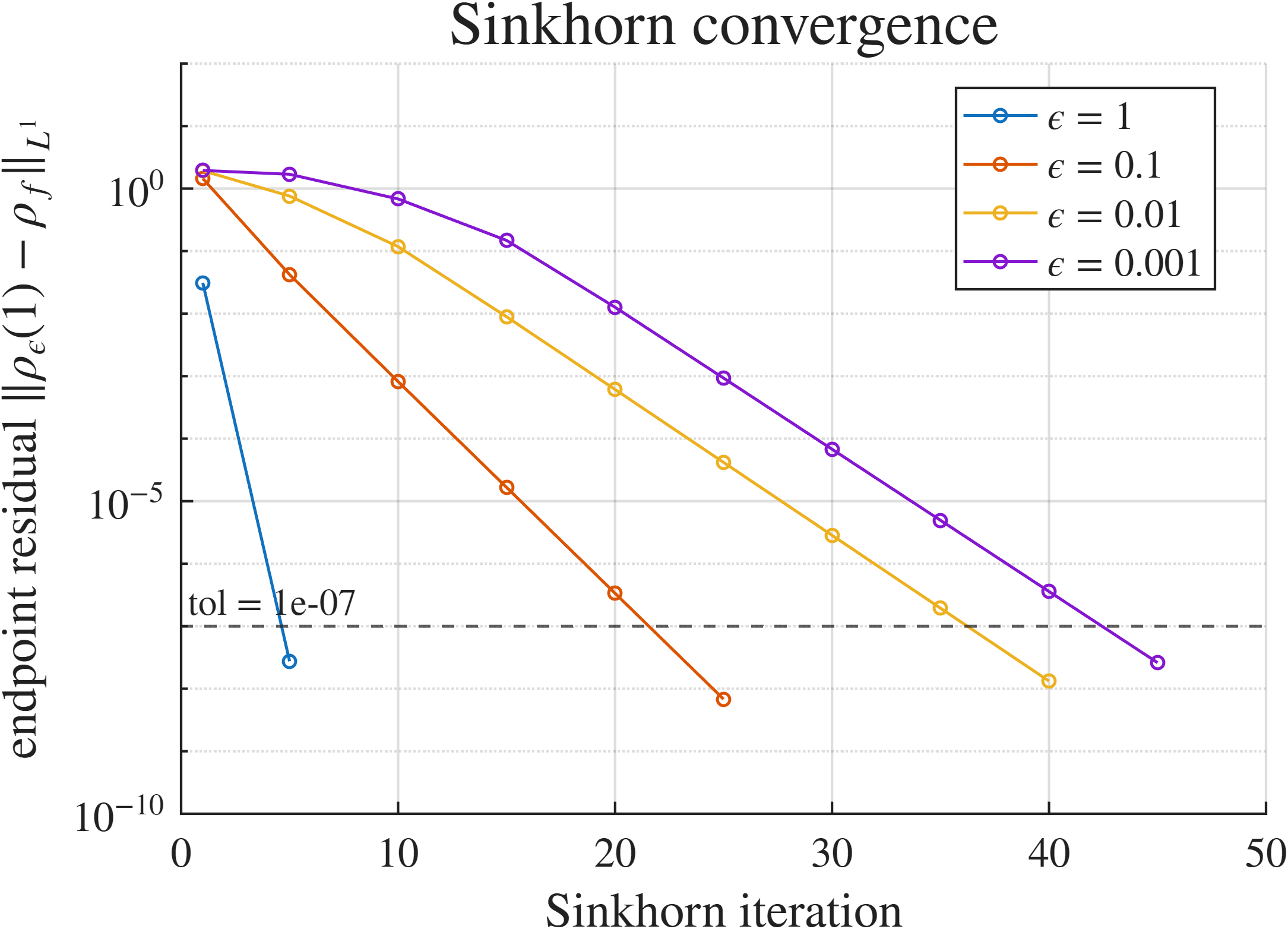}
        \caption{This shows the decay of the terminal marginal residual with Sinkhorn iterations for different values of $\epsilon$ in Figure~\ref{fig:all_eps_morphing}. Residuals are recorded at the first iteration and every five iterations thereafter; the dashed line indicates the stopping tolerance $10^{-7}$. The $\epsilon=0.001$ required more iterations to reach this tolerance than $\epsilon=1$.} %This plot measures how well those factors reproduce the prescribed final density Sinkhorn convergence with $\mathrm{tol}=10^{-7}$.
        \label{fig:Sinkhorn_convergence}
    \end{subfigure}
    \hfill
    \begin{subfigure}[b]{0.45\textwidth}
        \centering
        \includegraphics[width=\linewidth]{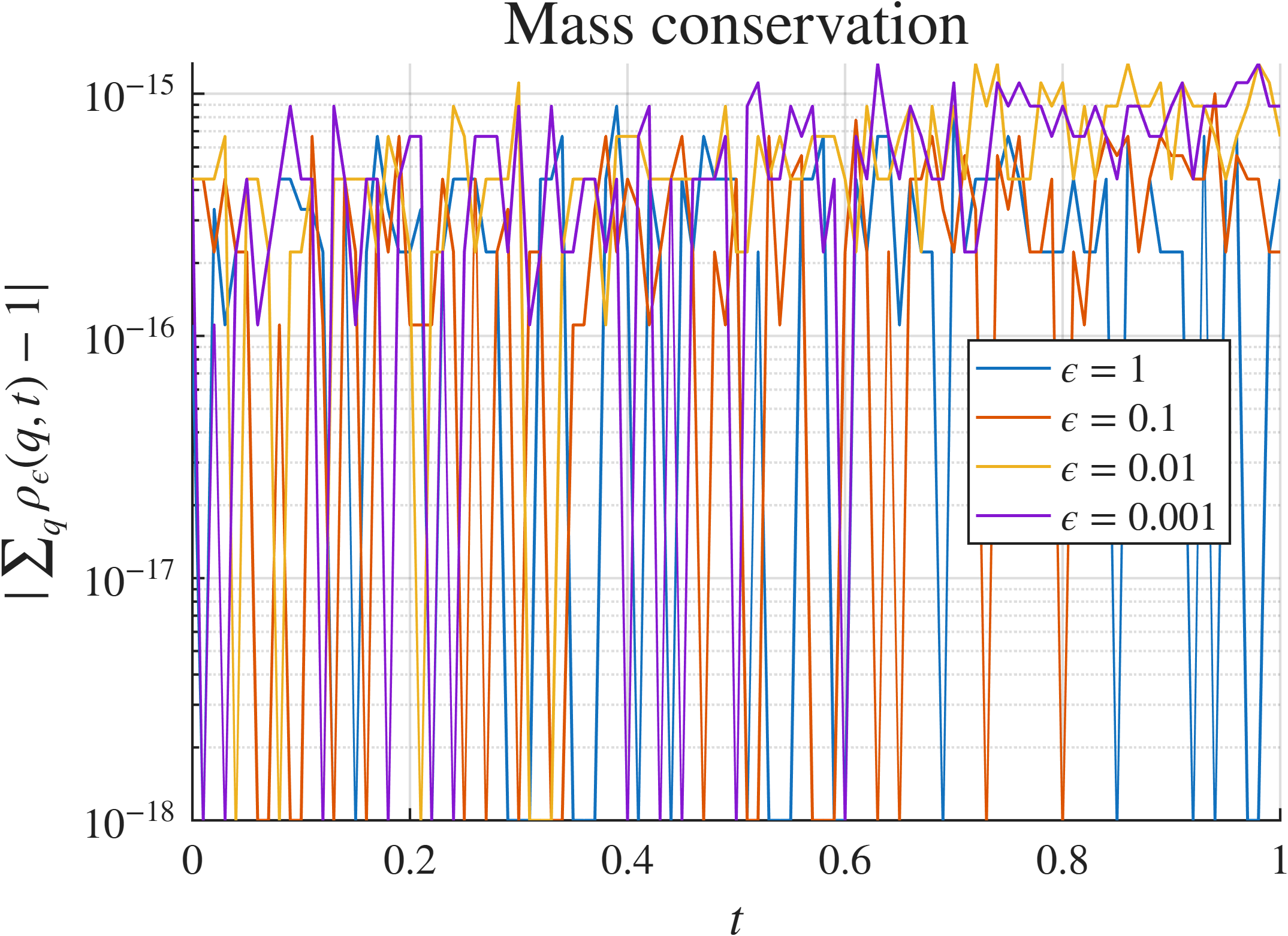}
        \caption{The above measures the deviation of the evolution of~\eqref{eq: optimal_density} from unit for the different values of $\epsilon={1,0.1,0.01,0.001}$. The plot demonstrates that the evolution of~\eqref{eq: optimal_density} is extremely close to $1$, confirming mass conservation.}
        \label{fig:Mass_conservation}
    \end{subfigure}

    \par\medskip

    % Bottom row
    \begin{subfigure}[b]{0.45\textwidth}
        \centering
        \includegraphics[width=\linewidth]{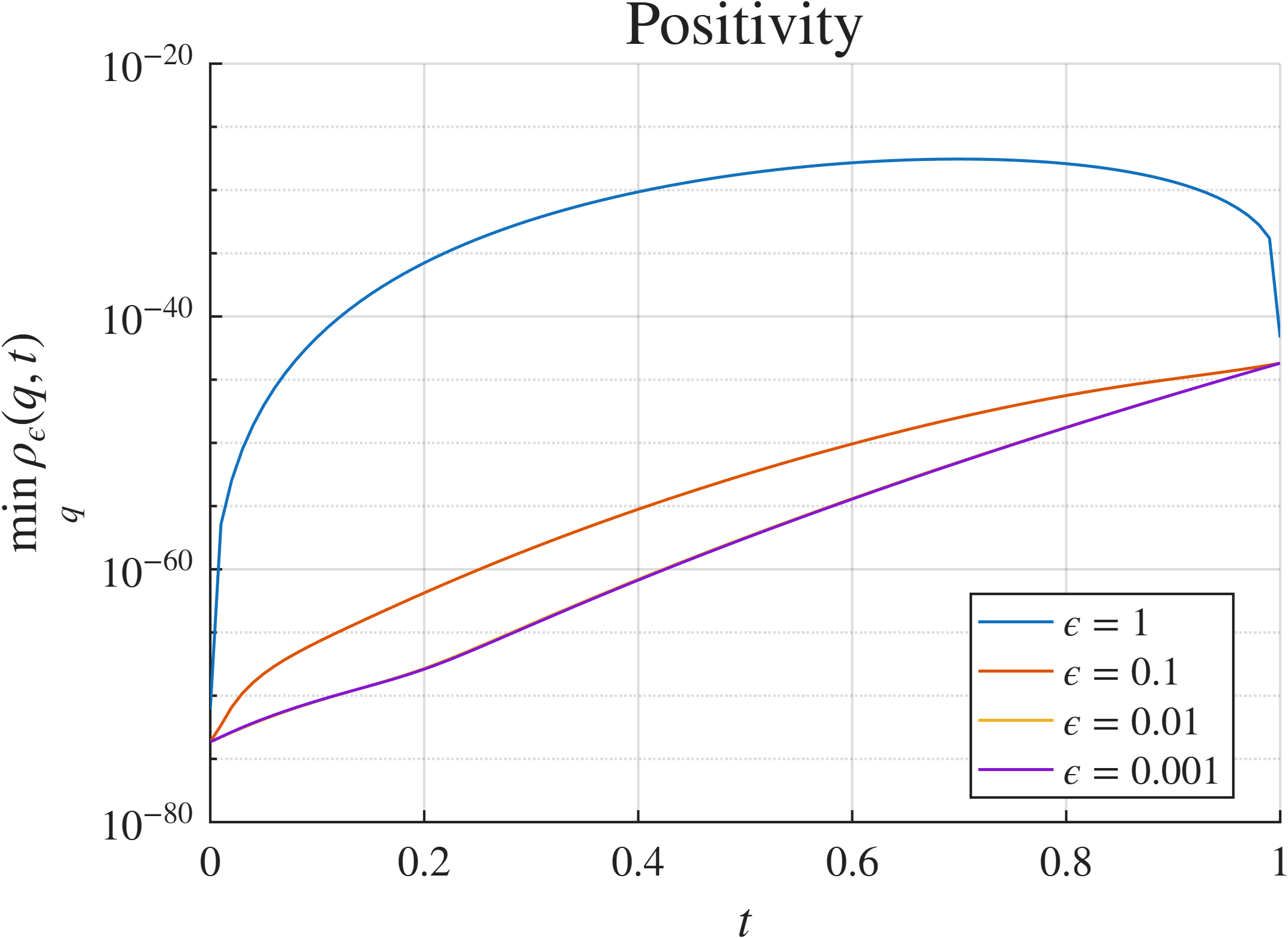}
        \caption{The above is the minimum computed density~\eqref{eq: optimal_density} over the spatial grid as a function of time. The minimum remains positive for all displayed values of $\epsilon={1,0.1,0.01,0.001}$.}
        \label{fig:Positivity}
    \end{subfigure}
    \hfill
    \begin{subfigure}[b]{0.45\textwidth}
        \centering
        \includegraphics[width=\linewidth]{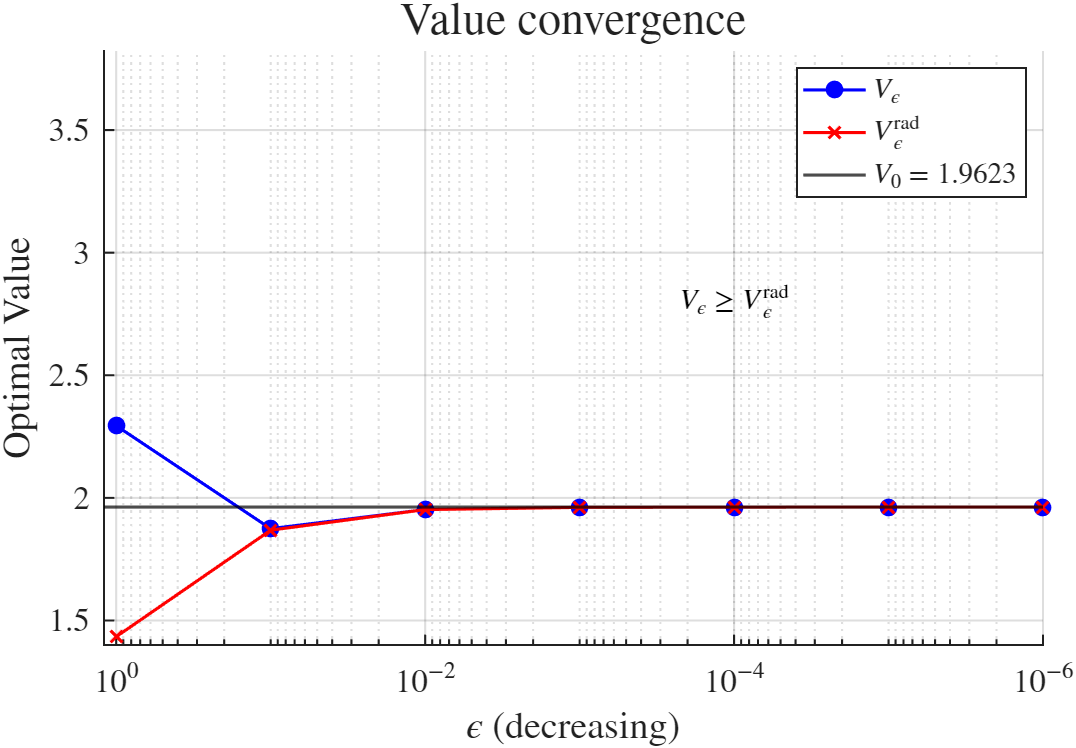}
        \caption{The above plot constitute a numerically computed optimal value $V_0=1.9623$ in~\eqref{eq: OTP_over_H} or~\eqref{eq: simplified_OTP} and the sequence of points, red cross for $\{V^{\mathrm{rad}}_{\epsilon}\}_{\epsilon>0}$ in~\eqref{eq:radial_sbp} and blue dots for  $\{V_{\epsilon}\}_{\epsilon>0}$ in~\eqref{eq: SB_optimal_value}. The above shows that $V_\epsilon\geq V_{\epsilon}^{\mathrm{rad}}$ and $\lim_{\epsilon\rightarrow 0}V_{\epsilon}=\lim_{\epsilon\rightarrow 0}V^{\mathrm{rad}}_{\epsilon}= V_0$.}
        \label{fig:Value_convergence}
    \end{subfigure}
\caption{The above is the numerical diagnostics for the Heisenberg Schr\"odinger bridge.}
    \label{fig:quantitative_measurements}
\end{figure}
%Sinkhorn convergence, massconservation, positivity, and convergence of the Schr\"odinger value $V_\epsilon$ to the deterministic optimal transport value $V_0$.
\begin{figure}[t]
  \centering
  \includegraphics[width=\textwidth]{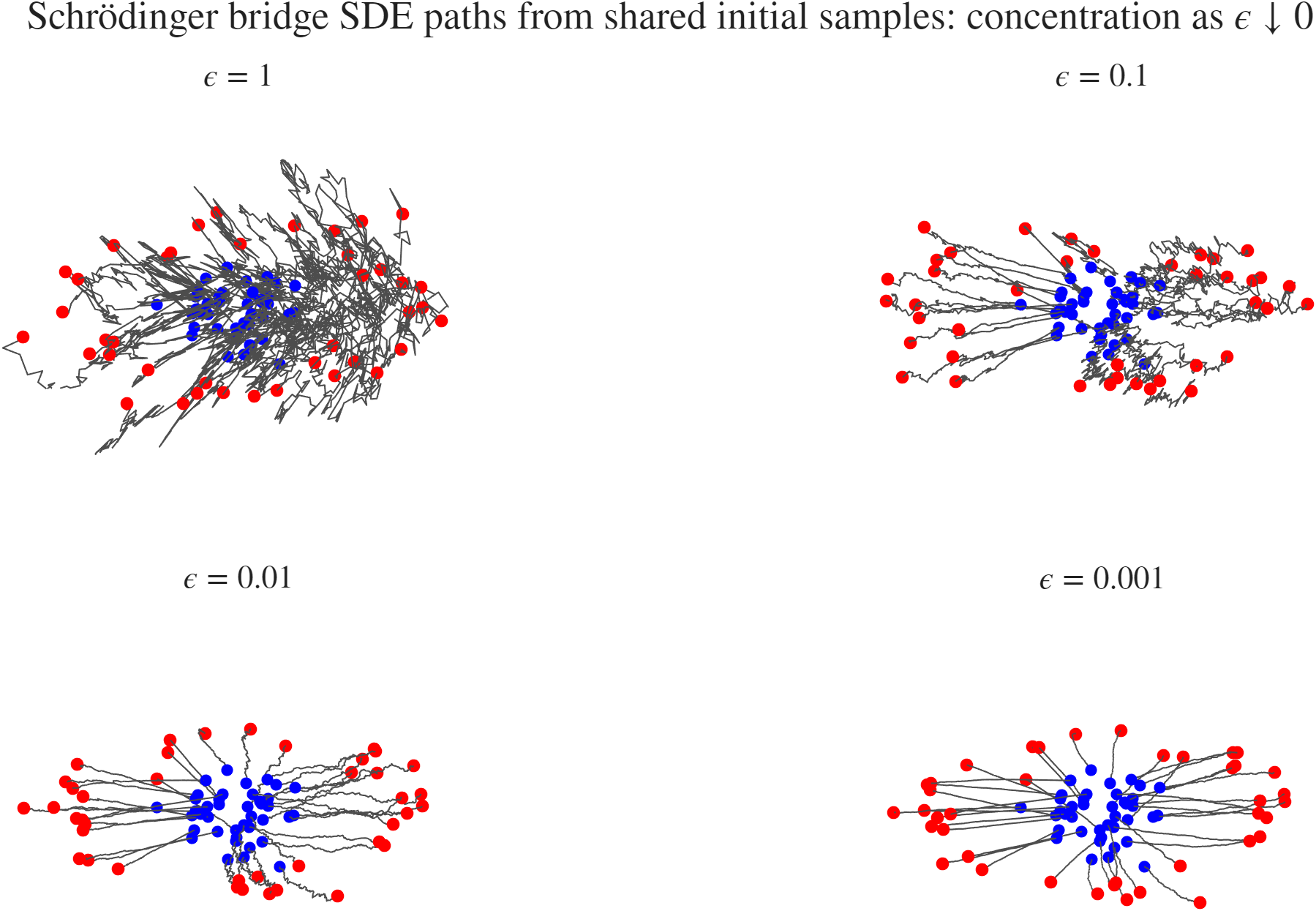}
  \caption{This plot is the corresponding Heisenberg bridge path associated to Example~\ref{ex: unified_example}. Figure~\ref{fig:all_eps_morphing} demonstrates the morphing from the initial to final densities for noise intensities $\epsilon\in\{1,0.1,0.01,0.001\}$. These are the optimal sample trajectories of the controlled diffusion whose distribution for $\epsilon=1$ and $\epsilon=0.001$ is in Figure~\ref{fig:all_eps_morphing}.
Black curves show $n=50$ (for better visualization) sample paths $t\mapsto\vec{q}_t=(x_t,y_t,z_t)$ in~\eqref{eq: ex_controlled_dynamics}; blue and red points indicate the initial and
terminal positions, respectively.
The bridge dynamics is obtained by changing the drift of the uncontrolled horizontally constrained Heisenberg-type reference diffusion in Fig.~\ref{fig:diffusion_comparison_three}-(a) so that its time-marginals interpolate
between prescribed endpoint laws.
As $\epsilon$ decreases, the ensemble concentrates
and trajectories exhibit constrained curvature indicating the effect of the presence of the horizontal constraint.}
\label{fig:heisenberg-sb-paths}
\end{figure}
%\margin{We have a subsection of Examples and they example environment. I would remove/change this subsection title}
%\subsection{Examples}
We state here that the minimum dimension for our setup to be non-trivial is $3$. The reason is that in this setting, one can genuinely have non-holonomic constraints with underactuated systems and the  Schr\"odinger bridge problem, interpreted as the optimal change in drift
becomes non-trivial. We illustrate our theoretical results with the following benchmark example.

%demonstrate this by considering three fundamental examples of optimal transport over non-holonomic constraints: the Brockett integrator, the Heisenberg system, and the Unicycle model which illustrate our theoretical results above.
%[Family of Sub-Riemannian Schr\"odinger Bridges]
%\margin{Using \em{} change the example environment from italic to standard, everywhere}
\begin{example}\label{ex: unified_example}
\em{We consider the distributional optimal control of Heisenberg control system, that is~\eqref{eq:intro_det_dyn}-\eqref{eq:intro_det_obj} in $\R^3$, where 
\begin{equation}\label{eq: diff_coefficient}
g(x,y,z) = \begin{pmatrix}1 & 0 \\ 0 & 1 \\ -2\alpha y & 2\alpha x\end{pmatrix}\quad\text{and}\quad u=\begin{pmatrix}u^1 \\ u^2\end{pmatrix}. 
\end{equation}
with $\alpha>0$, from $\mu_0\in\mathcal{P}(\R^3)$ to $\mu_f\in\mathcal{P}(\R^3)$ having densities
 \begin{equation}\label{eq: initial_final density}
\rho_i(x,y,z)
=
\frac{1}{Z_i}\,
\exp\bigl(-\beta\,U_i(x,y,z)\bigr),\quad\text{with}\quad Z_i
=
\int_{\mathbb{R}^{3}}
\exp\bigl(-\beta\,U_i(x,y,z)\bigr)
\,dx\,dy\,dz.    
 \end{equation}
Here $i\in\{0,f\} $ and $U_0(x,y,z)
=
x^{2}+y^{2}+h_\alpha(z)$ and $U_f(x,y,z)
=
\Bigl(\sqrt{x^{2}+y^{2}}-R_0\Bigr)^{2}+h_\alpha(z)$ with $h_\alpha(z)
=
\sqrt{1+\left(\frac{z}{\alpha}\right)^{2}}-1$ and $\beta,R_0>0$ (see Figure~\ref{fig:boundary_densities}). We equip the subbundle $\mathcal D$ in~\eqref{eq: distribution} with the metric $d_{\mathrm{SR}}$ for which
\begin{equation}\label{eq: unified_vector_fields}
g_1=\partial_x-2\alpha y\,\partial_z,
\qquad\text{and}\qquad
g_2=\partial_y+2\alpha x\,\partial_z
\end{equation} 
are orthonormal and obtain the corresponding optimal transport problem:
\begin{equation}\label{eq: OTP_over_H}
V_0:=\inf_{\substack{\dot{X}_t = g(X_t)\,u(t,X_t),\\ X_0\sim\mu_0,\; X_1\sim\mu_f}}
\mathbb{E}\!\left[\int_0^1 \frac{1}{2}\|u(t,X_t)\|^2\,dt\right]
=
\inf_{\gamma\in\Pi(\mu_0,\mu_f)}
\int_{\R^3\times\R^3} \frac{1}{2}\,d_{\mathrm{SR}}(\vec{q}_0,\vec{q}_1)^2d\gamma.
\end{equation}
The above is equivalent to the Bernamou-Brenier formulation in~\eqref{measure_theoretic_objective2} subject to~\eqref{eq:CE_measure_unref}.
%~\eqref{eq:OT-static}-\eqref{eq: transport_plans}  with transport cost: %the minimum control effort: %cost~\eqref{eq:intro_det_dyn}-\eqref{eq:intro_det_obj}, that is
%where
%\begin{equation}\label{eq: explicit_distant}
%\frac12d_{\mathrm{SR}}(\vec{q}_0,\vec{q}_1)^2=\inf_{\substack{\dot X_t=g(X_t)u(t,X_t)\\X_0=\vec{q}_0,\ X_1=\vec{q}_1}}\frac12\int_0^1\|u(t,X_t)\|^2\,dt,
%\end{equation}
% $\vec{q}_i=(x_i,y_i,z_i)\in\R^3$.
This is the optimal transport on Heisenberg group $\mathbb{H}_{\alpha}:=(\mathbb{R}^3,\cdot_{\alpha})$ with the law:
\begin{equation}\label{eq: group_structure}
(x,y,z) \cdot_{\alpha} (x',y',z') = (x + x', y + y', z + z' + 2\alpha(x y' - y x'))
\end{equation}
for all $(x,y,z),(x',y',z') \in \mathbb{R}^3$ and~\eqref{eq: unified_vector_fields} are exactly the left invariant vector field
adapted to~\eqref{eq: group_structure}. 
%are left-invariant and satisfy$[g_1,g_2]=4\alpha\partial_z$.Then $g_1,g_2$ are orthonormal on $\mathcal D=\operatorname{span}\{g_1,g_2\}$.
%Let $\mu_0,\mu_f\in\mathcal P(\R^3)$ have initial and final densities $\rho_0,\rho_f$ in~\eqref{eq: initial_final density} and consider 
From~\cite[Section~5.7.1]{agrachev2018curvature} or~\cite{beals2000hamilton}, we have the explicit form:
\begin{equation}\label{eq: explicit_distant}
d_{\mathrm{SR}}^2(q_1,q_2)=
\begin{cases}
(\Delta x^2 + \Delta y^2)\frac{(2\alpha\theta_c)^2}{\sin^2(2\alpha\theta_c)},
& (\Delta x^2 + \Delta y^2)>0,\ \Delta z\ne0,\\[6pt]
(\Delta x^2 + \Delta y^2),
& (\Delta x^2 + \Delta y^2)>0,\ \Delta z=0,\\[4pt]
\dfrac{\pi|\Delta z|}{\alpha},
& \Delta x, \Delta y=0.
\end{cases}    
\end{equation}
where $\theta_c \in (0,\pi/(2\alpha))$ uniquely solves $\frac{|\Delta z|}{\Delta x^2 + \Delta y^2} = \alpha\left(\frac{2\alpha\theta_c}{\sin^2(2\alpha\theta_c)} - \cot(2\alpha\theta_c)\right)$ with $\Delta x = x_2 - x_1, \quad \Delta y = y_2 - y_1, \quad \Delta z = z_2 - z_1 - 2\alpha(x_1 y_2 - y_1 x_2)$. 
The general framework has been studied in~\cite{ambrosio2004optimal} and since $\mu_0\ll d\vec{q}$ have finite second moment with respect to $d_{\mathrm{SR}}$, from~\cite[Theorem~5.1 and Remark~5.9]{ambrosio2004optimal}, we have that there exists a unique optimal coupling  $\gamma^*=(\mathrm {id},\psi)_{\#}\mu_0$, where $\map{\psi}{\R^3}{\R^3}$ is a measurable map.

\begin{comment}
We proceed to explicitly characterize the optimal flow in~\eqref{measure_theoretic_objective2}-\eqref{eq:CE_measure_unref} and demonstrate the approximation of the solution using our Schrodinger bridge framework. We will construct the minimizers of~\eqref{measure_theoretic_objective2}-\eqref{eq:CE_measure_unref} using the closed-form optimal transport map for our problem.

%Optimal transport problem on $\mathbb{H}_{\alpha}$, for general $\mu_0,\mu_f$ has been studied in~\cite{ambrosio2004optimal}. 
To derive the closed form optimal transport map, 
\end{comment}
To fully characterize the optimal control $u^*$ in~\eqref{eq: OTP_over_H} together with the corresponding solution $(\boldsymbol\mu^*,
\mathbf m^*)$ in~\eqref{measure_theoretic_objective2}-\eqref{eq:CE_measure_unref}, we will fully characterize $\psi$ associated to~\eqref{eq: initial_final density}. To this end, consider $\vec{q}=(r\cos\theta,r\sin\theta,z)$ with
$r\ge0$ and $\theta\in[0,2\pi)$. Then, from~\eqref{eq: initial_final density} we have that $\rho_i(r,\theta,z)r
=
\frac{r}{Z_i}\exp\bigl(-\beta h_\alpha(z)\bigr)\cdot
\exp\bigl(-\beta\, V_i(r)\bigr)$, where $V_0(r)=r^2$ and $V_f(r)=(r-R_0)^2$, are independent of $\theta$ and the $z$-marginal of $\rho_i$ coincide (both are $\exp\bigl(-\beta h_\alpha(z)\bigr))$. From~\eqref{eq: explicit_distant}, since 
\begin{equation}\label{eq: inequality_1}
d_{\mathrm{SR}}(q_0, q_1)^2 \ge |(x_1, y_1) - (x_0, y_0)|^2 \ge (r_1 - r_0)^2
\end{equation}
holds, for any $q_0, q_1 \in \mathbb{H}_\alpha$, for any $\gamma \in \Pi(\mu_0, \mu_f)$, we get that
\begin{equation}\label{eq: inequality_2}
\int \frac{1}{2} d_{\mathrm{SR}}^2 \, d\gamma \ge \int \frac{1}{2} (r_1 - r_0)^2 \, d\gamma \ge \min_{\kappa \in \Pi(\rho_0^{\mathrm{rad}} \, dr, \rho_f^{\mathrm{rad}} \, dr)} \int \frac{1}{2} (r_1 - r_0)^2 \, d\kappa,
\end{equation}
where $\kappa := \pi_\# \gamma$, with the projection $\pi(q_0, q_1) := (r_0, r_1)$ and
\begin{equation*}%\label{eq:radial-density}
\rho_i^{\mathrm{rad}}(r)
=
\frac{1}{Z_i^{\mathrm{rad}}}
\exp\bigl(-\beta V_i(r)\bigr)r,
\qquad\text{where}\qquad
Z_i^{\mathrm{rad}}=\int_0^\infty e^{-\beta V_i(r)}\,rd r
\end{equation*}
are the radial marginal densities for~\eqref{eq: initial_final density}. Thus, the candidate $\psi(r\cos\theta,r\sin\theta,z)=
\bigl(T_{\mathrm{rad}}(r)\cos\theta,T_{\mathrm{rad}}(r)\sin\theta,z\bigr)$ (or equivalently $\psi(r,\theta,z)=
(T_{\mathrm{rad}}(r),\theta,z)$) pushes $\mu_0$ to $\mu_f$ in~\eqref{eq: initial_final density}, since  $\theta$ and $z$ coincide and $T_{\mathrm{rad}}$ pushes $f_0^{\mathrm{rad}} \, dr$ to $f_f^{\mathrm{rad}} dr$. For this candidate, \(\Delta z = 0\) and the angle is unchanged, so both inequalities in~\eqref{eq: inequality_1}-\eqref{eq: inequality_2} are equalities and $d_{\mathrm{SR}}^2(\vec{q}_0,\psi(\vec{q}_0))=|r-T_{\mathrm{rad}}(r)|^2$. Thus,  the optimal transport problem~\eqref{eq: OTP_over_H} simplifies to %the cost equals the radial OT value $d_{\mathrm{SR}}^2(\vec{q}_0,\psi(\vec{q}_0))=|r-T_{\mathrm{rad}}(r)|^2$ and from~\eqref{eq: OTP_over_H}
\begin{equation}\label{eq: simplified_OTP}
V_0=\min_{
(T_{\mathrm{rad}})_\#(f_0^{\mathrm{rad}}d r)=f_f^{\mathrm{rad}}d r}
\int_0^\infty
\tfrac12\,\bigl|T_{\mathrm{rad}}(r)-r\bigr|^2\,
f_0^{\mathrm{rad}}(r)\,d r.
\end{equation}
%the optimal transport map $\psi$ must be rotation-invariant and leave the $z$ component fixed. Therefore, $\psi(r\cos\theta,r\sin\theta,z)=\bigl(T_{\mathrm{rad}}(r)\cos\theta,T_{\mathrm{rad}}(r)\sin\theta,z\bigr)$ is completely determined by $T_{\mathrm{rad}}(r)$. Furthermore, since $\theta$ and $z$ are fixed, are\\ and the sub-Riemannian distance~\eqref{eq: explicit_distant}  simplifies to the planar radial cost in case~$2$, that is $d_{\mathrm{SR}}^2(\vec{q}_0,\psi(\vec{q}_0))=|r-T_{\mathrm{rad}}(r)|^2$. 
%Hence the optimal transport problem~\eqref{eq: OTP_over_H} simplifies to
From~\cite[Theorem~2.4 and 2.9]{santambrogio2015optimal}, we have that the optimal radial map is monotone
\begin{equation}\label{eq: optimal radial}
T_{\mathrm{rad}}^*=(F_f^{\mathrm{rad}})^{-1}\circ F_0^{\mathrm{rad}},\quad \text{where}\quad F_i^{\mathrm{rad}}(r):=\int_0^r f_i^{\mathrm{rad}}(s)d s.
\end{equation}
Therefore,  $\psi(r\cos\theta,r\sin\theta,z)=
\bigl(T^*_{\mathrm{rad}}(r)\cos\theta,T^*_{\mathrm{rad}}(r)\sin\theta,z\bigr)$ is the optimal transport map for~\eqref{eq: OTP_over_H}.

We now use~\eqref{eq: optimal radial} to construct  $u^*$ in~\eqref{eq: OTP_over_H} and the corresponding $(\boldsymbol\mu^*,
\mathbf m^*)$. To this end, let
\begin{equation}\label{eq: flow_map}
X^*_t(\vec{q}_0):=\bigl(r_t\cos\theta,\;r_t\sin\theta,\;z\bigr),
\qquad\text{where}\qquad
r_t=(1-t)r_0+t\,T^*_{\mathrm{rad}}(r_0).
\end{equation}
Then, $X^*_0(\vec{q}_0)=(r_0\cos\theta,r_0\sin\theta,z)$ and
$X^*_1(\vec{q}_0)=(T^*_{\mathrm{rad}}(r_0)\cos\theta,T^*_{\mathrm{rad}}(r_0)\sin\theta,z)$.  Moreover, $\dot{X}^*_t(\vec{q}_0)=g(X^*_t(\vec{q}_0))\,a_t^*(\vec{q}_0)$, where $a_t^*(\vec{q}_0)=\bigl(T^*_{\mathrm{rad}}(r)-r\bigr)
\begin{pmatrix}\cos\theta\\ \sin\theta\end{pmatrix}$ and hence 
\begin{equation}\label{eq: optimal_lagrangian_control}
\frac12\int_0^1|a_t^*(q_0)|^2d t
=
\frac12\bigl|T^*_{\mathrm{rad}}(r_0)-r_0\bigr|^2.
\end{equation}
Thus,~\eqref{eq: flow_map} is optimal (see Figure~\ref{fig:OT_morphing}). Since $X^*_t$ in~\eqref{eq: flow_map} is invertible, from~\eqref{eq: simplified_OTP} and~\eqref{eq: optimal_lagrangian_control}, %the optimal feedback control in~\eqref{eq: OTP_over_H} is 
%$u^*(t,\vec{q}) := a_t^*\bigl(X_t^{*-1}(\vec{q})\bigr)$, where $\vec{q}=(r_{\vec{q}}\cos\theta_{\vec{q}},r_{\vec{q}}\sin\theta_{\vec{q}},z_{\vec{q}})$. Furthermore, by the optimality, we have that
\begin{equation}\label{eq: optimal_flow}
u^*= a^*\circ X^{*-1},\qquad\mu_t^*=(X^*_t)_\#\mu_0,
\qquad
\boldsymbol\mu^*=d t\,\mu_t^*,
\qquad\text{and}\qquad
\mathbf m^*=u^*\boldsymbol\mu^*,
\end{equation}
(see Figure~\ref{fig:OT_morphing}).

We proceed to approximate the above unique solution using our Schr\"odinger bridge framework. To this end,
one can check that~\eqref{eq: initial_final density} and~\eqref{eq: unified_vector_fields} satisfy satisfies Assumption~\ref{eq:density_regularities}, item~\ref{ass:two_tailed_Gaussian}  and Assumption~\ref{ass:wellposed}. Also, since $[g_1, g_2] = 4\alpha \partial_z$ is everywhere linearly independent of $g_1$ and $g_2$ we have that $\mathrm{span}\{g_1, g_2, [g_1, g_2]\} = \mathbb{R}^3$ satisfying Assumption~\ref{ass:Controllability condition}. %Therefore, rather than directly injecting a parameterized vertical direction which Riemannian approximates Problem~\eqref{eq:OT-static}-\eqref{eq: transport_plans} with~\eqref{eq: explicit_distant}-\eqref{eq: diff_coefficient}, 

%We introduce a small noise aligned with the control directions, leading to the stochastic process~\eqref{eq:intro_sto_dyn} where $W_t$ is a Brownian motion in $\mathbb{R}^2$. 
Since $\nabla_{g_1}g_1=\nabla_{g_2}g_2=\vec{0}\in\R^3$ and hence $b_\epsilon\equiv\vec{0}$, the stochastic problem simplifies to~\eqref{eq:intro_sto_obj} subject to %-~\eqref{eq:intro_sto_dyn}
%\begin{equation}\label{eq: ex_objective}
%\inf_{u \in \mathcal{U}} \E_{\mathrm P}\left[ \int_0^{1} \frac12 \|u(t,  X_t^u)\|_2^2  dt  \right]
%\end{equation}
%subject to 
\begin{align}\label{eq: ex_controlled_dynamics}
&d X_t^u = g(X_t^u) u(t,  X_t^u)  dt+ \sqrt{\epsilon} g(X_t^u)  dW_t,\\
& X_0^u\sim\mu_0\quad\text{and}\quad  X_{1}^u\sim\mu_f,\nonumber
\end{align}
where $W_t$ is a Brownian motion in $\mathbb{R}^2$ and  $\mathrm{R}^{\epsilon,\mu_0}$ in~\eqref{sbp_precise}-\eqref{eq: SB_constraint} is generated by
\begin{equation}\label{eq: Horizontal_Brownian_Motion}
d X_t = \sqrt{\epsilon} g(X_t)  dW_t,\qquad X_0\sim\mu_0
\end{equation}
(see Figure~\ref{fig:constrained}).
Since $G = gg^\top = \begin{pmatrix}
1 & 0 & -2\alpha y \\
0 & 1 & 2\alpha x \\
-2\alpha y & 2\alpha x & 4\alpha^2(x^2+y^2)
\end{pmatrix}$, from~\eqref{eq:Ito-generator}, the generator for~\eqref{eq: Horizontal_Brownian_Motion} is $\mathcal{L}_{\epsilon} f
= \frac{\epsilon}{2}\Bigl[
\partial_{xx} f
+ \partial_{yy} f
- 4\alpha y\,\partial_{xz}^{2} f
+ 4\alpha x\,\partial_{yz}^{2} f
+ 4\alpha^{2}(x^{2}+y^{2})\,\partial_{zz}^{2} f
\Bigr]$
and hence $\mathcal{L}_\epsilon=\mathcal{L}_\epsilon^{\star}$ with respect to Lebesgue measure. Furthermore, from Proposition~\ref{lem:kernel} and~\cite{beals2000hamilton}, we have that the the transition density %$\partial_t + \mathcal{L}_\epsilon$ and $\partial_t - \mathcal{L}_\epsilon^{\star}$ 
%are hypoelliptic and from~\cite{beals2000hamilton}, 
\begin{equation}
p_{t,\epsilon}(\vec{0}, \vec{q}) = \frac{1}{(2\pi)^2}
\int_{\mathbb R}
\frac{2\alpha\lambda}{\sinh(2\alpha\lambda\epsilon t)}
\exp\left(
-\alpha\lambda\coth(2\alpha\lambda\epsilon t)(x^2+y^2)
+i\lambda z
\right)\,d\lambda
\end{equation}
of~\eqref{eq: Horizontal_Brownian_Motion} from $\vec{0} = (0,0,0)$ at time $0$ to $(x,y,z)$ at time $t$ is smooth for $t>0$ and strictly positive. Using~\eqref{eq: group_structure}, for any  $\vec{q}_0=(x_0,y_0,z_0)$ and $\vec{q}=(x,y,z)$, since $p_{t,\epsilon}(\vec{q}_0, \vec{q})=p_{t,\epsilon}(\vec{0}, \vec{q}_0^{-1}\cdot\vec{q})$, 
where $\vec{q}_0^{-1} =(-x_0, -y_0, -z_0)$ we have that, 
\begin{equation}\label{eq: unified_heat_kernel}
p_{t,\epsilon}(\vec{q}_0, \vec{q}) =\frac{1}{(2\pi)^2} \int_{\mathbb{R}} \frac{2\alpha\lambda}{\sinh(2\alpha\lambda \epsilon t)} 
\exp\bigg(-\frac{\alpha\lambda R^2}{\tanh(2\alpha\lambda \epsilon t)}
+ i\lambda \Delta z\bigg) d\lambda,
\end{equation}
where $R^2:=(\Delta x^2 + \Delta y^2)$ and $\Delta x = x-x_0, ~\Delta y = y-y_0 ,~ \Delta z =  z-z_0 - 2\alpha(x_0 y - y_0 x)$. 
From~~\cite[Theorem~2.1]{dziubanski2019lie} we get that the above kernel satisfies Assumption~\ref{ass:global-lower-gaussian}.

Using~\eqref{eq: unified_heat_kernel} in~\eqref{eq:Sch-multplication_system}, let $\phi_{0,\epsilon},\phi_{f,\epsilon}$  be the Schr\"odinger  factors. Then, from Theorem~\ref{thm:properness}, the unique Schr\"odinger bridge for~\eqref{sbp_precise}-\eqref{eq: SB_constraint} with $\mathrm{R}^{\epsilon,\mu_0}$ generated by~\eqref{eq: Horizontal_Brownian_Motion} is $d\mathrm{P}^{u^*_\epsilon}=\phi_{0,\epsilon}(X_0)\phi_{f,\epsilon}(X_1)d\mathrm{R}^{\epsilon,\mu_0}$, where %the corresponding minimum energy control for~\eqref{eq:intro_sto_obj} subject to~\eqref{eq: ex_controlled_dynamics} is characterized as
\begin{equation}\label{eq: unified_optimal_control}
 u^*_{\epsilon}(t,x,y,z) = \epsilon \begin{pmatrix}
1 & 0 & -2\alpha y \\
0 & 1 & 2\alpha x
\end{pmatrix} \nabla\log\varphi_\epsilon(t,x,y,z),   
\end{equation}
with $\varphi_\epsilon(t,\vec{q}) = \int_{\R^3} \phi_{f,\epsilon}(\vec{q}_1) p_{1-t,\epsilon}(\vec{q}, \vec{q}_1) d\vec{q}_1$, minimizes~\eqref{eq:intro_sto_obj} and steers~\eqref{eq: ex_controlled_dynamics} from $X_0\sim\mu_0$ to $ X_{1}\sim\mu_f$.
From Theorem~\ref{prop: Strong duality and attainment}, the  optimal density-momentum interpolant is 
\begin{equation}\label{eq: optimal_density}
\rho_\epsilon^*=\varphi_\epsilon\hat\varphi_\epsilon~\text{and}~m_\epsilon^*=\rho_\epsilon^*u^*_{\epsilon},~\text{ where}~
\hat{\varphi}_\epsilon(t,\vec q)=\int_{\R^3} \phi_{0,\epsilon}(\vec{q}_0) p_{t,\epsilon}(\vec{q}_0, \vec{q})\rho_0(\vec{q}_0) d\vec{q}_0,
\end{equation}
with $p_{t,\epsilon}$ in~\eqref{eq: unified_heat_kernel}. The resulting optimal value satisfies 
\begin{equation}\label{eq: SB_optimal_value}
\mathsf V_\epsilon
=\tfrac12\int_0^1\!\int_{\R^3}\|u_\epsilon^*\|^2\rho_\epsilon^*d qd t
=\epsilon\KL(\mathrm{P}^{u_\epsilon^*} \| \mathrm{R}^{\epsilon,\mu_0})=\epsilon\KL(\gamma_\epsilon^*\| r_\epsilon),    
\end{equation} %the optimal endpoint Schr\"odinger coupling in
where~\eqref{eq:KL-disintegration} simplifies to $\gamma_\epsilon^*(d\vec{q}_0,d\vec{q}_1)=\phi_{0,\epsilon}(\vec{q}_0)\phi_{f,\epsilon}(\vec{q}_1)p_{1,\epsilon}(\vec{q}_0, \vec{q}_1)\rho_0(\vec{q}_0) d\vec{q}_0d\vec{q}_1$  and endpoint reference is $r_\epsilon(d \vec{q}_0,d \vec{q})=\rho_0(\vec{q}_0)p_{1,\epsilon}(\vec{q}_0,\vec{q})d\vec{q}_0d\vec{q}$. 
Therefore, by Theorem~\ref{thm:Gamma}, we get that $\mathsf V_\epsilon\longrightarrow\mathsf V_0$ and the sequence
$(\boldsymbol\mu_{\epsilon}^*,\mathbf m_{\epsilon}^*)$, characterized as $d\boldsymbol{\mu}_{\epsilon}^* = \rho_{\epsilon}^* dt dx$, $d\mathbf{m}_{\epsilon}^* = m_{\epsilon}^* dt dx$, where $(\rho_{\epsilon}^*,m_{\epsilon}^*)$ is in~\eqref{eq:optimal-primal} induced by~\eqref{eq: unified_heat_kernel}, converges weakly to  $(\boldsymbol\mu^*,\mathbf m^*)$ in~\eqref{eq: optimal_flow} as $\epsilon\downarrow0$

For the one-dimensional simplification, from~\eqref{eq: Horizontal_Brownian_Motion} , since the radial process $R_t\in(0,\infty)$ satisfies
\[
dR_t = \frac{\epsilon}{2R_t} dt + \sqrt{\epsilon} dW_t,\qquad R_0\sim\rho_0^{\mathrm{rad}},
\]
where $W_t\in\R$ is a Brownian motion and admits the infinitesimal generator $\mathcal{L}_\epsilon^{\mathrm{rad}} = \frac{\epsilon}{2}\left(\partial_r^2 + \frac{1}{r}\partial_r\right)$,
from~\cite[Eq.~(2.$\mathrm i^\prime$)]{pitman1982decomposition}, the transition density with respect to the Lebesgue measure is the Bessel heat kernel
\[
p_{1,\epsilon}^{\mathrm{Bessel}}(r,s) = \frac{s}{\epsilon} \exp\left(-\frac{r^2+s^2}{2\epsilon}\right) I_0\left(\frac{rs}{\epsilon}\right)
\]
where $I_0$ is the modified Bessel function of the first kind. Hence for the  radial Schr\"odinger bridge problem
\begin{equation}\label{eq:radial_sbp}
V_{\epsilon}^{\mathrm{rad}} := \inf_{\gamma \in \Pi(\rho_0^{\mathrm{rad}}, \rho_f^{\mathrm{rad}})} \epsilon D_{\mathrm{KL}}\Big(\gamma \,\Big\|\, p_{1,\epsilon}^{\mathrm{Bessel}}(r,s) \rho_0^{\mathrm{rad}}(r) dr ds\Big),
\end{equation}
where $V_\epsilon\geq V_{\epsilon}^{\mathrm{rad}}$, we get that $\lim_{\epsilon\rightarrow0}V_{\epsilon}=\lim_{\epsilon\rightarrow0}V_{\epsilon}^{\mathrm{rad}}=V_0$ in~\eqref{eq: simplified_OTP} (or~\eqref{eq: OTP_over_H})  (seee Figure~\ref{fig:Value_convergence}).
}
\end{example}

\bibliographystyle{siamplain}
\bibliography{references}
\end{document}